\documentclass[11pt]{amsart}
\usepackage{amsmath}
\usepackage{amssymb}
\usepackage{amsthm}
\usepackage{mathtools}
\usepackage{mathabx}
\usepackage{graphicx}
\usepackage{caption}
\usepackage{subcaption}
\usepackage{tikz-cd}
\usepackage[all,pdf]{xy}
\usepackage{lscape}
\usepackage[vmargin=1in, hmargin=1.25in]{geometry}
\usepackage{longtable}
\usepackage[backref=page, bookmarks, bookmarksdepth=2, colorlinks=true, linkcolor=blue,
citecolor=blue, urlcolor=blue]{hyperref}
\renewcommand*{\backref}[1]{}

\newtheorem{theorem}{Theorem}[section]

\newtheorem{prop}[theorem]{Proposition}
\newtheorem{lemma}[theorem]{Lemma}
\newtheorem{defn}[theorem]{Definition}
\newtheorem{question}[theorem]{Question}

\theoremstyle{remark}
\newtheorem{remark}[theorem]{Remark}
\newtheorem{example}[theorem]{Example}
\newtheorem{convention}[theorem]{Convention}
\newtheorem{case}{Case}
\newtheorem{subcase}{Case}[case]

\newtheorem{casea}{Case}

\title{Normal closure of finite subgroups of $\mathrm{Aut}(F_n)$ and $\mathrm{Out}(F_n)$}
\author{Jiayi Shen}
\begin{document}
\begin{abstract}
For $n\geq 3$, let $G$ be a nontrivial finite subgroup of $\mathrm{Aut}(F_n)$ with $|G|$ not a power of $2$. We prove that the normal closure $N(G)$ is $\mathrm{SAut}(F_n)$ if $G\subset\mathrm{SAut}(F_n)$ and $N(G)$ is $\mathrm{Aut}(F_n)$ otherwise. When $|G|$ is a power of $2$, we have a partial theorem. Similarly, let $G'$ be a nontrivial finite subgroup of $\mathrm{Out}(F_n)$ with $|G'|$ not a power of $2$. Then the normal closure $N(G')$ is $\mathrm{SOut}(F_n)$ if $G'\subset\mathrm{SOut}(F_n)$ and $N(G')$ is $\mathrm{Out}(F_n)$ otherwise. When $|G'|$ is a power of $2$, we have a partial theorem as well.
\end{abstract}
\maketitle

\section{Introduction}
Let $F_n = \langle x_1, x_2, \ldots, x_n \rangle$ be a free group of rank $n$ generated by $x_1, x_2, \ldots, x_n$. Let $\mathrm{Aut}(F_n)$ denote the automorphism group of $F_n$ and let $\mathrm{Out}(F_n)$ denote the outer automorphism group of $F_n$. By definition, $\mathrm{Out}(F_n)$ is the quotient of  $\mathrm{Aut}(F_n)$ by the inner automorphism group $\text{Inn}(F_n)$ where $\text{Inn}(F_n)\cong F_n$. In this paper, we study the normal closures of finite subgroups of $\mathrm{Aut}(F_n)$ and $\mathrm{Out}(F_n)$.

Roughly speaking, our main theorems say that these normal closures are usually all of  $\mathrm{Aut}(F_n)$ or $\mathrm{Out}(F_n)$. To state them precisely, we need to introduce some notation first. The abelianization map $F_n \rightarrow \mathbb{Z}^n$ induces a homomorphism from $\mathrm{Aut}(F_n)$ to the general linear group $\mathrm{GL}_n(\mathbb{Z})$.  This homomorphism factors through $\mathrm{Out}(F_n)$ and gives a homomorphism $\mathrm{Out}(F_n)\rightarrow \mathrm{GL}_n(\mathbb{Z})$. The \textit{special automorphism group of $F_n$}, denoted $\mathrm{SAut}(F_n)$, is the subgroup of $\mathrm{Aut}(F_n)$ consisting of elements mapping to  $\mathrm{GL}_n(\mathbb{Z})$ with determinant $1$. Similarly, we have the \textit{special outer automorphism group of $F_n$}, denoted $\mathrm{SOut}(F_n)$.

\subsection{Main theorem}
For a subgroup $G$ of a group $\Gamma$, let $N(G)$ be the normal closure of $G$ in $\Gamma$. Our main theorems are as follows:
\begin{theorem}\label{thm1}
For $n\geq 3$, let $G$ be a nontrivial finite subgroup of $\mathrm{Aut}(F_n)$ with $|G|$ not a power of $2$. Then 
\[
N(G) = \left\{
  \begin{array}{ll}
    \mathrm{SAut}(F_n) & \text{if } G\subset\mathrm{SAut}(F_n), \\
    \mathrm{Aut}(F_n) & \text{otherwise}.
  \end{array}
\right.\]

\end{theorem}

\begin{theorem}\label{thm2}
For $n\geq 3$, let $G'$ be a nontrivial finite subgroup of $\mathrm{Out}(F_n)$ with $|G'|$ not a power of $2$. Then
\[
N(G') = \left\{
  \begin{array}{ll}
    \mathrm{SOut}(F_n) & \text{if } G'\subset\mathrm{SOut}(F_n), \\
    \mathrm{Out}(F_n) & \text{otherwise}.
  \end{array}
\right.\]
\end{theorem}
\begin{example}\label{example1}Define automorphisms $f\in\mathrm{Aut}(F_6)$  and $g\in\mathrm{Aut}(F_4)$ as follows:
\begin{align*}
f: &x_1 \mapsto x_2               &g: &x_1 \mapsto x_2 \\
   &x_2 \mapsto x_2^{-1}x_1^{-1}  &   &x_2 \mapsto x_3 \\
   &x_3 \mapsto x_4               &   &x_3 \mapsto x_4 x_1 x_4^{-1} \\
   &x_4 \mapsto x_5               &   &x_4 \mapsto x_4 \\
   &x_5 \mapsto x_6               &   &\\
   &x_6 \mapsto x_3               &   &
\end{align*}
Let $\bar{g}$ be the image of $g$ in $\mathrm{Out}(F_4)$. Consider the subgroups $G = \langle f \rangle\le\mathrm{Aut}(F_6)$ and $G'= \langle \bar{g} \rangle\le\mathrm{Out}(F_4)$. Notice that $|G| = 12$ and the image of $f$ in $\mathrm{GL}_6(\mathbb{Z})$ has determinant $-1$.  By  Theorem \ref{thm1}, $N(G) = \mathrm{Aut}(F_6)$.  The subgroup $G'$ has order $3$ and is contained in $\mathrm{SOut}(F_4)$. By  Theorem \ref{thm2},  $N(G') = \mathrm{SOut}(F_4)$. 
\end{example}
\begin{remark}
From the example above, the outer automorphism $\bar{g}$ is a torsion element in $\mathrm{Out}(F_4)$ while in $\mathrm{Aut}(F_4)$, the automorphism $g$ is not. One can check that no torsion elements of $\mathrm{Aut}(F_4)$ project to $\bar{g}$. 
\end{remark}

\subsection{Previous work}
Let $S_g$ denote a connected, closed, orientable surface of genus $g$. The mapping class group $\mathrm{Mod}(S_g)$ is the group of homotopy classes of orientation-preserving homeomorphisms of $S_g$. The group $\mathrm{Out}(F_n)$ is in many ways analogous to $\mathrm{Mod}(S_g)$. For an overview of these similarities, see \cite{BV}. Our theorems are inspired by analogous theorems of Lanier--Margalit for $\mathrm{Mod}(S_g)$ which we now discuss.

If an element of $\mathrm{Mod}(S_g)$ has normal closure equal to the whole group, we say that the element \textit{normally generates} $\mathrm{Mod}(S_g)$ and we call this element a \textit{normal generator}. In $2018$, Lanier--Margalit \cite{LM} proved that for $g\geq 3$, every nontrivial periodic mapping class that is not a hyperelliptic involution normally generates $\mathrm{Mod}(S_g)$. \textit{Hyperelliptic involutions} are certain order-$2$ elements of $\mathrm{Mod}(S_g)$ that map to $-I\in\mathrm{Sp}_{2g}(\mathbb{Z})$ under the standard symplectic representation. They also showed that for $g\geq 3$, the normal closure of a hyperelliptic involution is the preimage of $\{\pm I\}$ under the symplectic representation.

\subsection{Proof idea} 
While our theorem is algebraic, we will use a geometric approach to study finite subgroups of $\mathrm{Aut}(F_n)$ and $\mathrm{Out}(F_n)$. A \textit{graph} is a $1-$dimensional CW-complex. An \textit{automorphism of a graph} is a cellular homeomorphism that is linear on the edges and has no inversions. Notice that any cellular homeomorphism that is linear on edges but inverts an edge can be made into an automorphism of the graph by barycentrically subdividing that edge. Let $X$ be a connected graph with fundamental group $F_n$ and $\mathrm{Aut}(X)$ be the automorphism group of the graph $X$. Let $(X, *)$ denote the graph $X$ with a basepoint. Since $X$ is a $K(F_n, 1)$, we can view $\mathrm{Aut}(F_n)$ as the group of homotopy classes of basepoint-preserving homotopy equivalences of  $X$ and $\mathrm{Out}(F_n)$ can be viewed as the group of homotopy classes of  homotopy equivalences of $X$. Notice that there are natural projections $\mathrm{Aut}(X,*)\rightarrow\mathrm{Aut}(F_n)$ and $\mathrm{Aut}(X)\rightarrow\mathrm{Out}(F_n)$. The groups $\mathrm{Aut}(X,*)$ and $\mathrm{Aut}(X)$ are finite.

Culler \cite{MC} showed that for any finite subgroup $G<\mathrm{Aut}(F_n)$, there exists a finite based graph $(X,*)$ and a subgroup $\tilde{G}<\mathrm{Aut}(X,*)$ such that the map $\mathrm{Aut}(X,*)\rightarrow\mathrm{Aut}(F_n)$ takes $\tilde{G}$ isomorphically onto $G$. He proved a similar statement for finite subgroups of $\mathrm{Out}(F_n)$. In these situations, we say that a finite subgroup  $G<\mathrm{Aut}(F_n)$ (resp. $G<\mathrm{Out}(F_n)$) is \textit{realized} by $\tilde{G}<\mathrm{Aut}(X,*)$ (resp. $\tilde{G}<\mathrm{Aut}(X)$). We will further explore the realization of finite subgroups in $\mathrm{Out}(F_n)$ by subgroups of $\mathrm{Aut}(X)$ in Section $4$. 

We prove our main theorems by studying the automorphism groups of graphs. Before we dive into the proof, let us see two examples.

\begin{convention}
The edges of our graphs will be labeled with labels like $s_i$, $\ell_i$ $(i\in\mathbb{Z})$, etc. The subscripts should always be taken modulo an appropriate number. For instance, if we have edges $s_1,\ldots,s_5$, then $s_6$ means $s_1$.
\end{convention}

\begin{figure}[htbp]
\centering
\begin{minipage}[t]{0.4\textwidth}
\centering
\includegraphics[scale=0.5]{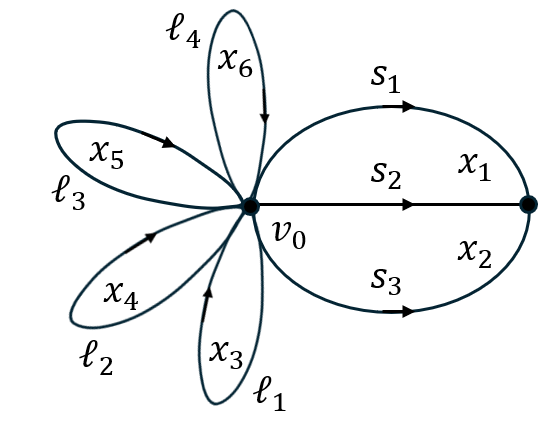}
\caption{Graph $X$ with $\pi_1(X,v_0)= F_6$\label{F_6}}
\end{minipage}
\begin{minipage}[t]{0.4\textwidth}
\centering
\includegraphics[scale=0.5]{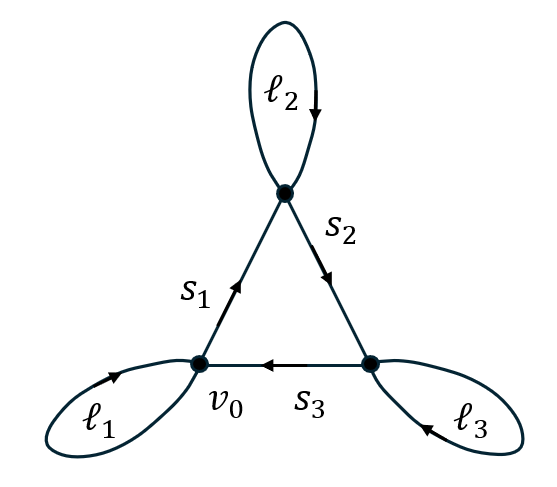}
\caption{Graph $Y$ with $\pi_1(Y,v_0)= F_4$\label{F_4}}
\end{minipage}
\end{figure}
\begin{example}\label{example2}
Let $X$ be the graph in Figure \ref{F_6}. This graph has oriented edges $\ell_1,\ldots, \ell_4$ and $s_1,\ldots,s_3$. Let  $\tilde{f}\in\mathrm{Aut}(X, v_0)$ be the automorphism that on these edges satisfies
\[ \tilde{f}(\ell_i)=\ell_{i+1}, \quad \tilde{f}(s_i)=s_{i+1}.\]
Consider the following basis of $\pi_1(X,v_0)$:
\[x_1=s_1s_2^{-1}, \quad x_2=s_2s_3^{-1},\quad x_3=\ell_1, \quad x_4=\ell_2,\quad x_5=\ell_3,\quad x_6=\ell_4.\]
Thus, the action of $\tilde{f}_*$ on $\pi_1(X,v_0)$ is 
\[\tilde{f}_*(x_1)=x_2, \quad \tilde{f}_*(x_2)=x_2^{-1}x_1^{-1},\quad \tilde{f}_*(x_3)=x_4, \quad \tilde{f}_*(x_4)=x_5,\quad \tilde{f}_*(x_5)=x_6,\quad \tilde{f}_*(x_6)=x_3.\]
Identify $\pi_1(X,v_0)$ with $F_6$. We see that $\tilde{f}$  realizes $f$ where $f$ is defined in Example \ref{example1}.
\end{example}

\begin{example}\label{example3}
Let $Y$ be the graph in Figure \ref{F_4}. Define $\tilde{g}\in\mathrm{Aut}(Y)$ via $\tilde{g}(s_i)=s_{i+1}$ and $\tilde{g}(\ell_i)=\ell_{i+1}$. Consider the following basis of $\pi_1(Y,v_0)$:
\[x_1 = \ell_1, \quad x_2 = s_1 \ell_2 s_1^{-1}, \quad x_3 = s_1 s_2 \ell_3 s_2^{-1} s_1^{-1}, \quad x_4 = s_1 s_2 s_3.\]
The action $\tilde{g}_*$ on $\pi_1(Y, v_0)$ satisfies
\[\tilde{g}_*(x_1)=x_2, \quad \tilde{g}_*(x_2)=x_3,\quad \tilde{g}_*(x_3)=x_4x_1x_4^{-1}, \quad \tilde{g}_*(x_4)=x_4.\]
Identify $\pi_1(Y,v_0)$ with $F_4$. We see that $\tilde{g}$  realizes $g$ where $g$ is defined in Example \ref{example1}.
\end{example}
\begin{remark}
In the realm of $\mathrm{Mod}(S_g)$,  we have the Nielsen realization theorem of Kerckhoff \cite{SK}. Let $\text{Homeo}^+(S_g)$ denote the orientation-preserving homeomorphism group of $S_g$. This theorem says that for a finite subgroup $G$ in $\mathrm{Mod}(S_g)$, there exists a finite group $\tilde{G}$ in $\text{Homeo}^+(S_g)$ such that the natural projection $\text{Homeo}^+(S_g)\rightarrow \mathrm{Mod}(S_g)$ restricts to an isomorphism $\tilde{G}\rightarrow G$. In other words, every finite subgroup of $\mathrm{Mod}(S_g)$ comes from a finite subgroup of $\text{Homeo}^+(S_g)$. The realization theorem in $\mathrm{Out}(F_n)$ that Culler proved is a $1-$dimensional analog of  the Nielsen realization theorem.
\end{remark}
\subsection{The case $p=2$}
Let $f$ be an order $2$ element in $\mathrm{Aut}(F_n)$. Our main theorem for $p=2$ is too complicated to state here. We will discuss it in Section $5$. There is one case we are unable to handle. Consider the following question:
\begin{question}\label{question}
Given $f\in\mathrm{Aut}(F_n)$ defined as follows:
\begin{align*}
f:&x_0\mapsto x_0 \\      
    &x_1 \mapsto x_1^{-1} \\       
    &x_2 \mapsto x_1x_2^{-1}x_1^{-1} \\  
    &\vdots \quad\quad\quad\vdots \\                          
    &x_{n-1} \mapsto (x_1 \ldots x_{n-2})x_{n-1}(x_{n-2} \ldots x_1^{-1})^{-1}
\end{align*}
What is the normal closure of $f$?
\end{question}

\subsection{The case $n = 2$}
The condition $n\geq 3$ in our theorem is necessary. Here is what happens for $n=2$. Nielsen \cite{JN} showed that $\mathrm{Out}(F_2)\cong\mathrm{GL}_2(\mathbb{Z})$. One can check that all finite order matrices in $\mathrm{GL}_2(\mathbb{Z})$ with order $\geq 3$ are in $\mathrm{SL}_2(\mathbb{Z})$. Since
\[
\mathrm{SL}_2(\mathbb{Z})\cong\mathbb{Z}/4\mathbb{Z}\underset{\mathbb{Z}/2\mathbb{Z}}{\ast}\mathbb{Z}/6\mathbb{Z},\]
if a matrix $M\in\mathrm{SL}_2(\mathbb{Z})$ has finite order, then it is conjugate to either the $\mathbb{Z}/4\mathbb{Z}$-factor or the $\mathbb{Z}/6\mathbb{Z}$-factor; see, e.g., \cite{JS}. The order of $M$ will be $2,3,4$ or $6$. Let $G$ be the normal closure of $M$. We have
\[
    \mathrm{SL}_2(\mathbb{Z})/G \cong 
    \begin{cases}
        \mathrm{PSL}_2(\mathbb{Z})& \text{if  order} (M) = 2, \\
        \mathbb{Z}/4\mathbb{Z} & \text{if order}(M) = 3, \\
        \mathbb{Z}/3\mathbb{Z} & \text{if order}(M) = 4, \\
        \mathbb{Z}/2\mathbb{Z} & \text{if order}(M) = 6. \\
    \end{cases}
\]
Any order $2$ matrix in $\mathrm{GL}_2(\mathbb{Z})$ that does not lie in $\mathrm{SL}_2(\mathbb{Z})$ is conjugate to
\[
\begin{pmatrix}
-1 & 0 \\
0 & 1
\end{pmatrix}.
\]
One can show that the normal closure of this matrix is the level $2$ subgroup of $\mathrm{GL}_2(\mathbb{Z})$.

\subsection{Outline of paper}
For a finite subgroup $G<\mathrm{Aut}(F_n)$ (resp. $G<\mathrm{Out}(F_n)$) with $|G|$ not a power of $2$, it is enough to prove the theorems for the generators of $G$ with odd prime order. Sections $3$ and $4$ of the paper are devoted to proving the following two propositions. We have a partial theorem for the $p=2$ case. We will state and discuss it in Section $5$.
\begin{prop}
\label{prop1}
Let $f$ be an order $p$ element in $\mathrm{Aut}(F_n)$, where $n\geq 3$ and $p$ is an odd prime number. Let $\bar{f}$ be the image of $f$ in  $\mathrm{GL}_n(\mathbb{Z})$. Then,
\[
    N(f) = 
    \begin{cases}
        \mathrm{SAut}(F_n) & \text{if det}(\bar{f}) = 1, \\
        \mathrm{Aut}(F_n) & \text{if det}(\bar{f}) = -1.
    \end{cases}
\]
\end{prop}

\begin{prop}\label{prop2}
Let $f$ be an order $p$ element in $\mathrm{Out}(F_n)$, where $n\geq 3$ and $p$ is an odd prime number. Let $\bar{f}$ be the image of $f$ in  $\mathrm{GL}_n(\mathbb{Z})$. Then,
\[
    N(f) = 
    \begin{cases}
        \mathrm{SOut}(F_n) & \text{if det}(\bar{f}) = 1, \\
        \mathrm{Out}(F_n) & \text{if det}(\bar{f}) = -1.
    \end{cases}
\]
\end{prop}

\section{Notation}
We now introduce some notation.

\begin{defn}\label{defn}
Let $F_n = \langle x_1, x_2, \ldots, x_n \rangle$ be the free group of rank $n$ generated by $x_1, x_2, \ldots, x_n$.
\begin{align*}
    L_{x_i x_j} & = \text{the automorphism that takes $x_i$ to $x_jx_i$ and it fixes $x_k$ for $k\neq i$}, \\
    R_{x_i x_j} & = \text{the automorphism that takes $x_i$ to $x_ix_j$ and it fixes $x_k$ for $k\neq i$}, \\
    C_{x_i x_j} & = \text{the automorphism that takes $x_i$ to $x_jx_ix_j^{-1}$ and it fixes $x_k$ for $k\neq i$}, \\
    P_{x_i x_j} & = \text{the automorphism that takes $x_i$ to $x_j$, $x_j$ to $x_i$ and it fixes $x_k$ for $k\neq i, j$}, \\
    I_{x_i} & = \text{the automorphism that takes $x_i$ to $x_i^{-1}$ and it fixes $x_k$ for $k\neq i$}.
\end{align*}
We call $L_{x_i x_j}$ a left multiplication, $R_{x_i x_j}$ a right multiplication, $C_{x_i x_j}$ a conjugation, $P_{x_i x_j}$ a permutation, and $I_{x_i}$ an inversion map. 
\end{defn}
\begin{remark}The inverse of $L_{x_i x_j}$ is the automorphism that takes $x_i$ to $x_j^{-1}x_i$ and it fixes $x_k$ for $k\neq i$. Similarly, the inverse of $R_{x_i x_j}$ is the automorphism that takes $x_i$ to $x_ix_j^{-1}$ and it fixes the other generators. The elements $P_{x_i x_j}$ and $I_{x_i}$ are order $2$ automorphisms whose inverses are themselves.
\end{remark}
\begin{defn}
The kernel of the $\mathrm{GL}_n(\mathbb{Z})$ representation of $\mathrm{Aut}(F_n)$ is called the Torelli subgroup of $\mathrm{Aut}(F_n)$, denoted $\text{IA}_n$. Similarly, the Torelli subgroup of $\mathrm{Out}(F_n)$ is denoted as $\text{IO}_n$.
\end{defn}

\section{Proof of the $\mathrm{Aut}(F_n)$ case}
The starting point for our proof is the following proposition and lemmas. By saying an automorphism $\tilde{f}$ of a based graph $(X,v_0)$ is a \textit{realization} of $f\in\mathrm{Aut}(F_n)$, we mean the cyclic group generated by $f$ is realized by the cyclic group generated by $\tilde{f}$. 

\begin{prop}
\label{collapse}
Let $f\in\mathrm{Aut}(F_n)$ have prime order $p$. Then $f$ can be realized by an automorphism $\tilde{f}$ of a based graph $(X, v_0)$ such that the following hold:
\begin{enumerate}
    \item the automorphism $\tilde{f}$ fixes all vertices of $X$; and
    \item the only edges fixed by $\tilde{f}$ are loops.
\end{enumerate}
\end{prop}

\begin{proof}
Culler \cite{MC} proved that an element $f\in\mathrm{Aut}(F_n)$ can be realized by an automorphism $\tilde{f}$ of a based graph $(X, v_0)$. Pick such an $(X, v_0)$ with as few vertices as possible. We claim that $\tilde{f}$ fixes all vertices of $X$. Indeed, if it acts nontrivially on some vertex, then since $\tilde{f}(v_0)=v_0$, there must be an edge $e$ from a vertex $v_1$ to a vertex $v_2$ such that $\tilde{f}$ fixes $v_1$ but not $v_2$. We claim that the $\cup_{k=1}^{p}\tilde{f}^{k}(e)$ is a tree. In fact, since each $\tilde{f}^k(e)$ is an edge with the same endpoint $v_1=f^k(v_1)$ for $1\leq k\leq p$, it is enough to prove the other endpoints $f^k(v_2)$ with $1\leq k\leq p$ are distinct. Indeed, they are all distinct since the order $p$ cyclic group has no nontrivial subgroups.

Then, we can collapse the tree  $\cup_{k=1}^{p}\tilde{f}^{k}(e)$ without changing the homotopy type to get a new graph with fewer vertices. This contradicts the minimality of the number of vertices of $X$. For such an $(X, v_0)$, the only edges fixed by $\tilde{f}$ are loops since if there is a fixed edge with distinct endpoints, we can always collapse it to one point, contradicting the minimality of the number of vertices of $X$ as well. 
\end{proof}
We call an automorphism $\tilde{f}$ of a based graph $(X, v_0)$ \textit{good} if it satisfies the conclusion of Proposition \ref{collapse}.

\begin{remark}
Notice that the order of $f$ being prime is essential. For instance, let $f$ be an order $4$ element in $\mathrm{Aut}(F_4)$ such that
\begin{align*}
f: &x_1 \mapsto x_2^{-1}  \\
     &x_2 \mapsto x_1\\
     &x_3 \mapsto x_2x_4x_2^{-1}\\
     &x_4 \mapsto x_3
\end{align*}
Let $X$ be the graph in Figure \ref{order 4}. This graph has oriented edges $\ell_1,\ell_2$ and $s_1,\ldots,s_4$. Let  $\tilde{f}\in\mathrm{Aut}(X, v_0)$ be the automorphism that on these edges satisfies
\[ \tilde{f}(\ell_m)=\ell_{m+1}, \quad \tilde{f}(s_i)=s_{i+1}^{-1}.\]
Consider the following basis of $\pi_1(X,v_0)$:
\[x_1=s_1^{-1}s_3, \quad x_2=s_4s_2^{-1},\quad x_3=s_3\ell_1 s_3^{-1}, \quad x_4=s_2\ell_2 s_2^{-1}.\]
Thus, the action of $\tilde{f}_*$ on $\pi_1(X,v_0)$ satisfies:
\[\tilde{f}_*(x_1)=x_2^{-1}, \quad \tilde{f}_*(x_2)=x_1,\quad \tilde{f}_*(x_3)=x_2x_4x_2^{-1}, \quad \tilde{f}_*(x_4)=x_3.\]
It realizes the map $f$ defined above. We can see that the map $\tilde{f}$ has a fixed point $v_0$ while $\cup_{k=1}^{4}\tilde{f}^{k}(s_1)$ is not a tree.
\begin{figure}[htbp]
    \centering
    \includegraphics[scale=0.55]{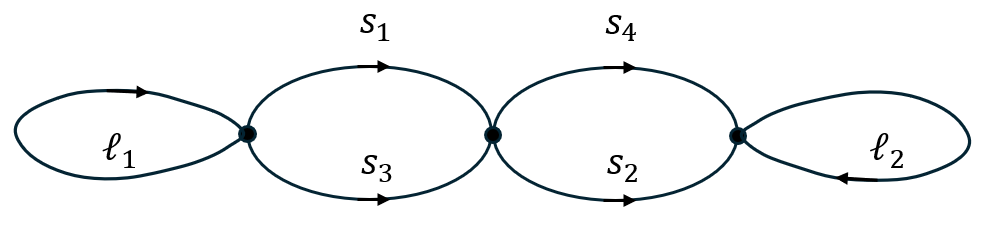}
    \caption{\label{order 4}}
\end{figure} 
\end{remark}

The following lemmas reduce computing the normal closure of $f$ to finding some specific elements inside $N(f)$. Denote the image of $f$ in  $\mathrm{GL}_n(\mathbb{Z})$ by $\bar{f}$.

\begin{lemma}\label{left multi}
Let $f\in\mathrm{Aut}(F_n)$.  If there exists a left multiplication $L_{x_i x_j}$ in $N(f)$, then
\[
    N(f) = 
    \begin{cases}
        \mathrm{SAut}(F_n) & \text{if det}(\bar{f}) = 1, \\
        \mathrm{Aut}(F_n) & \text{if det}(\bar{f}) = -1.
    \end{cases}
\]
\end{lemma}
\begin{proof}
Any left multiplications $L_{x_ix_j}$ and $L_{x_mx_{\ell}}$ are conjugated by a $\sigma\in\mathrm{Aut}(F_n)$ that permutes the generators with $\sigma(x_i)=x_m$ and $\sigma(x_j)=x_l$. Without loss of generality, assume $L_{x_1x_2}\in N(f)$. Then $L_{x_1x_2}^{-1}\in N(f)$ as well which implies the right multiplication $R_{x_1x_2} = I_{x_1} L_{x_1x_2}^{-1} I_{x_1}\in N(f)$. The conjugation map $C_{x_1x_2}$ can be written as $C_{x_1x_2} = L_{x_1x_2} R_{x_1x_2}$ which is in $N(f)$. Magnus \cite{WM} proved that the Torelli subgroup $IA_n$ of  $\mathrm{Aut}(F_n)$ is $\mathrm{Aut}(F_n)$-normally generated by $C_{x_1x_2}$.  Day--Putman \cite{DP} give a modern proof of this theorem.  Hence, we can conclude that $N(f)$ contains $IA_n$. Notice that $\bar{L}_{x_1x_2}$ normally generates SL$_n(\mathbb{Z})$. Thus we have the group extension
\[
    1 \rightarrow IA_n \rightarrow N(f) \rightarrow \mathrm{SL}_n(\mathbb{Z}) \rightarrow 1
\]
if $\det(\bar{f}) = 1$. The group $\mathrm{GL}_n(\mathbb{Z})$ is generated by $\mathrm{SL}_n(\mathbb{Z})$ and any determinant $-1$ element. Then we have the group extension
\[
    1 \rightarrow IA_n \rightarrow N(f) \rightarrow \mathrm{GL}_n(\mathbb{Z}) \rightarrow 1
\]
if $\det(\bar{f}) = -1$. The lemma follows.
\end{proof}
We also have the following $\mathrm{Out}(F_n)$-analogue of Lemma \ref{left multi}.

\begin{lemma}\label{left multi Out}
Let $f\in\mathrm{Out}(F_n)$.  If there exists a left multiplication $L_{x_i x_j}$ in $N(f)$, then
\[
    N(f) = 
    \begin{cases}
        \mathrm{SOut}(F_n) & \text{if det}(\bar{f}) = 1, \\
        \mathrm{Out}(F_n) & \text{if det}(\bar{f}) = -1.
    \end{cases}
\]
\end{lemma}
\begin{proof}
Identical to the proof of Lemma \ref{left multi}.
\end{proof}
We make the following definition.

\begin{defn}
The projection $\mathbb{Z}\rightarrow\mathbb{Z}/\ell\mathbb{Z}$ induces a map $\pi_{\ell}: \mathrm{SL}_n(\mathbb{Z})\rightarrow\mathrm{SL}_n(\mathbb{Z}/\ell\mathbb{Z})$.  The kernel of  $\pi_{\ell}$, denoted $\Gamma_n(\ell)$, is called the level $\ell$ congruence subgroup of  $\mathrm{SL}_n(\mathbb{Z})$.
\end{defn}

Denote the image of $G<\mathrm{Aut}(F_n)$ in $\mathrm{GL}_n(\mathbb{Z})$ by $\bar{G}$. We then have:

\begin{lemma}\label{level 3}
Let $f\in\mathrm{Aut}(F_n)$. Assume the following three conditions are satisfied:
\begin{enumerate}
\item the Torelli subgroup $IA_n\subset N(f)$,
\item the level $3$ subgroup $\Gamma_n(3)\subset\overline{N(f)}$,  and
\item the group $\mathrm{Im}(\overline{N(f)}\rightarrow\mathrm{PSL}_n(\mathbb{Z}/3\mathbb{Z}))$ is nontrivial.
\end{enumerate}
Then, we have
\[
    N(f) = 
    \begin{cases}
        \mathrm{SAut}(F_n) & \text{if det}(\bar{f}) = 1, \\
        \mathrm{Aut}(F_n) & \text{if det}(\bar{f}) = -1.
    \end{cases}
\]
\end{lemma}
\begin{proof}
Let $G = \text{Im}(\overline{N(f)} \cap \mathrm{SL}_n(\mathbb{Z}) \rightarrow \mathrm{SL}_n(\mathbb{Z}/3\mathbb{Z}))$ and $\hat{G} = \text{Im}(G \rightarrow \mathrm{PSL}_n(\mathbb{Z}/3\mathbb{Z}))$ where the maps are the natural projection maps.  Then we have the short exact sequence:
\[  1 \rightarrow \Gamma_n(3) \rightarrow \overline{N(f)} \cap \mathrm{SL}_n(\mathbb{Z}) \rightarrow G \rightarrow 1\]
Our goal is to prove $G = \mathrm{SL}_n(\mathbb{Z}/3\mathbb{Z})$. This assertion implies that $\overline{N(f)}$ contains $\mathrm{SL}_n(\mathbb{Z})$ and the lemma will follow as in the proof of Lemma \ref{left multi}. 

According to hypothesis $3$, we have that $\hat{G}$ is a nontrivial normal subgroup in $\mathrm{PSL}_n(\mathbb{Z}/3\mathbb{Z})$. Since $\mathrm{PSL}_n(\mathbb{Z}/3\mathbb{Z})$ is simple for $n \geq 3$, see \cite{CJ}, we conclude that $\hat{G} = \mathrm{PSL}_n(\mathbb{Z}/3\mathbb{Z})$, making the map $G\rightarrow \mathrm{PSL}_n(\mathbb{Z}/3\mathbb{Z})$ a surjection. 

Next, we consider the center $Z(\mathrm{SL}_n(\mathbb{Z}/3\mathbb{Z}))$ of $\mathrm{SL}_n(\mathbb{Z}/3\mathbb{Z})$. We have
\[Z(\mathrm{SL}_n(\mathbb{Z}/3\mathbb{Z}))=\{\lambda I|\lambda^n \equiv 1 \text{ mod } 3\} =\begin{cases}
        \{I\} & \text{if }n \text{ is odd}, \\
         \{\pm I\} & \text{if }n  \text{ is even}.
    \end{cases}.\]
When $n$ is odd, $\mathrm{SL}_n(\mathbb{Z}/3\mathbb{Z})\cong \mathrm{PSL}_n(\mathbb{Z}/3\mathbb{Z})$. Hence, in this case, $G = \mathrm{SL}_n(\mathbb{Z}/3\mathbb{Z})$. When $n$ is even, the projection $ \mathrm{SL}_n(\mathbb{Z}/3\mathbb{Z})\rightarrow  \mathrm{PSL}_n(\mathbb{Z}/3\mathbb{Z})$ does not split, see \cite{KC}. Since $G\rightarrow \mathrm{PSL}_n(\mathbb{Z}/3\mathbb{Z})$is a surjection, if $G\rightarrow \mathrm{PSL}_n(\mathbb{Z}/3\mathbb{Z})$ was an injection, then it would be an isomorphism. Its inverse $\mathrm{PSL}_n(\mathbb{Z}/3\mathbb{Z})\rightarrow G\subset\mathrm{SL}_n(\mathbb{Z}/3\mathbb{Z})$ would be a splitting. Contradiction! Therefore, we conclude that $G\rightarrow \mathrm{PSL}_n(\mathbb{Z}/3\mathbb{Z})$ is not an injection. The kernel $K$ of $G\rightarrow \mathrm{PSL}_n(\mathbb{Z}/3\mathbb{Z})$ is a nontrivial subgroup of $Z(\mathrm{SL}_n(\mathbb{Z}/3\mathbb{Z}))\cong \mathbb{Z}/2\mathbb{Z}$, so $K=Z(\mathrm{SL}_n(\mathbb{Z}/3\mathbb{Z}))$. This implies $G = \mathrm{SL}_n(\mathbb{Z}/3\mathbb{Z})$ when $n $ is even.
\end{proof}

We have a similar lemma for the situation when level $4$ subgroup $\Gamma_n(4)\subset\overline{N(f)}$. We need to introduce a short exact sequence first.

\begin{prop}\label{ses}
Let $M^0_n(\mathbb{Z}/2\mathbb{Z})=\{ \text{trace zero matrices over }\mathbb{Z}/2\mathbb{Z}\}$. Then we have the short exact sequence
\[
1 \rightarrow M^0_n(\mathbb{Z}/2\mathbb{Z}) \rightarrow \mathrm{SL}_n(\mathbb{Z}/4\mathbb{Z}) \rightarrow \mathrm{SL}_n(\mathbb{Z}/2\mathbb{Z}) \rightarrow 1.
\]
\end{prop}
\begin{proof}
By the definition of the congruence subgroup $\Gamma_n(2)$ of $\mathrm{SL}_n(\mathbb{Z})$, we have the short exact sequence
\[
1 \rightarrow\Gamma_n(2) \hookrightarrow \mathrm{SL}_n(\mathbb{Z}) \twoheadrightarrow \mathrm{SL}_n(\mathbb{Z}/2\mathbb{Z}) \rightarrow 1.
\]
Elements $M\in \Gamma_n(2)$ can be written as $M=I+2A$ where $A$ is an $n$ by $n$ matrix. Lee--Szczarba \cite{LeeS} proved that for $M=I+2A\in \Gamma_n(2)$, the trace of $A$ is $0$ mod $2$, and the map $\Psi: \Gamma_n(2)\rightarrow M^0_n(\mathbb{Z}/2\mathbb{Z})$ defined by $\Psi(I+2A)=A \mod 2$ is a surjective homomorphism. Since $\mathrm{Ker}(\Psi)=\Gamma_n(4)$, we thus have $\Gamma_n(2)/\Gamma_n(4)\cong M^0_n(\mathbb{Z}/2\mathbb{Z})$ and a short exact sequence
\begin{center}
\begin{tikzcd}
1 \arrow{r} & \Gamma_n(2)/\Gamma_n(4) \arrow{r} \arrow[d, phantom, sloped, "\cong"] & \mathrm{SL}_n(\mathbb{Z})/\Gamma_n(4) \arrow{r} \arrow[d, phantom, sloped, "\cong"] & \mathrm{SL}_n(\mathbb{Z}/2\mathbb{Z}) \arrow{r} & 1\\
            & M^0_n(\mathbb{Z}/2\mathbb{Z}) & \mathrm{SL}_n(\mathbb{Z}/4\mathbb{Z}) & &
\end{tikzcd}
\end{center}
which proves the proposition.
\end{proof}

\begin{lemma}
\label{level 4}
Let $f\in\mathrm{Aut}(F_n)$. Assume the following three conditions are satisfied:
\begin{enumerate}
\item the Torelli subgroup $IA_n\subset N(f)$,
\item the level $4$ subgroup $\Gamma_n(4)\subset\overline{N(f)}$,  and
\item the group $\mathrm{Im}(\overline{N(f)}\rightarrow\mathrm{SL}_n(\mathbb{Z}/2\mathbb{Z}))$ is nontrivial.
\end{enumerate}
Then, when $n\neq 4$, we have
\[
    N(f) = 
    \begin{cases}
        \mathrm{SAut}(F_n) & \text{if det}(\bar{f}) = 1, \\
        \mathrm{Aut}(F_n) & \text{if det}(\bar{f}) = -1.
    \end{cases}
\]
\end{lemma}

\begin{proof}
Let $G = \text{Im}(\overline{N(f)} \cap \mathrm{SL}_n(\mathbb{Z}) \rightarrow \mathrm{SL}_n(\mathbb{Z}/4\mathbb{Z}))$ and $\hat{G} = \text{Im}(G \rightarrow \mathrm{SL}_n(\mathbb{Z}/2\mathbb{Z}))$ where the maps are the natural projection maps.  Then we have the short exact sequence:
\[  1 \rightarrow \Gamma_n(4) \rightarrow \overline{N(f)} \cap \mathrm{SL}_n(\mathbb{Z}) \rightarrow G \rightarrow 1\]
Our goal is to prove $G = \mathrm{SL}_n(\mathbb{Z}/4\mathbb{Z})$. This will imply that $\overline{N(f)}$ contains $\mathrm{SL}_n(\mathbb{Z})$ and the lemma will follow as in the proof of Lemma \ref{left multi}. 

Let $G'=\mathrm{Ker}(G\rightarrow\hat{G})$. The short exact sequence in Proposition \ref{ses} induces the short exact sequence:
\begin{align}\label{ses_G}
    1 \rightarrow G' \rightarrow G \rightarrow \hat{G} \rightarrow 1
\end{align}
By hypothesis $3$, we have that $\hat{G}$ is a nontrivial normal subgroup in $\mathrm{SL}_n(\mathbb{Z}/2\mathbb{Z})$. Since $\mathrm{SL}_n(\mathbb{Z}/2\mathbb{Z})$ is simple for all $n \geq 3$, we have $\hat{G} = \mathrm{SL}_n(\mathbb{Z}/2\mathbb{Z})$.

We have $G'<M^0_n(\mathbb{Z}/2\mathbb{Z})$. It is closed under conjugation by $G$. Since $M^0_n(\mathbb{Z}/2\mathbb{Z})$ is abelian, this conjugation action factors through $\hat{G} = \mathrm{SL}_n(\mathbb{Z}/2\mathbb{Z})$. The action of $\mathrm{SL}_n(\mathbb{Z}/2\mathbb{Z})$ on $M^0_n(\mathbb{Z}/2\mathbb{Z})$ is the usual adjoint representation, so $G'$ is a subrepresentation of $M^0_n(\mathbb{Z}/2\mathbb{Z})$. By \cite{LS},  the group $G'$ can only be either $\{I\}$ or $M^0_n(\mathbb{Z}/2\mathbb{Z})$ when $n$ is odd, and group $G'$ can be the group of  scalar matrices as well when $n$ is even. If $G' = \{I\}$, then $G\cong \hat{G}$ where $\hat{G} = \mathrm{SL}_n(\mathbb{Z}/2\mathbb{Z})$. By \cite[Theorem $7$ in \S $2$]{SH}, when $n\geq 3$, the map $\mathrm{SL}_n(\mathbb{Z}/4\mathbb{Z}) \rightarrow \mathrm{SL}_n(\mathbb{Z}/2\mathbb{Z})$ does not split which means we cannot find a subgroup in $\mathrm{SL}_n(\mathbb{Z}/4\mathbb{Z})$ that is isomorphic to $\mathrm{SL}_n(\mathbb{Z}/2\mathbb{Z})$. Thus, $G'\neq\{I\}$. 

When $n$ is even, if $G'$ is the group of scalar matrices, that is to say, $G' = \{\pm I\}$, then we consider the short exact sequence in Proposition \ref{ses} quotiented by the center subgroup: 
\[
1 \rightarrow M^0_n(\mathbb{Z}/2\mathbb{Z})/ \{\mathbf{0}, I\} \rightarrow \mathrm{PSL}_n(\mathbb{Z}/4\mathbb{Z}) \rightarrow \mathrm{PSL}_n(\mathbb{Z}/2\mathbb{Z}) \rightarrow 1.
\]
The short exact sequence (\ref{ses_G}) gives us 
\[G/G'\cong \hat{G}/G',
\]
where $\hat{G}/G'=\mathrm{PSL}_n(\mathbb{Z}/2\mathbb{Z})$. By Proposition \ref{split}, which we will prove below, the map $\mathrm{PSL}_n(\mathbb{Z}/4\mathbb{Z}) \rightarrow \mathrm{PSL}_n(\mathbb{Z}/2\mathbb{Z})$ does not split when $n \geq 6$. This implies that there is no subgroup in $\mathrm{PSL}_n(\mathbb{Z}/4\mathbb{Z})$ that is isomorphic to $\mathrm{PSL}_n(\mathbb{Z}/2\mathbb{Z})$. Thus, $G/G'\not\cong\mathrm{PSL}_n(\mathbb{Z}/2\mathbb{Z})$ and then $G'\neq\{\pm I\}$. We conclude that $G'$ is the whole group $M^0_n(\mathbb{Z}/2\mathbb{Z})$ so that $G = \mathrm{SL}_n(\mathbb{Z}/4\mathbb{Z})$.
\end{proof}

Before we proceed with the proof of Proposition \ref{split} mentioned in Lemma \ref{level 4}, there are two powerful propositions derived from Lemma \ref{level 3} and Lemma \ref{level 4}. Let $e_{k,r}$  denote the elementary matrix with 1’s along the diagonal and at position $ (k, r)$, and $0$’s elsewhere. 

\begin{prop}
\label{level 3 prop}
Let $f\in\mathrm{Aut}(F_n)$. Assume the following three conditions are satisfied:
\begin{enumerate}
\item there exist distinct $1\leq i,j\leq n$ such that the conjugation map $C_{x_ix_j}\in N(f)$,
\item there exist distinct $1\leq k,r\leq n$ such that the matrix $e_{k,r}^3\in\overline{N(f)}$,  and
\item there exists a matrix $\Phi\in \overline{N(f)}$ such that $\Phi\not\equiv I$ mod $ 3$. 
\end{enumerate}
Then, we have
\[
    N(f) = 
    \begin{cases}
        \mathrm{SAut}(F_n) & \text{if det}(\bar{f}) = 1, \\
        \mathrm{Aut}(F_n) & \text{if det}(\bar{f}) = -1.
    \end{cases}
\]
\end{prop}
\begin{proof}
Magnus \cite{WM} proved that the Torelli subgroup $IA_n$ of  $\mathrm{Aut}(F_n)$ is $\mathrm{Aut}(F_n)$-normally generated by $C_{x_ix_j}$ for distinct $1\leq i,j\leq n$. Thus, assumption $1$ tells us $IA_n\subset N(f)$. According to \cite{BMS}, the group $\Gamma_n(\ell)$ is $\mathrm{SL}_n(\mathbb{Z})$-normally generated by any single $e_{k,r}^\ell$ with distinct $1\leq k,r\leq n$. Therefore, $e_{k,r}^3\in\overline{N(f)}$ implies that $\Gamma_n(3)\in\overline{N(f)}$. Assumption $3$ directly shows the group $\mathrm{Im}(\overline{N(f)}\rightarrow\mathrm{PSL}_n(\mathbb{Z}/3\mathbb{Z}))$ is nontrivial. Applying Lemma \ref{level 3}, we derive the stated result.
\end{proof}

Similarly, we have the following proposition.
\begin{prop}
\label{level 4 prop}
For $n\neq 4$, consider $f\in\mathrm{Aut}(F_n)$. Assume the following three conditions are satisfied:
\begin{enumerate}
\item there exist distinct $1\leq i,j\leq n$ such that the conjugation map $C_{x_ix_j}\in N(f)$,
\item there exist distinct $1\leq k,r\leq n$ such that the matrix $e_{k,r}^4\in\overline{N(f)}$,  and
\item there exists a matrix $\Phi\in \overline{N(f)}$ such that $\Phi\not\equiv I$ mod $ 2$. 
\end{enumerate}
Then, we have
\[
    N(f) = 
    \begin{cases}
        \mathrm{SAut}(F_n) & \text{if det}(\bar{f}) = 1, \\
        \mathrm{Aut}(F_n) & \text{if det}(\bar{f}) = -1.
    \end{cases}
\]
\end{prop}

\begin{proof}
The proof follows a similar argument to that of Proposition \ref{level 3 prop}.
\end{proof}

Now, let us prove the Proposition \ref{split} mentioned in Lemma \ref{level 4}. Given the short exact sequence \[
1 \rightarrow M^0_n(\mathbb{Z}/2\mathbb{Z}) \rightarrow \mathrm{SL}_n(\mathbb{Z}/4\mathbb{Z}) \rightarrow \mathrm{SL}_n(\mathbb{Z}/2\mathbb{Z}) \rightarrow 1,
\]
taking the quotient by the center of each matrix group, we have the following short exact sequence: \[
1 \rightarrow M^0_n(\mathbb{Z}/2\mathbb{Z})/\{\mathbf{0},I\} \rightarrow \mathrm{PSL}_n(\mathbb{Z}/4\mathbb{Z}) \rightarrow \mathrm{PSL}_n(\mathbb{Z}/2\mathbb{Z}) \rightarrow 1.
\]
We would like to prove the following property.
\begin{prop}\label{split}
For an even integer $n\geq 6$, the map $\mathrm{PSL}_n(\mathbb{Z}/4\mathbb{Z}) \rightarrow \mathrm{PSL}_n(\mathbb{Z}/2\mathbb{Z})$ does not split.
\end{prop}
\begin{proof}
First, we introduce some notation.
\begin{align*}
a_{i,j} &= n \times n  \text{ matrix over }\mathbb{Z}/2\mathbb{Z} \text{ with }(i,j) \text{ entry  }1  \text{ and other entries } 0;\\
e_{i,j} &= I + a_{i,j}, \text{ which is an elementary matrix over }\mathbb{Z}/2\mathbb{Z};\\
A_{i,j} &= n \times n  \text{ matrix over }\mathbb{Z}/4\mathbb{Z} \text{ with }(i,j) \text{ entry  }1  \text{ and other entries } 0;\\
E_{i,j} &= I + A_{i,j}, \text{ which is an elementary matrix over }\mathbb{Z}/4\mathbb{Z};\\
[g] &= \text{ the equivalence class of $g\in\mathrm{SL}_n(\mathbb{Z}/4\mathbb{Z})$ in the quotient group }\mathrm{PSL}_n(\mathbb{Z}/4\mathbb{Z}) .
\end{align*}
These elements map as follows: 
\begin{center}
\begin{tikzcd}[row sep=0, column sep=small]
0 \arrow[r] & M^0_n(\mathbb{Z}/2\mathbb{Z}) \arrow[r] & \mathrm{SL}_n(\mathbb{Z}/4\mathbb{Z}) \arrow[r] & \mathrm{SL}_n(\mathbb{Z}/2\mathbb{Z}) \arrow[r] & 1 \\
            & a_{ij} \arrow[r, mapsto]                & I + 2A_{ij} = E_{ij}^2                          &                                                 &   \\
            &                                         & E_{ij} \arrow[r, mapsto]                        & e_{ij}                                          &
\end{tikzcd}
\end{center}

Assume there exists a splitting $\sigma:\mathrm{PSL}_n(\mathbb{Z}/2\mathbb{Z}) \rightarrow \mathrm{PSL}_n(\mathbb{Z}/4\mathbb{Z})$. There then exists a matrix $X_{i,j}\in  M^0_n(\mathbb{Z}/4\mathbb{Z})$ with entries $0$ or $1$ such that  $\sigma(e_{i,j})=[(I+2X_{i,j})E_{i,j}]$. Let $X_{i,j}=(x_{\ell,m}^{i,j})$, where $x_{\ell,m}^{i,j}$ denotes the $(\ell,m)$ entry of the matrix $X_{i,j}$. 

Since $e_{i,j}$ has order $2$, we have
\[\sigma(e_{i,j}^2)=\sigma(e_{i,j})^2=[(I+2X_{i,j})E_{i,j}]^2=[I].\]
That is to say
\[((I+2X_{i,j})E_{i,j})^2 = I\text{ or }-I.\]
Since we are working over $\mathbb{Z}/4\mathbb{Z}$, after expanding and simplifying the left hand side, we get
\[I+2A_{i,j}+2A_{i,j}X_{i,j}+2XA_{i,j}+2A_{i,j}X_{i,j}A_{i,j}= I\text{ or }-I.\]
Therefore, we have
\begin{align}\label{eq1}
2A_{i,j}+2A_{i,j}X_{i,j}+2X_{i,j}A_{i,j}+2A_{i,j}X_{i,j}A_{i,j}= \mathbf{0}\text{ or }2I \text{ over }\mathbb{Z}/4\mathbb{Z}.
\end{align}
By examining the left hand side of \eqref{eq1}, the diagonal entries are all $0$ except for the $(i,i)$ entry and the $(j,j)$ entry. Thus, when $n\geq 4$, the only possibility is 
\begin{align}\label{eq2}
2A_{i,j}+2A_{i,j}X_{i,j}+2X_{i,j}A_{i,j}+2A_{i,j}X_{i,j}A_{i,j}= \mathbf{0}\text{ over }\mathbb{Z}/4\mathbb{Z}.
\end{align}
This gives us that the $i$-th column and $j$-th row of $X_{i,j}$ are all $0$ except for the $(i,i)$ entry and the $(j,j)$ entry. Moreover, we have $2x_{i,i}^{i,j}+2x_{j,j}^{i,j}\equiv 2$ mod $4$, or equivalently, 
\begin{align}\label{trace=1}
x_{i,i}^{i,j}+x_{j,j}^{i,j}\equiv 1 \mod 2.
\end{align}
On the other hand, since
\[e_{i,j}e_{u,v}=e_{u,v}e_{i,j}\quad\text{ for } j\neq u, \text{ and } i\neq v,\]
we have 
\begin{align}\label{eq3}
[(I+2X_{i,j})E_{i,j}][(I+2X_{u,v})E_{u,v}] = [(I+2X_{u,v})E_{u,v}][(I+2X_{i,j})E_{i,j}].
\end{align}
After expanding the equation \eqref{eq3}, we have
\[[E_{i,j}E_{u,v}+2X_{i,j}E_{i,j}E_{u,v}+2E_{i,j}X_{u,v}E_{u,v}]=[E_{u,v}E_{i,j}+2X_{u,v}E_{u,v}E_{i,j}+2E_{u,v}X_{i,j}E_{i,j}].\]
When $n\geq 6$, there exists an integer $s$ which is distinct from $i,j,u$ and $v$. Computing the $(s,s)$ entry of both sides, we get
\begin{align*}
1+2x_{s,s}^{i,j}+2x_{s,s}^{u,v}&=1+2x_{s,s}^{u,v}+2x_{s,s}^{i,j}, \text{ or}\\
1+2x_{s,s}^{i,j}+2x_{s,s}^{u,v}&=-(1+2x_{s,s}^{u,v}+2x_{s,s}^{i,j}) \text{ over }\mathbb{Z}/4\mathbb{Z}.
\end{align*} 
The only possibility is the first case which implies
\begin{align}\label{eq4}
2X_{i,j}E_{i,j}E_{u,v}+2E_{i,j}X_{u,v}E_{u,v}=2X_{u,v}E_{u,v}E_{i,j}+2E_{u,v}X_{i,j}E_{i,j}\text{ over }\mathbb{Z}/4\mathbb{Z}.
\end{align}
Examining the $(u,v)$ entry of equation \eqref{eq4} gives us
\[x_{u,u}^{i,j}+x_{v,v}^{i,j}\equiv 0\mod 2\]
Taking $u,v$ in pairs and distinct from $i,j$, using \eqref{trace=1}, we can conclude that tr$(X_{i,j})\equiv 1$ mod $2$. This contradicts the fact that $X_{i,j}\in  M^0_n(\mathbb{Z}/4\mathbb{Z})$.
\end{proof}
\begin{remark}
When $n=4$, the map $\mathrm{PSL}_4(\mathbb{Z}/4\mathbb{Z}) \rightarrow \mathrm{PSL}_4(\mathbb{Z}/2\mathbb{Z})$ does split. For reasons of  space, we do not include the splitting here.
\end{remark}

\subsection{The case $p\geq 5$} \label{p=5}
Now we would like to prove Proposition \ref{prop1} by discussing the cases $p\geq 5$ and $p=3$ separately. Consider a prime order $f\in\mathrm{Aut}(F_n)$. Proposition \ref{collapse} enables us to study $f$  by analyzing a good realization $\tilde{f}$ of a based graph $(X, v_0)$. Consider a good realization $\tilde{f}:(X, v_0)\rightarrow(X, v_0)$. Recall the property of a good realization: the automorphism $\tilde{f}$ fixes all vertices of $X$ and the only edges fixed by $\tilde{f}$ are loops. The map $\tilde{f}$ acts on the edges of $(X,v_0)$ as follows:
\begin{enumerate}
\item For each vertex $v$, the loops based at $v$ are permuted. Some of them might be fixed.
\item For each pair of distinct vertices $v$ and $w$, the edges joining $v$ and $w$ are permuted. None of them are fixed.
\end{enumerate}
See Examples \ref{exR_5} and \ref{exH_5} for illustration. 

A \textit{rose with $n$ petals} $R_n$ is the graph that consists of one vertex and $n$ loops at the vertex. A \textit{hairy graph with $n$ edges} $H_n$ is the graph that consists of two vertices and $n$ edges connecting these two vertices. 
\begin{example}\label{exR_5}
\begin{figure}[htbp]
\centering
\begin{minipage}[t]{0.48\textwidth}
\centering
\includegraphics[scale=0.55]{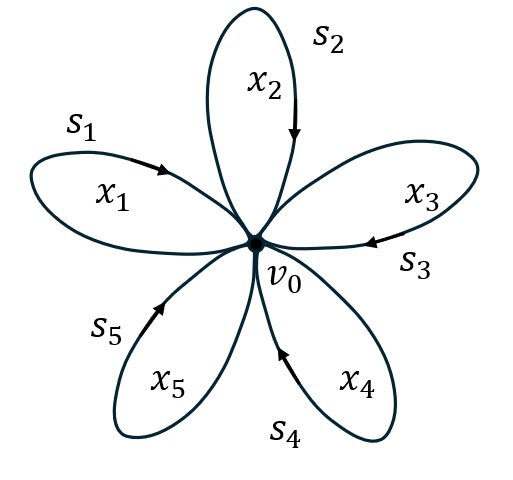}
\caption{rose with five petals $R_5$\label{R_5}}
\end{minipage}
\begin{minipage}[t]{0.48\textwidth}
\centering
\includegraphics[scale=0.55]{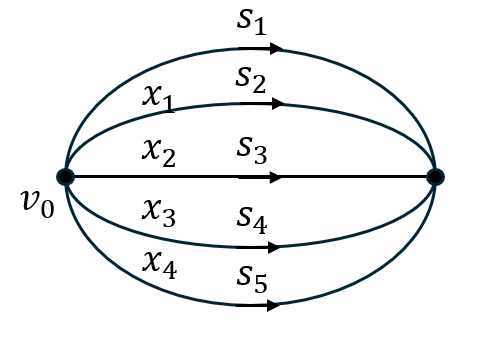}
\caption{hairy graph $H_5$\label{H_5}}
\end{minipage}
\end{figure}
In Figure \ref{R_5}, the action of $\tilde{f}$ on $R_5$ is $\tilde{f}(s_i)=s_{i+1}$ and the basis of $\pi_1(R_5,v_0)$ is 
\[x_1=s_1, \quad x_2=s_2,\quad x_3=s_3, \quad x_4=s_4,\quad x_5=s_5.\]
Thus, the action of $f$ and $f^{-1}$ on $\pi_1(R_5,v_0)$ satisfy:
\[
f(x_i)=x_{i+1},\quad f^{-1}(x_{i+1})=x_{i},\quad 1\leq i \leq 5.
\]
Consider the map $\Phi_1 = L_{x_1x_2}^{-1}fL_{x_1x_2}f^{-1}$ and the map $\Phi_2 = L_{x_1x_2}fL_{x_1x_2}^{-1}f^{-1}$. Both $\Phi_1$ and $\Phi_2$ are in  $N(f)$. By direct computation, we see that $\Phi_1 \circ \Phi_2 = L_{x_1x_3}$. Since the image of $f$ in $\mathrm{GL}_5(\mathbb{Z})$ is the matrix
\[
\begin{pmatrix}
0 & 0 & 0 & 0 & 1 \\
1 & 0 & 0 & 0 & 0 \\
0 & 1 & 0 & 0 & 0 \\
0 & 0 & 1 & 0 & 0 \\
0 & 0 & 0 & 1 & 0 \\
\end{pmatrix},
\]
which has determinant $1$, Lemma \ref{left multi} implies $N(f) = \mathrm{SAut}(F_5)$.
\end{example}
\begin{example}\label{exH_5}
In Figure \ref{H_5}, the action of $\tilde{f}$ on $H_5$ is $\tilde{f}(s_i) = s_{i+1}$ as well and the basis of $\pi_1(H_5, v_0)$ is
\[ x_1 = s_1s_2^{-1}, \quad x_2 = s_2s_3^{-1}, \quad x_3 = s_3s_4^{-1}, \quad x_4 = s_4s_5^{-1}. \]
Thus, the action of $f$ and $f^{-1}$ on $\pi_1(H_5, v_0)$ satisfy:
\begin{align*}
f:&x_i \mapsto x_{i+1}, \quad \text{for } 1\leq i\leq 3, &f^{-1}:&x_1 \mapsto x_4^{-1}x_3^{-1}x_2^{-1}x_1^{-1} \\
    &x_4 \mapsto x_4^{-1}x_3^{-1}x_2^{-1}x_1^{-1} &         &x_{i+1} \mapsto x_i,\quad \text{for } 1\leq i\leq 3.
\end{align*}
We still consider the map $\Phi_1 = L_{x_1x_2}^{-1}fL_{x_1x_2}f^{-1}$ and the map $\Phi_2 = L_{x_1x_2}fL_{x_1x_2}^{-1}f^{-1}$. Both $\Phi_1$ and $\Phi_2$ are in  $N(f)$. The composition of these two maps gives us $\Phi_1 \circ \Phi_2 = L_{x_1x_3}$. Since $\det(\bar{f})=1$, Lemma \ref{left multi} implies $N(f) = \mathrm{SAut}(F_4)$.
\end{example}

\begin{proof}[Proof of Proposition \ref{prop1} for $p\geq 5$]
Let $f$ be an order $p\geq 5$ element in $\mathrm{Aut}(F_n)$ and $\tilde{f}\in\mathrm{Aut}(X,v_0)$ be a good realization of $f$. 
\begin{case}\label{p>=5, case1}
If $X$ has only one vertex $v_0$, then $X$ is a rose $R_n$. Some of the loops in $X=R_n$ are permuted by $\tilde{f}$ and some are fixed. Assume that $pk$ loops are permuted, so $n-pk$ are fixed. We can then label the permuted loops as $s_{i,j}$ with $1 \leq i \leq p$ and $1 \leq j \leq k$ such that $\tilde{f}(s_{i,j})=s_{i+1,j}$. We can label the fixed loops as $\ell_m$ with $1\leq m\leq n-pk$ so $\tilde{f}(\ell_m) = \ell_m$. Consider the following basis of $\pi_1(X, v_0)$:
\begin{align*}
x_{i,j} &= s_{i,j} \quad 1 \leq i \leq p, 1 \leq j \leq k,\\
x_m &= \ell_m \quad 1\leq m\leq n-pk.
\end{align*}
The action of $f$ and $f^{-1}$ on $\mathrm{Aut}(F_n)\cong\pi_1(X,v_0)$ satisfy: 
\begin{align*}
f:  &x_{i,j} \mapsto x_{i+1,j} &  f^{-1}:&x_{i+1,j} \mapsto x_{i,j} \\
    &x_m\mapsto x_m & &x_m\mapsto x_m
\end{align*}
where $1 \leq i \leq p$ , $1 \leq j \leq k$ and $1\leq m\leq n-pk$.

Let $\Phi_1 = L_{x_{1,1}x_{2,1}}^{-1}fL_{x_{1,1}x_{2,1}}f^{-1}$ and $\Phi_2 = L_{x_{1,1}x_{2,1}}fL_{x_{1,1}x_{2,1}}^{-1}f^{-1}$. The maps $\Phi_1$ and $\Phi_2$ are in $N(f)$. We calculate $\Phi_1(x_{1,1})$ and $\Phi_1(x_{2,1})$ as follows:
\begin{align*}
     & x_{1,1} \xmapsto{f^{-1}} x_{p,1} \xmapsto{L_{x_{1,1}x_{2,1}}} x_{p,1} \xmapsto{ f } x_{1,1} \xmapsto{L_{x_{1,1}x_{2,1}}^{-1}} x_{2,1}^{-1}x_{1,1}, \\
     & x_{2,1} \xmapsto{f^{-1}} x_{1,1} \xmapsto{L_{x_{1,1}x_{2,1}}} x_{2,1}x_{1,1} \xmapsto{ f } x_{3,1}x_{2,1} \xmapsto{L_{x_{1,1}x_{2,1}}^{-1}} x_{3,1}x_{2,1}.
\end{align*}
The map $\Phi_1$ fixes the other basis elements in $\pi_1(X,v_0)$. Similarly, we calculate $\Phi_2(x_{1,1})$ and $\Phi_2(x_{2,1})$ as follows: 
\begin{align*}
     & x_{1,1} \xmapsto{f^{-1}} x_{p,1} \xmapsto{L_{x_{1,1}x_{2,1}}^{-1}} x_{p,1} \xmapsto{ f } x_{1,1} \xmapsto{L_{x_{1,1}x_{2,1}}} x_{2,1}x_{1,1}, \\
     & x_{2,1} \xmapsto{f^{-1}} x_{1,1} \xmapsto{L_{x_{1,1}x_{2,1}}^{-1}} x_{2,1}^{-1}x_{1,1} \xmapsto{ f } x_{3,1}^{-1}x_{2,1} \xmapsto{L_{x_{1,1}x_{2,1}}} x_{3,1}^{-1}x_{2,1}.
\end{align*}
The map $\Phi_2$ fixes the other basis elements in $\pi_1(X,v_0)$. Consider 
\[ \Phi_1 \circ \Phi_2 = L_{x_{1,1}x_{2,1}}^{-1}fL_{x_{1,1}x_{2,1}}f^{-1} \circ L_{x_{1,1}x_{2,1}}fL_{x_{1,1}x_{2,1}}^{-1}f^{-1}. \]
By direct computation, 
\begin{align*}
     & x_{1,1} \xmapsto{\Phi_2} x_{2,1}x_{1,1} \xmapsto{\Phi_1} x_{3,1}x_{1,1},\\
     & x_{2,1} \xmapsto{\Phi_2} x_{3,1}^{-1}x_{2,1} \xmapsto{\Phi_1} x_{2,1},
\end{align*}
and other basis elements are fixed. We can see that $\Phi_1 \circ \Phi_2 = L_{x_{1,1}x_{3,1}}\in N(f)$. The image of $f$ in $\mathrm{GL}_n(\mathbb{Z})$ is the direct sum $M^k\oplus I_{n-pk}$, where $I$ is the identity matrix and $M$ is the $p$ by $p$ matrix
\[
\begin{pmatrix}
0 & 0 & 0 &\ldots & 0 & 1 \\
1 & 0 & 0 &\ldots & 0 & 0 \\
0 & 1 & 0 &\ldots & 0 & 0 \\
0 & 0 & 1 &\ldots & 0 & 0 \\
0 & 0 & 0 &\ldots & 0 & 0 \\
\vdots&\vdots&\vdots&&\vdots&\vdots\\
0 & 0 & 0 &\ldots & 1 & 0 \\
\end{pmatrix}.
\]
Then $\det(\bar{f})=(-1)^{p-1}=1$. By Lemma \ref{left multi}, $N(f)=\mathrm{SAut}(F_n)$.
\end{case}
\begin{case}\label{p>=5, case2}
If $X$ has an edge that is not a loop, then $X$ contains a hairy graph $H_{pk}$ as a subgraph. By replacing $f\in \mathrm{Aut}(F_n)$ by a conjugate element, we can assume our basepoint $v_0$ lies in the subgraph $H_{pk}$ of $X$. We denote the other vertex of $H_{pk}$ to be $v_1$ and label the edges of $H_{pk}$ as $s_{i,j},(1 \leq i \leq p, 1 \leq j \leq k)$ such that $\tilde{f}(s_{i,j}) = s_{i+1,j}$ for $1\leq j\leq k $. Consider the following basis of $\pi_1(H_{pk}, v_0)$:
\begin{align*}
    x_{i,j} = s_{i,j}s_{i+1,j}^{-1} & \quad 1 \leq i \leq p-1, 1 \leq j \leq k; \\ x_{p,j} = s_{p,j}s_{1,j+1}^{-1} & \quad 1 \leq j \leq k-1.
\end{align*}
Notice that $x_{p,j}$ exists if and only if $k \geq 2$. The actions of $f|_{\pi_1(H_{pk}, v_0)}$ and $f^{-1}|_{\pi_1(H_{pk}, v_0)}$  satisfy: 
\begin{align*}
f|_{\pi_1(H_{pk}, v_0)}:  &x_{i,j} \mapsto x_{i+1,j} &  f^{-1}|_{\pi_1(H_{pk}, v_0)}:&x_{1,j} \mapsto x_{p-1,j}^{-1}x_{p-2,j}^{-1}\ldots x_{1,j}^{-1} \\
     &x_{p-1,j} \mapsto x_{p-1,j}^{-1}x_{p-2,j}^{-1}\ldots x_{1,j}^{-1} &         &x_{i+1,j} \mapsto x_{i,j} \\
    &x_{p,j} \mapsto x_{1,j}\ldots x_{p,j}x_{1,j+1} &         &x_{p,j} \mapsto x_{p-1,j}x_{p,j}x_{1,j+1}\ldots x_{p-1,j+1}
\end{align*}
where $1 \leq i \leq p-1$, $1\leq j\leq k $.

\begin{subcase}\label{A connected}
Let $A = X \setminus (\text{edges in } H_{pk})$. We first deal with the case where $A$ is connected. Let $\gamma_0$ be a path in $A$ connecting the vertices $v_1$ and $v_0$. See Figure \ref{H_10 A connected} for an illustration. We claim that
\[ \pi_1(X, v_0) = \pi_1(H_{pk}, v_0) * \pi_1(A,v_0) * \Gamma,\quad\text{where }\Gamma\cong <s_{p,1}\gamma_0>.\]
\begin{proof}
Let $A'=A\cup s_{p,1}$. Since $X=H_{pk}\cup A'$ and $H_{pk}\cap A'=s_{p,1}$ which is contractible, it follows that 
\[ \pi_1(X, v_0) = \pi_1(H_{pk}, v_0) * \pi_1(A',v_0).\]
Let $Z= s_{p,1}\cup\gamma_0$. Then, $A'=A\cup Z$ and $A\cap Z = \gamma_0$ which is contractible, so 
\[ \pi_1(A', v_0) = \pi_1(A, v_0) * \pi_1(Z,v_0)= \pi_1(A, v_0) * \mathbb{Z}.\]
Thus, we prove the claim.
\end{proof}

\begin{figure}[htbp]
    \centering
    \includegraphics[scale=0.4]{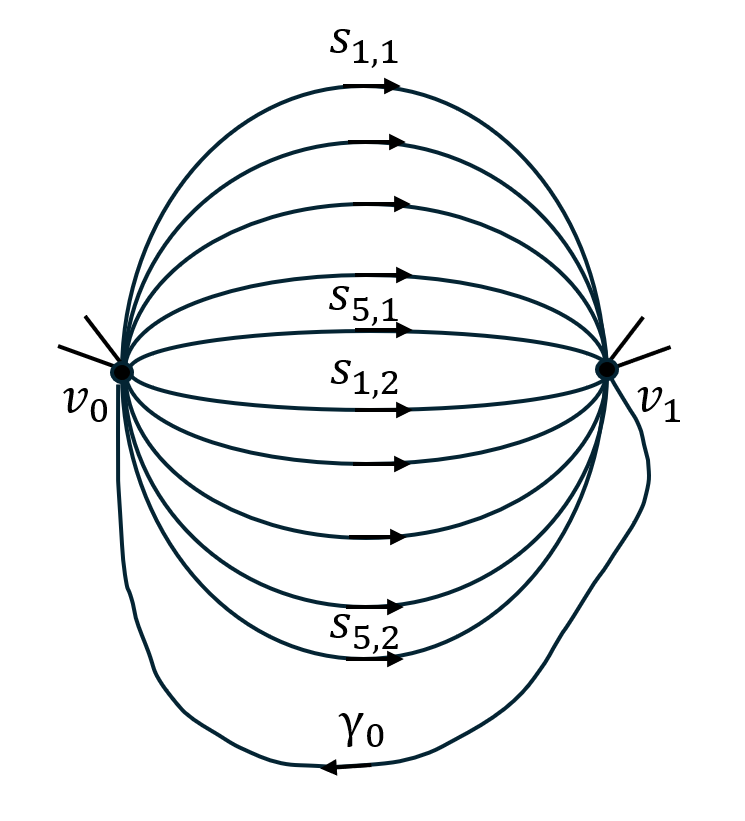}
    \caption{subgraph $H_{10}$  and subgraph $A$ is connected\label{H_10 A connected}}
\end{figure} 

Since the automorphism $\tilde{f}$ fixes all vertices,  the automorphism $f$ preserves $\pi_1(H_{pk}, v_0)$ and $\pi_1(A, v_0)$. Let $\Phi_1 = L_{x_{1,1}x_{2,1}}^{-1}fL_{x_{1,1}x_{2,1}}f^{-1}$ and $\Phi_2 = L_{x_{1,1}x_{2,1}}fL_{x_{1,1}x_{2,1}}^{-1}f^{-1}$. The maps $\Phi_1$ and $\Phi_2$ are in $N(f)$. We calculate $\Phi_1(x_{1,1})$ and $\Phi_1(x_{2,1})$ as follows:
\begin{align*}
    & x_{1,1} \xmapsto{f^{-1}} x_{p-1,1}^{-1}\ldots x_{1,1}^{-1} \xmapsto{L_{x_{1,1}x_{2,1}}} x_{p-1,1}^{-1}\ldots x_{2,1}^{-1}x_{1,1}^{-1}x_{2,1}^{-1} \xmapsto{f} x_{1,1}x_{3,1}^{-1} \xmapsto{L_{x_{1,1}x_{2,1}}^{-1}} x_{2,1}^{-1}x_{1,1}x_{3,1}^{-1}, \\
    & x_{2,1} \xmapsto{f^{-1}} x_{1,1} \xmapsto{L_{x_{1,1}x_{2,1}}} x_{2,1}x_{1,1} \xmapsto{f} x_{3,1}x_{2,1} \xmapsto{L_{x_{1,1}x_{2,1}}^{-1}} x_{3,1}x_{2,1}.
\end{align*}
The map $\Phi_1$ fixes the other basis elements in $\pi_1(H_{pk},v_0)$. Similarly, we calculate $\Phi_2(x_{1,1})$ and $\Phi_2(x_{2,1})$ as follows: 
\begin{align*}
    & x_{1,1} \xmapsto{f^{-1}} x_{p-1,1}^{-1}\ldots x_{1,1}^{-1} \xmapsto{L_{x_{1,1}x_{2,1}}^{-1}} x_{p-1,1}^{-1}\ldots x_{2,1}^{-1}x_{1,1}^{-1}x_{2,1} \xmapsto{f} x_{1,1}x_{3,1} \xmapsto{L_{x_{1,1}x_{2,1}}} x_{2,1}x_{1,1}x_{3,1}, \\
    & x_{2,1} \xmapsto{f^{-1}} x_{1,1} \xmapsto{L_{x_{1,1}x_{2,1}}^{-1}} x_{2,1}^{-1}x_{1,1} \xmapsto{f} x_{3,1}^{-1}x_{2,1} \xmapsto{L_{x_{1,1}x_{2,1}}} x_{3,1}^{-1}x_{2,1}.
\end{align*}
The map $\Phi_2$ fixes the other basis elements in $\pi_1(H_{pk},v_0)$.  Consider 
\[
\Phi_1 \circ \Phi_2 = L_{x_{1,1}x_{2,1}}^{-1}fL_{x_{1,1}x_{2,1}}f^{-1} \circ L_{x_{1,1}x_{2,1}}fL_{x_{1,1}x_{2,1}}^{-1}f^{-1}\in N(f).
\]
Then $(\Phi_1 \circ \Phi_2)|_{\pi_1(H_{pk},v_0)} =L_{x_{1,1}x_{3,1}}$ since
\begin{align*}
     & x_{1,1} \xmapsto{\Phi_2} x_{2,1}x_{1,1}x_{3,1} \xmapsto{\Phi_1} x_{3,1}x_{1,1},\\
     & x_{2,1} \xmapsto{\Phi_2} x_{3,1}^{-1}x_{2,1} \xmapsto{\Phi_1} x_{2,1},
\end{align*}
and it fixes the other basis elements of $\pi_1(H_{pk},v_0)$. The above computation is where we use that $p\geq 5$ since it relies on the fact that $f(x_{2,1})=x_{3,1}$.  Since $L_{x_1x_2}$ acts trivially on $\pi_1(A,v_0)$, so we have
\[
    (\Phi_1 \circ \Phi_2)|_{\pi_1(A,v_0)} =(ff^{-1}\circ ff^{-1})|_{\pi_1(A,v_0)} = \mathrm{Id}_{\pi_1(A,v_0)}.
\]
Since 
\[
    \tilde{f}^{-1}(s_{p,1}\gamma_0) = \tilde{f}^{-1}(s_{p,1})\tilde{f}^{-1}(\gamma_0)=s_{p-1,1}\tilde{f}^{-1}(\gamma_0)=(s_{p-1,1}s_{p,1}^{-1})(s_{p,1}\gamma_0)(\gamma_0^{-1}\tilde{f}^{-1}(\gamma_0)),
\]
it follows that on $\Gamma$, we have
\begin{align}\label{finverse on gamma}
     f^{-1}(s_{p,1}\gamma_0)=(s_{p-1,1}s_{p,1}^{-1})(s_{p,1}\gamma_0)(\gamma_0^{-1}\tilde{f}^{-1}(\gamma_0)) =  x_{p-1,1}(s_{p,1}\gamma_0)(\gamma_0^{-1}\tilde{f}^{-1}(\gamma_0)).
\end{align}
Using the equation \eqref{finverse on gamma}, 
\[f(x_{p-1,1}(s_{p,1}\gamma_0)(\gamma_0^{-1}\tilde{f}^{-1}(\gamma_0)))  = f(f^{-1}(s_{p,1}\gamma_0)) = s_{p,1}\gamma_0.\]
Then we calculate the map $\Phi_1$ and $\Phi_2$ acting on $\Gamma$:
\begin{align*}
\Phi_1|_\Gamma: s_{p,1}\gamma_0 &\xmapsto{f^{-1}} x_{p-1,1}(s_{p,1}\gamma_0)(\gamma_0^{-1}\tilde{f}^{-1}(\gamma_0))\\              &\xmapsto{L_{x_{1,1}x_{2,1}}^{-1}} x_{p-1,1}(s_{p,1}\gamma_0)(\gamma_0^{-1}\tilde{f}^{-1}(\gamma_0)) \\
&\xmapsto{f} s_{p,1}\gamma_0 \\
&\xmapsto{L_{x_{1,1}x_{2,1}}} s_{p,1}\gamma_0;
\end{align*}
\begin{align*}
\Phi_2|_\Gamma: s_{p,1}\gamma_0 &\xmapsto{f^{-1}} x_{p-1,1}(s_{p,1}\gamma_0)(\gamma_0^{-1}\tilde{f}^{-1}(\gamma_0))\\
& \xmapsto{L_{x_{1,1}x_{2,1}}}x_{p-1,1}(s_{p,1}\gamma_0)(\gamma_0^{-1}\tilde{f}^{-1}(\gamma_0))\\
&\xmapsto{f} s_{p,1}\gamma_0 \\
&\xmapsto{L_{x_{1,1}x_{2,1}}^{-1}} s_{p,1}\gamma_0.     
\end{align*}
This implies that $(\Phi_1 \circ \Phi_2)_{\Gamma} = \text{Id}_{\Gamma}$. Thus, $\Phi_1 \circ \Phi_2 = L_{x_{1,1}x_{3,1}}$ on $\pi_1(X,v_0)=F_n$. Since we have $L_{x_{1,1}x_{3,1}}=\Phi_1 \circ \Phi_2\in N(f)$, by Lemma \ref{left multi}, 
\[
    N(f) = 
    \begin{cases}
        \mathrm{SAut}(F_n) & \text{if det}(\bar{f}) = 1, \\
        \mathrm{Aut}(F_n) & \text{if det}(\bar{f}) = -1.
    \end{cases}
\]
\end{subcase}

\begin{subcase}
In this case, $A$ is disconnected. Let $A = A_1 \cup A_2$ where $A_1$ contains the vertex $v_0$ and $A_2$ contains the vertex $v_1$. We have
\[ \pi_1(X, v_0) = \pi_1(H_{pk}, v_0) * \pi_1(A_1, v_0) * \Gamma' , \]
where $\Gamma'=\{s_{p,1}\gamma s_{p,1}^{-1}|\gamma\in\pi_1(A_2, v_1)\}$. See Figure \ref{H_5 A disconnected} for an illustration. 
\begin{figure}[htbp]
    \centering
    \includegraphics[scale=0.4]{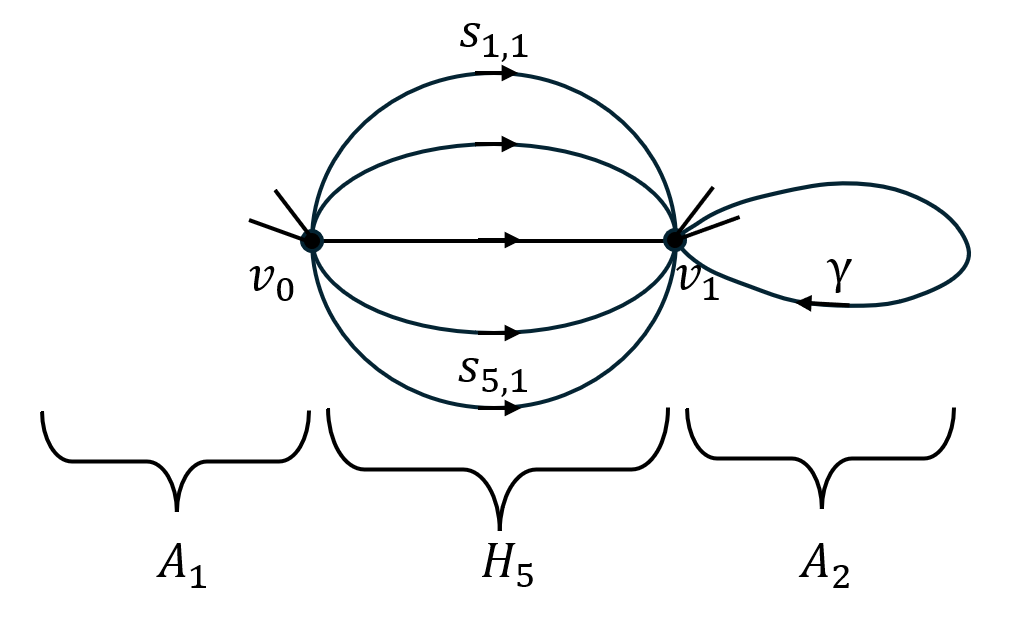}
    \caption{subgraph $H_{5}$  and subgraph $A$ is disconnected\label{H_5 A disconnected}}
\end{figure}
Since the automorphism $\tilde{f}$ fixes all vertices,  the automorphism $f$ preserves $\pi_1(H_{pk}, v_0)$ and $\pi_1(A_1, v_0)$.  Using same strategy as in the Case \ref{A connected}, we consider
\[
    \Phi_1 \circ \Phi_2 = L_{x_{1,1}x_{2,1}}^{-1}fL_{x_{1,1}x_{2,1}}f^{-1} \circ L_{x_{1,1}x_{2,1}}fL_{x_{1,1}x_{2,1}}^{-1}f^{-1}.
\]
Then
\begin{align*}
    (\Phi_1 \circ \Phi_2)|_{\pi_1(H_{pk}, v_0)} &=L_{x_{1,1}x_{3,1}},\text{ and}\\
    (\Phi_1 \circ \Phi_2)|_{\pi_1(A_1, v_0)} &=(ff^{-1}\circ ff^{-1})|_{\pi_1(A_1, v_0)} = \text{Id}_{\pi_1(A_1, v_0)}.
\end{align*}
Since
\[
    \tilde{f}^{-1}(s_{p,1}\gamma s_{p,1}^{-1}) = \tilde{f}^{-1}(s_{p,1})\tilde{f}^{-1}(\gamma)\tilde{f}^{-1}(s_{p,1}^{-1})=s_{p-1,1}\tilde{f}^{-1}(\gamma)s_{p-1,1}^{-1}
\]
it follows that on $\Gamma'$,  we have 
\begin{align}\label{finverse on gamma2}
     f^{-1}(s_{p,1}\gamma s_{p,1}^{-1})=x_{p-1,1}(s_{p,1}\tilde{f}(\gamma) s_{p,1}^{-1})x_{p-1,1}^{-1}, \quad\text{where }s_{p,1}\tilde{f}(\gamma) s_{p,1}^{-1}\in \Gamma'.
\end{align}
Using the equation  \eqref{finverse on gamma2},
\[f(x_{p-1,1}(s_{p,1}\tilde{f}(\gamma) s_{p,1}^{-1})x_{p-1,1}^{-1})  = f(f^{-1}(s_{p,1}\gamma s_{p,1}^{-1}) = s_{p,1}\gamma s_{p,1}^{-1}.\]
We calculate the map $\Phi_1$ and $\Phi_2$ acting on $\Gamma'$:
\begin{align*}
\Phi_1|_{\Gamma'}: s_{p,1}\gamma s_{p,1}^{-1}&\xmapsto{f^{-1}} x_{p-1,1}(s_{p,1}\tilde{f}(\gamma) s_{p,1}^{-1})x_{p-1,1}^{-1} \\
&\xmapsto{L_{x_{1,1}x_{2,1}}^{-1}} x_{p-1,1}(s_{p,1}\tilde{f}(\gamma) s_{p,1}^{-1})x_{p-1,1}^{-1} \\
&\xmapsto{f} s_{p,1}\gamma s_{p,1}^{-1}\\
&\xmapsto{L_{x_{1,1}x_{2,1}}} s_{p,1}\gamma s_{p,1}^{-1};\\
\Phi_2|_{\Gamma'}: s_{p,1}\gamma s_{p,1}^{-1}&\xmapsto{f^{-1}} x_{p-1,1}(s_{p,1}\tilde{f}(\gamma) s_{p,1}^{-1})x_{p-1,1}^{-1} \\
&\xmapsto{L_{x_{1,1}x_{2,1}}} x_{p-1,1}(s_{p,1}\tilde{f}(\gamma) s_{p,1}^{-1})x_{p-1,1}^{-1} \\
&\xmapsto{f} s_{p,1}\gamma s_{p,1}^{-1}\\
&\xmapsto{L_{x_{1,1}x_{2,1}}^{-1}} s_{p,1}\gamma s_{p,1}^{-1}.     
\end{align*}
This implies that $(\Phi_1 \circ \Phi_2)|_{\Gamma'} = \text{Id}_{\Gamma'}$. Hence, $\Phi_1 \circ \Phi_2 = L_{x_{1,1}x_{3,1}}$ on $\pi_1(X,v_0)=F_n$. Since we have $L_{x_{1,1}x_{3,1}}=\Phi_1 \circ \Phi_2\in N(f)$, by Lemma \ref{left multi}, 
\[
    N(f) = 
    \begin{cases}
        \mathrm{SAut}(F_n) & \text{if det}(\bar{f}) = 1, \\
        \mathrm{Aut}(F_n) & \text{if det}(\bar{f}) = -1.
    \end{cases}
\]
\end{subcase}
\end{case}
\end{proof}

\subsection{The case $p=3$}
Now we prove Proposition \ref{prop1} when $p = 3$.

\begin{proof}
Let $f$ be an order $3$ element in $\mathrm{Aut}(F_n)$ and $\tilde{f}\in\mathrm{Aut}(X,v_0)$ be a good realization of $f$. 
\setcounter{case}{0}
\begin{case}
If $X$ has only one vertex $v_0$, then the argument is identical to the proof of Case \ref{p>=5, case1} of the case  $p\geq 5$ of Proposition \ref{prop1}, which works for $p=3$.
\end{case}
\begin{case}
If $X$ has an edge that is not a loop, then $X$ contains a hairy graph $H_{3k}$ as a subgraph. By replacing $f\in \mathrm{Aut}(F_n)$ by a conjugate element, we can assume our basepoint $v_0$ lies in the subgraph $H_{3k}$ of $X$. We denote the other vertex of $H_{3k}$ by $v_1$ and label the edges of $H_{3k}$ as $s_{i,j},(1 \leq i \leq 3, 1 \leq j \leq k)$ such that $\tilde{f}(s_{i,j}) = s_{i+1,j}$ for $1\leq j\leq k $. Consider the following basis of $\pi_1(H_{3k}, v_0)$:
\begin{align*}
    x_{i,j} = s_{i,j}s_{i+1,j}^{-1} & \quad i=1,2, \quad 1 \leq j \leq k; \\ 
    x_{3,j} = s_{3,j}s_{1,j+1}^{-1} & \quad 1 \leq j \leq k-1.
\end{align*}
Notice that $x_{3,j}$ exists if and only if $k \geq 2$. The actions of $f|_{\pi_1(H_{3k}, v_0)}$ and $f^{-1}|_{\pi_1(H_{3k}, v_0)}$  satisfy: 
\begin{align*}
f|_{\pi_1(H_{3k}, v_0)}:  &x_{1,j} \mapsto x_{2,j} &  f^{-1}|_{\pi_1(H_{3k}, v_0)}:&x_{1,j} \mapsto x_{2,j}^{-1}x_{1,j}^{-1} \\
     &x_{2,j} \mapsto x_{2,j}^{-1}x_{1,j}^{-1} &         &x_{2,j} \mapsto x_{1,j} \\
    &x_{3,j} \mapsto x_{1,j}x_{2,j}x_{3,j}x_{1,j+1} &         &x_{3,j} \mapsto x_{2,j}x_{3,j}x_{1,j+1}x_{2,j+1}
\end{align*}
where $1\leq j\leq k $.

Let $A = X \setminus (\text{edges in } H_{3k})$. We first consider the situation when $n\neq 4$. 

\begin{subcase}\label{A connected2}
In this case, $A$ is connected and $n\neq 4$. Let $\gamma_0$ be a path in $A$ connecting the vertices $v_1$ and $v_0$. See Figure \ref{H_6 A connected} for an illustration. We claim that
\[ \pi_1(X, v_0) = \pi_1(H_{3k}, v_0) * \pi_1(A,v_0) * \Gamma,\quad\text{where }\Gamma\cong <s_{3,1}\gamma_0>.\]
The argument is identical to the proof in Case \ref{A connected} of the case  $p\geq 5$ of Proposition \ref{prop1}.

\begin{figure}[htbp]
    \centering
    \includegraphics[scale=0.4]{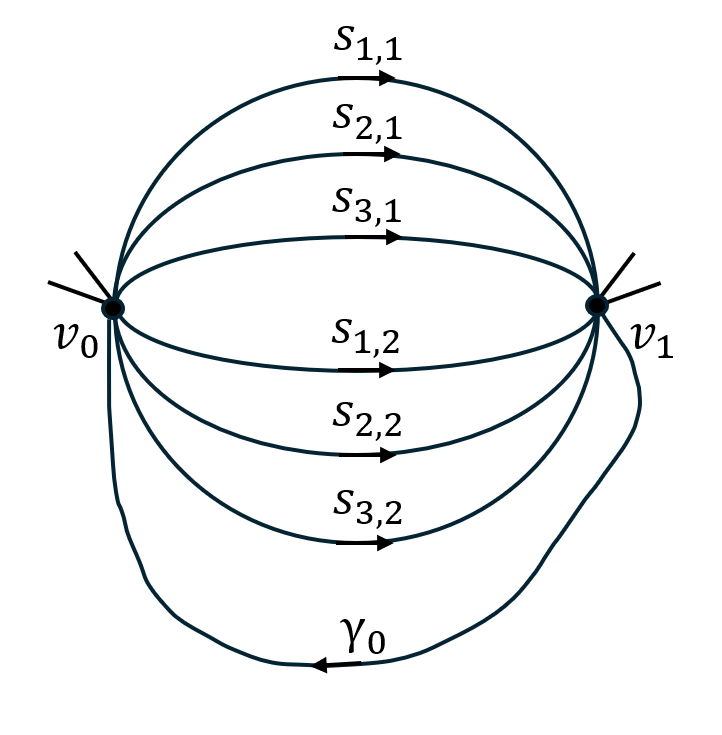}
    \caption{subgraph $H_{6}$  and subgraph $A$ is connected\label{H_6 A connected}}
\end{figure} 

Since the automorphism $\tilde{f}$ fixes all vertices,  the automorphism $f$ preserves $\pi_1(H_{3k}, v_0)$ and $\pi_1(A, v_0)$. Let $\Phi_1 = fL_{x_{1,1}x_{2,1}}f^{-1}L_{x_{1,1}x_{2,1}}^{-1}$ and $\Phi_2 = R_{x_{1,1}x_{2,1}}fR_{x_{1,1}x_{2,1}}^{-1}f^{-1}$. The maps $\Phi_1$ and $\Phi_2$ are in $N(f)$. We calculate $\Phi_1(x_{1,1})$, $\Phi_1(x_{2,1})$ and $\Phi_1(x_{3,1})$ as follows. Notice that $x_{3,1}$ only exists if $k\geq 2$.
\begin{align*}
    & x_{1,1} \xmapsto{L_{x_{1,1}x_{2,1}}^{-1}} x_{2,1}^{-1} x_{1,1} \xmapsto{f^{-1}} x_{1,1}^{-1}x_{2,1}^{-1}x_{1,1}^{-1} \xmapsto{L_{x_{1,1}x_{2,1}}} x_{1,1}^{-1}x_{2,1}^{-1}x_{2,1}^{-1} x_{1,1}^{-1}x_{2,1}^{-1}\xmapsto{f} x_{2,1}^{-1}x_{1,1}x_{2,1}x_{1,1}x_{1,1}x_{2,1}, \\
    & x_{2,1} \xmapsto{L_{x_{1,1}x_{2,1}}^{-1}} x_{2,1} \xmapsto{f^{-1}} x_{1,1} \xmapsto{L_{x_{1,1}x_{2,1}}} x_{2,1}x_{1,1} \xmapsto{f} x_{2,1}^{-1}x_{1,1}^{-1}x_{2,1}, \\
     & x_{3,1} \xmapsto{L_{x_{1,1}x_{2,1}}^{-1}} x_{3,1} \xmapsto{f^{-1}} x_{2,1}x_{3,1}x_{1,2}x_{2,2} \xmapsto{L_{x_{1,1}x_{2,1}}} x_{2,1}x_{3,1}x_{1,2}x_{2,2} \xmapsto{f} x_{3,1}. 
\end{align*}
The map $\Phi_1$ fixes the other basis elements in $\pi_1(H_{3k},v_0)$. Similarly, we calculate $\Phi_2(x_{1,1})$, $\Phi_2(x_{2,1})$ and $\Phi_2(x_{3,1})$ as follows: 
\begin{align*}
    & x_{1,1} \xmapsto{f^{-1}} x_{2,1}^{-1}x_{1,1}^{-1} \xmapsto{R_{x_{1,1}x_{2,1}}^{-1}} x_{1,1}^{-1} \xmapsto{f} x_{2,1}^{-1} \xmapsto{R_{x_{1,1}x_{2,1}}} x_{2,1}^{-1}, \\
    & x_{2,1} \xmapsto{f^{-1}} x_{1,1} \xmapsto{R_{x_{1,1}x_{2,1}}^{-1}} x_{1,1}x_{2,1}^{-1} \xmapsto{f} x_{2,1}x_{1,1}x_{2,1} \xmapsto{R_{x_{1,1}x_{2,1}}} x_{2,1}x_{1,1}x_{2,1}x_{2,1},\\
    & x_{3,1} \xmapsto{f^{-1}} x_{2,1}x_{3,1}x_{1,2}x_{2,2} \xmapsto{R_{x_{1,1}x_{2,1}}^{-1}} x_{2,1}x_{3,1}x_{1,2}x_{2,2} \xmapsto{f} x_{3,1} \xmapsto{R_{x_{1,1}x_{2,1}}} x_{3,1}.
\end{align*}
The map $\Phi_2$ fixes the other basis elements in $\pi_1(H_{3k},v_0)$.  Consider 
\[
\Phi_1 \circ \Phi_2 = fL_{x_{1,1}x_{2,1}}f^{-1}L_{x_{1,1}x_{2,1}}^{-1}\circ R_{x_{1,1}x_{2,1}}fR_{x_{1,1}x_{2,1}}^{-1}f^{-1}\in N(f).
\]
Then $(\Phi_1 \circ \Phi_2)|_{\pi_1(H_{pk},v_0)} = C_{x_{1,1}x_{2,1}}^{-1}$ since
\begin{align*}
     & x_{1,1} \xmapsto{\Phi_2} x_{2,1}^{-1} \xmapsto{\Phi_1} x_{2,1}^{-1}x_{1,1}x_{2,1},\\
     & x_{2,1} \xmapsto{\Phi_2} x_{2,1}x_{1,1}x_{2,1}x_{2,1} \xmapsto{\Phi_1} x_{2,1},
\end{align*}
and it fixes the other basis elements of $\pi_1(H_{3k},v_0)$. Since $L_{x_1x_2}$ and $R_{x_1x_2}$ act trivially on $\pi_1(A,v_0)$, so we have
\[
    (\Phi_1 \circ \Phi_2)|_{\pi_1(A,v_0)} =(ff^{-1}\circ ff^{-1})|_{\pi_1(A,v_0)} = \mathrm{Id}_{\pi_1(A,v_0)}.
\]

Since 
\[
    \tilde{f}^{-1}(s_{3,1}\gamma_0) = \tilde{f}^{-1}(s_{3,1})\tilde{f}^{-1}(\gamma_0)=s_{2,1}\tilde{f}^{-1}(\gamma_0)=(s_{2,1}s_{3,1}^{-1})(s_{3,1}\gamma_0)(\gamma_0^{-1}\tilde{f}^{-1}(\gamma_0)),
\]
it follows that on $\Gamma$, we have
\begin{align}\label{finverse on gammap=3}
     f^{-1}(s_{3,1}\gamma_0)=(s_{2,1}s_{3,1}^{-1})(s_{3,1}\gamma_0)(\gamma_0^{-1}\tilde{f}^{-1}(\gamma_0)) =  x_{2,1}(s_{3,1}\gamma_0)(\gamma_0^{-1}\tilde{f}^{-1}(\gamma_0)).
\end{align}
Using the equation \eqref{finverse on gammap=3}, 
\[f(x_{2,1}(s_{3,1}\gamma_0)(\gamma_0^{-1}\tilde{f}^{-1}(\gamma_0)))  = f(f^{-1}(s_{3,1}\gamma_0)) = s_{3,1}\gamma_0.\]
Then we calculate the map $\Phi_1$ and $\Phi_2$ acting on $\Gamma$:
\begin{align*}
\Phi_1|_\Gamma: s_{3,1}\gamma_0 &\xmapsto{L_{x_{1,1}x_{2,1}}^{-1}} s_{3,1}\gamma_0\\              &\xmapsto{f^{-1}} x_{2,1}(s_{3,1}\gamma_0)(\gamma_0^{-1}\tilde{f}^{-1}(\gamma_0)) \\
&\xmapsto{L_{x_{1,1}x_{2,1}}} x_{2,1}(s_{3,1}\gamma_0)(\gamma_0^{-1}\tilde{f}^{-1}(\gamma_0)) \\
&\xmapsto{f} s_{3,1}\gamma_0;
\end{align*}
\begin{align*}
\Phi_2|_\Gamma: s_{3,1}\gamma_0 &\xmapsto{f^{-1}} x_{2,1}(s_{3,1}\gamma_0)(\gamma_0^{-1}\tilde{f}^{-1}(\gamma_0))\\
& \xmapsto{R_{x_{1,1}x_{2,1}}^{-1}}x_{2,1}(s_{3,1}\gamma_0)(\gamma_0^{-1}\tilde{f}^{-1}(\gamma_0))\\
&\xmapsto{f} s_{3,1}\gamma_0 \\
&\xmapsto{R_{x_{1,1}x_{2,1}}} s_{3,1}\gamma_0.     
\end{align*}
This implies that $(\Phi_1 \circ \Phi_2)_{\Gamma} = \text{Id}_{\Gamma}$. Thus, $\Phi_1 \circ \Phi_2 = C_{x_{1,1}x_{2,1}}^{-1}$ on $\pi_1(X,v_0)=F_n$. Consider $\Psi=fI_{x_{1,1}}f^{-1}I_{x_{1,1}}\in N(f)$. By direct computation,
\begin{align*}
&x_{1,1} \xmapsto{I_{x_{1,1}}} x_{1,1}^{-1} \xmapsto{f^{-1}} x_{1,1}x_{2,1}  \xmapsto{I_{x_{1,1}}} x_{1,1}^{-1}x_{2,1}  \xmapsto{f} x_{2,1}^{-1} x_{2,1}^{-1} x_{1,1}^{-1},\\
&x_{2,1} \xmapsto{I_{x_{1,1}}} x_{2,1} \xmapsto{f^{-1}} x_{1,1}  \xmapsto{I_{x_{1,1}}} x_{1,1}^{-1}  \xmapsto{f} x_{2,1}^{-1}
\end{align*}
and it fixes the other generators. Then, $\Psi^2(x_{1,1})=x_{2,1}x_{2,1}x_{1,1}x_{2,1}x_{2,1}$ and it fixes the other generators. For distinct $1\leq i,j\leq n$, let $e_{i,j}$ denote the elementary matrix. Then, when raised to the power $\ell$, the matrix $e_{i,j}^\ell$ has the $(i,j)$ entry equal to $\ell$, diagonal entries 1, and all other entries 0.  Bass–Milnor–Serre \cite{BMS} proved that for $n \geq 3$ and any $l \geq 2$, the group $\Gamma_n(\ell)$ is normally generated by $e_{i,j}^\ell$. Given this foundation, we consider the specific case where $\ell = 4$. Since that the image of $\Psi^2$ in $\mathrm{SL}_n(\mathbb{Z})$ is $e_{2,1}^4$, we conclude that it normally generates $\Gamma_n(4)$. Therefore, $\Gamma_n(4)\subset\overline{N(f)}$ . Additionally, consider $\Phi_1\in N(f)$. The automorphism $\Phi_1$ maps $x_{2,1}$ to $x_{2,1}^{-1}x_{1,1}^{-1}x_{2,1}$ so that the image of $\Phi_1$ in $\mathrm{SL}_n(\mathbb{Z}/2\mathbb{Z})$ is nontrivial. By Lemma \ref{level 4}, 
\[
    N(f) = 
    \begin{cases}
        \mathrm{SAut}(F_n) & \text{if det}(\bar{f}) = 1, \\
        \mathrm{Aut}(F_n) & \text{if det}(\bar{f}) = -1.
    \end{cases}
\]
\end{subcase}

\begin{subcase}
In this case, $A$ is disconnected and $n\neq 4$. Let $A = A_1 \cup A_2$ where $A_1$ contains the vertex $v_0$ and $A_2$ contains the vertex $v_1$. We have
\[ \pi_1(X, v_0) = \pi_1(H_{3k}, v_0) * \pi_1(A_1, v_0) * \Gamma', \]
where $\Gamma'=\{s_{3,1}\gamma s_{3,1}^{-1}|\gamma\in\pi_1(A_2, v_1)\}$. See Figure \ref{H_3 A disconnected} for an illustration. 
\begin{figure}[htbp]
    \centering
    \includegraphics[scale=0.4]{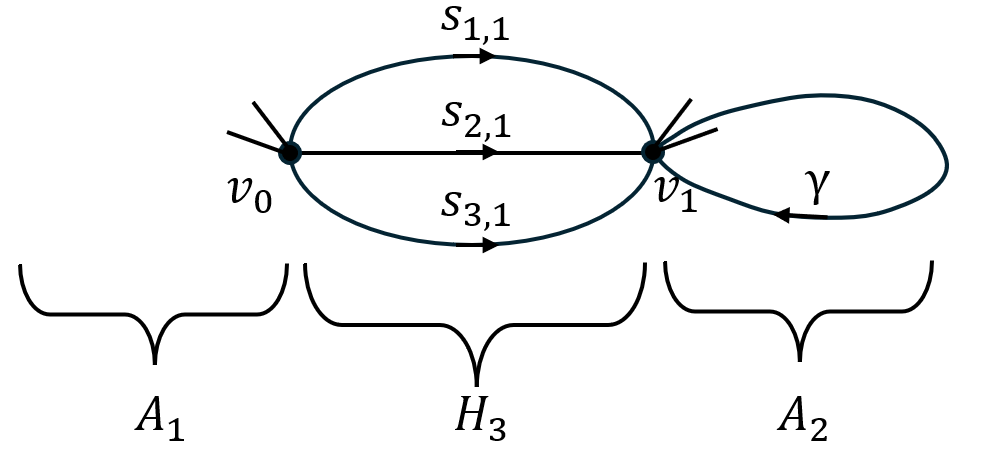}
    \caption{subgraph $H_{3}$  and subgraph $A$ is disconnected\label{H_3 A disconnected}}
\end{figure} 
Since the automorphism $\tilde{f}$ fixes all vertices,  the automorphism $f$ preserves $\pi_1(H_{3k}, v_0)$ and $\pi_1(A_1, v_0)$.  Using same strategy as in the Case \ref{A connected}, we consider
\[
    \Phi_1 \circ \Phi_2 = fL_{x_{1,1}x_{2,1}}f^{-1}L_{x_{1,1}x_{2,1}}^{-1}\circ R_{x_{1,1}x_{2,1}}fR_{x_{1,1}x_{2,1}}^{-1}f^{-1}\in N(f).
\]
Then
\begin{align*}
    (\Phi_1 \circ \Phi_2)|_{\pi_1(H_{3k}, v_0)} &=C_{x_{1,1}x_{2,1}}^{-1},\text{ and}\\
    (\Phi_1 \circ \Phi_2)|_{\pi_1(A_1, v_0)} &=(ff^{-1}\circ ff^{-1})|_{\pi_1(A_1, v_0)} = \text{Id}_{\pi_1(A_1, v_0)}.
\end{align*}
Since
\[
    \tilde{f}^{-1}(s_{3,1}\gamma s_{3,1}^{-1}) = \tilde{f}^{-1}(s_{3,1})\tilde{f}^{-1}(\gamma)\tilde{f}^{-1}(s_{3,1}^{-1})=s_{2,1}\tilde{f}^{-1}(\gamma)s_{2,1}^{-1}
\]
it follows that on $\Gamma'$,  we have 
\begin{align}\label{finverse on gammap=32}
     f^{-1}(s_{3,1}\gamma s_{3,1}^{-1})=x_{2,1}(s_{3,1}\tilde{f}(\gamma) s_{3,1}^{-1})x_{2,1}^{-1}, \quad\text{where }s_{3,1}\tilde{f}(\gamma) s_{3,1}^{-1}\in \Gamma'.
\end{align}
Using the equation  \eqref{finverse on gammap=32},
\[f(x_{2,1}(s_{3,1}\tilde{f}(\gamma) s_{3,1}^{-1})x_{2,1}^{-1})  = f(f^{-1}(s_{3,1}\gamma s_{3,1}^{-1}) = s_{3,1}\gamma s_{3,1}^{-1}.\]
We calculate the map $\Phi_1$ and $\Phi_2$ acting on $\Gamma'$:
\begin{align*}
\Phi_1|_{\Gamma'}: s_{3,1}\gamma s_{3,1}^{-1}&\xmapsto{L_{x_{1,1}x_{2,1}}^{-1}}  s_{3,1}\gamma s_{3,1}^{-1} \\
&\xmapsto{f^{-1}} x_{2,1}(s_{3,1}\tilde{f}(\gamma) s_{3,1}^{-1})x_{2,1}^{-1} \\
&\xmapsto{L_{x_{1,1}x_{2,1}}} x_{2,1}(s_{3,1}\tilde{f}(\gamma) s_{3,1}^{-1})x_{2,1}^{-1}\\
&\xmapsto{f} s_{3,1}\gamma s_{3,1}^{-1};\\
\Phi_2|_{\Gamma'}: s_{3,1}\gamma s_{3,1}^{-1}&\xmapsto{f^{-1}} x_{2,1}(s_{3,1}\tilde{f}(\gamma) s_{3,1}^{-1})x_{2,1}^{-1} \\
&\xmapsto{R_{x_{1,1}x_{2,1}}^{-1}} x_{2,1}(s_{3,1}\tilde{f}(\gamma) s_{3,1}^{-1})x_{2,1}^{-1} \\
&\xmapsto{f} s_{3,1}\gamma s_{3,1}^{-1}\\
&\xmapsto{R_{x_{1,1}x_{2,1}}} s_{3,1}\gamma s_{3,1}^{-1}.     
\end{align*}
This implies that $(\Phi_1 \circ \Phi_2)|_{\Gamma'} = \text{Id}_{\Gamma'}$. Hence, $\Phi_1 \circ \Phi_2 = C_{x_{1,1}x_{2,1}}^{-1}$ on $\pi_1(X,v_0)=F_n$. The proof then follows as the argument in Case \ref{A connected2}.
\end{subcase}

When $n=4$, there are $3$ possible graphs with at least two vertices, namely those in Figure \ref{n=4case1}, \ref{n=4case2} and \ref{n=4case3}. We deal with these cases separately. By replacing $f\in \mathrm{Aut}(F_n)$ by a conjugate element, we can assume our basepoint $v_0$ lies as in the figures.
\begin{figure}
\begin{minipage}[t]{0.48\textwidth}
\centering
\includegraphics[scale=0.55]{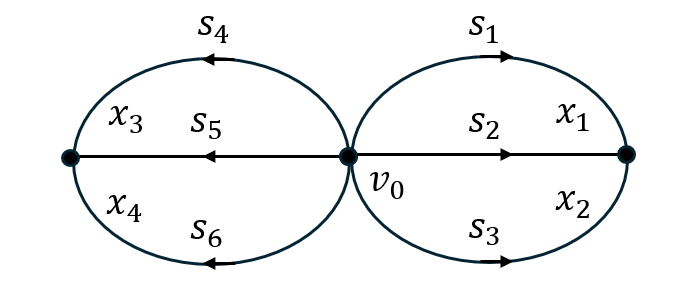}
\caption{case $1$, graph $X_1$ does not contain loops \label{n=4case1}}
\end{minipage}
\begin{minipage}[t]{0.48\textwidth}
\centering
\includegraphics[scale=0.55]{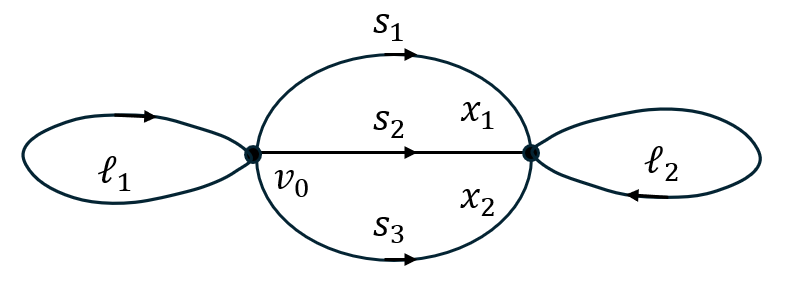}
\caption{case $2$, graph $X_2$ contains two loops attached at different vertices\label{n=4case2}}
\end{minipage}
\begin{minipage}[t]{0.48\textwidth}
\centering
\includegraphics[scale=0.55]{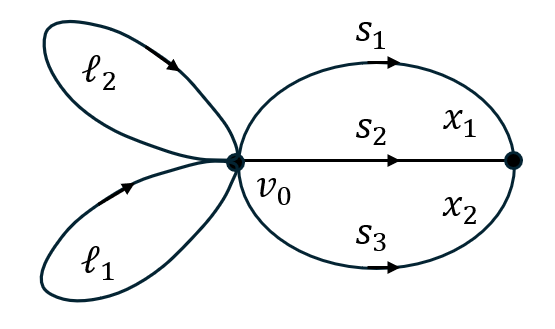}
\caption{case $3$, graph $X_3$ contains two loops attached at the same vertex\label{n=4case3}}
\end{minipage}
\end{figure}
\begin{subcase}\label{n=4case1proof}
In this case, $n=4$ and let $X_1$ be the graph in Figure \ref{n=4case1}. This graph has oriented edges $s_1,\ldots,s_6$. Let  $\tilde{f}_1\in\mathrm{Aut}(X_1, v_0)$ be the automorphism that on these edges satisfies
\[ \tilde{f}_1(s_i)=s_{i+1},\quad 1\leq i\leq 6.\]
Consider the following basis of $\pi_1(X_1,v_0)$:
\[x_1=s_1s_2^{-1}, \quad x_2=s_2s_3^{-1},\quad x_3=s_3s_4^{-1}, \quad x_4=s_4s_5^{-1}.\]
The map $\tilde{f}_{1*}$ on $\pi_1(X_1,v_0)$ realizes the map $f_1\in\mathrm{Aut}(F_4)$ where
\begin{align*}
f_1: &x_1 \mapsto x_2               &f_1^{-1}: &x_1 \mapsto  x_2^{-1}x_1^{-1} \\
     &x_2 \mapsto x_2^{-1}x_1^{-1}  &   &x_2 \mapsto x_1 \\
     &x_3 \mapsto x_4               &   &x_3 \mapsto x_3^{-1}x_4^{-1} \\
     &x_4 \mapsto x_3^{-1}x_4^{-1}               &   &x_4 \mapsto x_3. 
\end{align*}
We aim to apply Proposition \ref{level 3 prop} to compute $N(f_1)$. To meet the three requirements specified in the proposition, we need to find a conjugation $C_{x_ix_j}\in N(f_1)$ for distinct $1\leq i,j\leq n$, a matrix $e_{k,r}^3\in\overline{N(f_1)}$ for distinct $1\leq k,r\leq n$, and an element $\Phi\not\equiv I$ mod $ 3$ in $\overline{N(f_1)}$. 

Let $\Phi_1 = f_1L_{x_{1}x_{2}}f^{-1}_1L_{x_{1}x_{2}}^{-1}$ and $\Phi_2 = R_{x_{1}x_{2}}f_1R_{x_{1}x_{2}}^{-1}f_1^{-1}$. The maps $\Phi_1$ and $\Phi_2$ are in $N(f_1)$. We calculate $\Phi_1$ acting on the basis of $\pi_1(X_1,v_0)$ as follows.
\begin{align*}
    & x_1 \xmapsto{L_{x_1x_2}^{-1}} x_2^{-1} x_1 \xmapsto{f_1^{-1}} x_1^{-1}x_2^{-1}x_1^{-1} \xmapsto{L_{x_1x_2}} x_1^{-1}x_2^{-1}x_2^{-1} x_1^{-1}x_2^{-1}\xmapsto{f_1} x_2^{-1}x_1x_2x_1x_1x_2, \\
    & x_2 \xmapsto{L_{x_1x_2}^{-1}} x_2 \xmapsto{f_1^{-1}} x_1 \xmapsto{L_{x_1x_2}} x_2x_1 \xmapsto{f_1} x_2^{-1}x_1^{-1}x_2,
\end{align*}
and it fixes $x_3$ and $x_4$. Similarly, we calculate $\Phi_2(x_1)$ and $\Phi_2(x_2)$ as follows: 
\begin{align*}
    & x_1 \xmapsto{f_1^{-1}} x_2^{-1}x_1^{-1} \xmapsto{R_{x_1x_2}^{-1}} x_1^{-1} \xmapsto{f_1} x_2^{-1} \xmapsto{R_{x_1x_2}} x_2^{-1}, \\
    & x_2 \xmapsto{f_1^{-1}} x_1 \xmapsto{R_{x_1x_2}^{-1}} x_1x_2^{-1} \xmapsto{f_1} x_2x_1x_2 \xmapsto{R_{x_1x_2}} x_2x_1x_2x_2.
\end{align*}
The map $\Phi_2$ fixes the other basis elements in $\pi_1(X_1,v_0)$.  Consider 
\[
\Phi_1 \circ \Phi_2 = f_1L_{x_1x_2}f_1^{-1}L_{x_1x_2}^{-1}\circ R_{x_1x_2}f_1R_{x_1x_2}^{-1}f_1^{-1}\in N(f_1).
\]
Then $\Phi_1 \circ \Phi_2 = C_{x_1x_2}^{-1}$ since
\begin{align*}
     & x_1 \xmapsto{\Phi_2} x_2^{-1} \xmapsto{\Phi_1} x_2^{-1}x_1x_2,\\
     & x_2 \xmapsto{\Phi_2} x_2x_1x_2x_2 \xmapsto{\Phi_1} x_2,
\end{align*}
and $\Phi_1 \circ \Phi_2 $ fixes the other basis elements of $\pi_1(X_1,v_0)$. Therefore, $C_{x_1x_2}\in N(f_1)$. Now, for distinct $1\leq k,r\leq n$, we would like to find a matrix $e_{k,r}^3\in N(f_1)$. Consider the map \[\Psi=(I_{x_2}f_1I_{x_2}f_1^{-1})(R_{x_{1}x_{2}}f_1R_{x_{1}x_{2}}^{-1}f_1^{-1})\in N(f_1).\]
It acts on the basis elements of $\pi_1(X_1,v_0)$ as follows:
\begin{align*}
     x_1 &\xmapsto{f_1^{-1}} x_2^{-1} x_1^{-1} \xmapsto{R_{x_1x_2}^{-1}} x_1^{-1} \xmapsto{f_1}x_2^{-1}\xmapsto{R_{x_1x_2}} x_2^{-1}\\
    &\xmapsto{f_1^{-1}} x_1^{-1} \xmapsto{I_{x_2}} x_1^{-1} \xmapsto{f_1}x_2^{-1}\xmapsto{I_{x_2}} x_2, \\
     x_2&\xmapsto{f_1^{-1}} x_1 \xmapsto{R_{x_1x_2}^{-1}} x_1x_2^{-1} \xmapsto{f_1}x_2x_1x_2\xmapsto{R_{x_1x_2}} x_2x_1x_2x_2\\
    &\xmapsto{f_1^{-1}} x_1x_2^{-1}x_1 \xmapsto{I_{x_2}} x_1x_2x_1 \xmapsto{f_1}x_1^{-1}x_2\xmapsto{I_{x_2}} x_1^{-1}x_2^{-1}, 
\end{align*}
and it fixes $x_3$ and $x_4$. Thus, we have
\begin{align*}
\Psi: &x_1 \mapsto x_2               &\Psi^{-1}: &x_1 \mapsto  x_1^{-1}x_2^{-1} \\
     &x_2 \mapsto x_1^{-1}x_2^{-1}  &   &x_2 \mapsto x_1 \\
     &x_3 \mapsto x_3               &   &x_3 \mapsto x_3 \\
     &x_4 \mapsto x_4               &   &x_4 \mapsto x_4 
\end{align*}
Now we consider the map
\[(\Psi L_{x_{2}x_{3}}^{-1}\Psi^{-1}L_{x_{2}x_{3}})(\Psi L_{x_{1}x_{3}}^{-1}\Psi^{-1} L_{x_{1}x_{3}})\in N(f_1).\]
It acts on the basis elements as follows:
\begin{align*}
     x_1 &\xmapsto{L_{x_1x_3}} x_3 x_1 \xmapsto{\Psi^{-1}} x_3x_1^{-1}x_2^{-1} \xmapsto{L_{x_1x_3}^{-1}}x_3x_1^{-1}x_3x_2^{-1}\xmapsto{\Psi} x_3x_2^{-1}x_3x_2x_1\\
    &\xmapsto{L_{x_2x_3}} x_3x_2^{-1}x_3x_2x_1\xmapsto{\Psi^{-1}} x_3x_1^{-1}x_3x_2^{-1}  \xmapsto{L_{x_2x_3}^{-1}}x_3x_1^{-1}x_3x_2^{-1}x_3\xmapsto{\Psi} x_3x_2^{-1}x_3x_2x_1x_3, \\
     x_2 &\xmapsto{L_{x_1x_3}} x_2 \xmapsto{\Psi^{-1}} x_1 \xmapsto{L_{x_1x_3}^{-1}}x_3^{-1}x_1\xmapsto{\Psi} x_3^{-1}x_2\\
    &\xmapsto{L_{x_2x_3}} x_2\xmapsto{\Psi^{-1}} x_1  \xmapsto{L_{x_2x_3}^{-1}}x_1\xmapsto{\Psi} x_2, 
\end{align*}
and it fixes $x_3$ and $x_4$. Therefore, the image of this map  in $\mathrm{GL}_4(\mathbb{Z})$ is 
\[
\begin{pmatrix}
1 & 0 & 0 & 0  \\
0 & 1 & 0 & 0  \\
3 & 0 & 1 & 0 \\
0 & 0 & 0 & 1 
\end{pmatrix}= e_{3,1}^3.
\]
Moreover, the matrix $\bar{f_1} \not\equiv I$ mod $ 3$. Using Proposition \ref{level 3 prop}, we conclude that $N(f_1)=\mathrm{SAut}(F_4)$.
\end{subcase}
\begin{subcase} \label{n=4case2proof}
In this case, $n=4$ and let $X_2$ be the graph in Figure \ref{n=4case2}. This graph has oriented edges $\ell_1, \ell_2$ and $s_1,\ldots,s_3$. Let  $\tilde{f}_2\in\mathrm{Aut}(X_2, v_0)$ be the automorphism that on these edges satisfies
\begin{align*}
 \tilde{f}_2(\ell_m)=\ell_m,\quad m = 1,2\\
 \quad \tilde{f}_2(s_i)=s_{i+1}, \quad 1\leq i\leq 3.
\end{align*}
Consider the following basis of $\pi_1(X_2,v_0)$:
\[x_1=s_1s_2^{-1}, \quad x_2=s_2s_3^{-1},\quad x_3=\ell_1, \quad x_4=s_3\ell_2 s_3^{-1}.\]
The map $\tilde{f}_{2*}$ on $\pi_1(X_2,v_0)$ realizes the map $f_2\in\mathrm{Aut}(F_4)$ where
\begin{align*}
f_2: &x_1 \mapsto x_2               &f_2^{-1}: &x_1 \mapsto  x_2^{-1}x_1^{-1} \\
     &x_2 \mapsto x_2^{-1}x_1^{-1}  &   &x_2 \mapsto x_1 \\
     &x_3 \mapsto x_3               &   &x_3 \mapsto x_3 \\
     &x_4 \mapsto x_1x_2x_4x_2^{-1}x_1^{-1}               &   &x_4 \mapsto x_2x_4x_2^{-1} 
\end{align*}
We aim to apply Proposition \ref{level 3 prop} to compute $N(f_2)$. To meet the three requirements specified in the proposition, we need to find a conjugation $C_{x_ix_j}\in N(f_2)$ for distinct $1\leq i,j\leq n$, a matrix $e_{k,r}^3\in\overline{N(f_2)}$ for distinct $1\leq k,r\leq n$, and an element $\Phi\not\equiv I$ mod $ 3$ in $\overline{N(f_2)}$. 

Let $\Phi_1 = f_2L_{x_{1}x_{2}}f^{-1}_2L_{x_{1}x_{2}}^{-1}$ and $\Phi_2 = R_{x_{1}x_{2}}f_2R_{x_{1}x_{2}}^{-1}f_2^{-1}$. The maps $\Phi_1$ and $\Phi_2$ are in $N(f_2)$. We calculate $\Phi_1$ acting on the basis of $\pi_1(X_2,v_0)$ as follows.
\begin{align*}
    & x_1 \xmapsto{L_{x_1x_2}^{-1}} x_2^{-1} x_1 \xmapsto{f_2^{-1}} x_1^{-1}x_2^{-1}x_1^{-1} \xmapsto{L_{x_1x_2}} x_1^{-1}x_2^{-1}x_2^{-1} x_1^{-1}x_2^{-1}\xmapsto{f_2} x_2^{-1}x_1x_2x_1x_1x_2, \\
    & x_2 \xmapsto{L_{x_1x_2}^{-1}} x_2 \xmapsto{f_2^{-1}} x_1 \xmapsto{L_{x_1x_2}} x_2x_1 \xmapsto{f_2} x_2^{-1}x_1^{-1}x_2,
\end{align*}
and it fixes $x_3$ and $x_4$. Similarly, we calculate $\Phi_2(x_1)$ and $\Phi_2(x_2)$ as follows: 
\begin{align*}
    & x_1 \xmapsto{f_2^{-1}} x_2^{-1}x_1^{-1} \xmapsto{R_{x_1x_2}^{-1}} x_1^{-1} \xmapsto{f_2} x_2^{-1} \xmapsto{R_{x_1x_2}} x_2^{-1}, \\
    & x_2 \xmapsto{f_2^{-1}} x_1 \xmapsto{R_{x_1x_2}^{-1}} x_1x_2^{-1} \xmapsto{f_2} x_2x_1x_2 \xmapsto{R_{x_1x_2}} x_2x_1x_2x_2.
\end{align*}
The map $\Phi_2$ fixes the other basis elements in $\pi_1(X_2,v_0)$.  Consider 
\[
\Phi_1 \circ \Phi_2 = f_2L_{x_1x_2}f_2^{-1}L_{x_1x_2}^{-1}\circ R_{x_1x_2}f_2R_{x_1x_2}^{-1}f_2^{-1}\in N(f_2).
\]
Then $\Phi_1 \circ \Phi_2 = C_{x_1x_2}^{-1}$ since
\begin{align*}
     & x_1 \xmapsto{\Phi_2} x_2^{-1} \xmapsto{\Phi_1} x_2^{-1}x_1x_2,\\
     & x_2 \xmapsto{\Phi_2} x_2x_1x_2x_2 \xmapsto{\Phi_1} x_2,
\end{align*}
and $\Phi_1 \circ \Phi_2 $ fixes the other basis elements of $\pi_1(X_2,v_0)$. Therefore, $C_{x_1x_2}\in N(f_2)$. Moreover, the image of  $f_2$ in $\mathrm{GL}_4(\mathbb{Z})$ is the same as the image of the map $\Psi$ from Case \ref{n=4case1proof}. Thus, the same calculation as we did in Case \ref{n=4case1proof} shows that the image of
\[(f_2 L_{x_{2}x_{3}}^{-1}f_2^{-1}L_{x_{2}x_{3}})(f_2 L_{x_{1}x_{3}}^{-1}f_2^{-1} L_{x_{1}x_{3}})\in N(f_2)\]
 in $\mathrm{GL}_4(\mathbb{Z})$ is
\[
\begin{pmatrix}
1 & 0 & 0 & 0  \\
0 & 1 & 0 & 0  \\
3 & 0 & 1 & 0 \\
0 & 0 & 0 & 1 
\end{pmatrix}= e_{3,1}^3.
\]
Moreover, the matrix $\bar{f_2} \not\equiv I$ mod $ 3$. Using Proposition \ref{level 3 prop}, we conclude that $N(f_2)=\mathrm{SAut}(F_4)$.
\end{subcase}
\begin{subcase}
In this case, $n=4$ and let $X_3$ be the graph in Figure \ref{n=4case3}. This graph has oriented edges $\ell_1, \ell_2$ and $s_1,\ldots,s_3$. Let  $\tilde{f}_3\in\mathrm{Aut}(X_3, v_0)$ be the automorphism that on these edges satisfies
\begin{align*}
 \tilde{f}_3(\ell_m)=\ell_m,\quad m = 1,2\\
 \quad \tilde{f}_3(s_i)=s_{i+1}, \quad 1\leq i\leq 3.
\end{align*}
Consider the following basis of $\pi_1(X_3,v_0)$:
\[x_1=s_1s_2^{-1}, \quad x_2=s_2s_3^{-1},\quad x_3=\ell_1, \quad x_4=\ell_2 .\]
The map $\tilde{f}_{3*}$ on $\pi_1(X_3,v_0)$ realizes the map $f_3\in\mathrm{Aut}(F_4)$ where
\begin{align*}
f_3: &x_1 \mapsto x_2               &f_3^{-1}: &x_1 \mapsto  x_2^{-1}x_1^{-1} \\
     &x_2 \mapsto x_2^{-1}x_1^{-1}  &   &x_2 \mapsto x_1 \\
     &x_3 \mapsto x_3               &   &x_3 \mapsto x_3 \\
     &x_4 \mapsto x_4              &   &x_4 \mapsto x_4
\end{align*}
We aim to apply Proposition \ref{level 3 prop} to compute $N(f_3)$. To meet the three requirements specified in the proposition, we need to find a conjugation $C_{x_ix_j}\in N(f_3)$ for distinct $1\leq i,j\leq n$, a matrix $e_{k,r}^3\in\overline{N(f_3)}$ for distinct $1\leq k,r\leq n$, and an element $\Phi\not\equiv I$ mod $ 3$ in $\overline{N(f_3)}$. 

Let $\Phi_1 = f_3L_{x_{1}x_{2}}f^{-1}_3L_{x_{1}x_{2}}^{-1}$ and $\Phi_2 = R_{x_{1}x_{2}}f_3R_{x_{1}x_{2}}^{-1}f_3^{-1}$. The maps $\Phi_1$ and $\Phi_2$ are in $N(f_3)$. We calculate $\Phi_1$ acting on the basis of $\pi_1(X_3,v_0)$ as follows.
\begin{align*}
    & x_1 \xmapsto{L_{x_1x_2}^{-1}} x_2^{-1} x_1 \xmapsto{f_3^{-1}} x_1^{-1}x_2^{-1}x_1^{-1} \xmapsto{L_{x_1x_2}} x_1^{-1}x_2^{-1}x_2^{-1} x_1^{-1}x_2^{-1}\xmapsto{f_3} x_2^{-1}x_1x_2x_1x_1x_2, \\
    & x_2 \xmapsto{L_{x_1x_2}^{-1}} x_2 \xmapsto{f_3^{-1}} x_1 \xmapsto{L_{x_1x_2}} x_2x_1 \xmapsto{f_3} x_2^{-1}x_1^{-1}x_2,
\end{align*}
and it fixes $x_3$ and $x_4$. Similarly, we calculate $\Phi_2(x_1)$ and $\Phi_2(x_2)$ as follows: 
\begin{align*}
    & x_1 \xmapsto{f_3^{-1}} x_2^{-1}x_1^{-1} \xmapsto{R_{x_1x_2}^{-1}} x_1^{-1} \xmapsto{f_3} x_2^{-1} \xmapsto{R_{x_1x_2}} x_2^{-1}, \\
    & x_2 \xmapsto{f_3^{-1}} x_1 \xmapsto{R_{x_1x_2}^{-1}} x_1x_2^{-1} \xmapsto{f_3} x_2x_1x_2 \xmapsto{R_{x_1x_2}} x_2x_1x_2x_2.
\end{align*}
The map $\Phi_2$ fixes the other basis elements in $\pi_1(X_3,v_0)$.  Consider 
\[
\Phi_1 \circ \Phi_2 = f_3L_{x_1x_2}f_3^{-1}L_{x_1x_2}^{-1}\circ R_{x_1x_2}f_3R_{x_1x_2}^{-1}f_3^{-1}\in N(f_3).
\]
Then $\Phi_1 \circ \Phi_2 = C_{x_1x_2}^{-1}$ since
\begin{align*}
     & x_1 \xmapsto{\Phi_2} x_2^{-1} \xmapsto{\Phi_1} x_2^{-1}x_1x_2,\\
     & x_2 \xmapsto{\Phi_2} x_2x_1x_2x_2 \xmapsto{\Phi_1} x_2,
\end{align*}
and $\Phi_1 \circ \Phi_2 $ fixes the other basis elements of $\pi_1(X_3,v_0)$. Therefore, $C_{x_1x_2}\in N(f_3)$. Moreover, the image of  $f_3$ in $\mathrm{GL}_4(\mathbb{Z})$ is the same as the image of the map $\Psi$ from Case \ref{n=4case1proof}. Thus, the same calculation as we did in Case \ref{n=4case1proof} shows that the image of
\[(f_3 L_{x_{2}x_{3}}^{-1}f_3^{-1}L_{x_{2}x_{3}})(f_3 L_{x_{1}x_{3}}^{-1}f_3^{-1} L_{x_{1}x_{3}})\in N(f_3)\]
 in $\mathrm{GL}_4(\mathbb{Z})$ is
\[
\begin{pmatrix}
1 & 0 & 0 & 0  \\
0 & 1 & 0 & 0  \\
3 & 0 & 1 & 0 \\
0 & 0 & 0 & 1 
\end{pmatrix}= e_{3,1}^3.
\]
Moreover, the matrix $\bar{f_3} \not\equiv I$ mod $ 3$. Using Proposition \ref{level 3 prop}, we conclude that $N(f_3)=\mathrm{SAut}(F_4)$.
\end{subcase}
Thus, we have proved Proposition \ref{prop1}.
\end{case}
\end{proof}

\section{Proof of the $\mathrm{Out}(F_n)$ case}
In this section, we aim to prove Proposition \ref{prop2}. We begin by considering a prime order element $f\in\mathrm{Out}(F_n)$. According to Proposition \ref{prop1}, if $f$ can be lifted to a finite order element in $\mathrm{Aut}(F_n)$, then the proof of Proposition \ref{prop2} follows immediately. However, if such a lift does not exist, we must explore alternative properties of $f$. Previously, we used $\tilde{f}\in \mathrm{Aut}(X,x_0)$ to analyze  $f\in\mathrm{Aut}(F_n)$ where $X$ is a graph with $\pi_1(X,x_0)=F_n$. Here, we  will employ a similar approach by using elements in $\mathrm{Aut}(X)$  to study  $f\in\mathrm{Out}(F_n)$. Before applying this strategy, we need to establish the connection between $\mathrm{Aut}(X)$ and $\mathrm{Out}(\pi_1(X,x_0))$.

Let $G$ be a group. Given an element $g\in \mathrm{Aut}(G)$, we use $[g]$ to denote the corresponding element in $\mathrm{Out}(G)$.
\begin{prop}\label{relation}
Let $X$ be a path-connected space and $x_0\in X$. There exists a group homomorphism
\begin{align*}
    \phi: \mathrm{Homeo}(X)&\rightarrow\mathrm{Out}(\pi_1(X,x_0))\\
    f&\mapsto [\phi_{\delta,f}]
\end{align*}
where $\delta$ is any path from $x_0$ to $f(x_0)$, and $\phi_{\delta,f}(\gamma)=\delta f(\gamma)\delta^{-1}$ for any loop $\gamma\in \pi_1(X,x_0)$.
\end{prop}
\begin{proof}
First, we would like to prove that $\phi$ is well-defined. Suppose $\delta_1$ and $\delta_2$ are two paths from $x_0$ to $f(x_0)$. Then for any loop $\gamma$ based at $x_0$:
\begin{align*}
    \phi_{\delta_2, f}(\gamma) &= \delta_2 f(\gamma) \delta_2^{-1} \\
    &= (\delta_2 \delta_1^{-1})\cdot (\delta_1 f(\gamma) \delta_1^{-1}) \cdot(\delta_2 \delta_1^{-1})^{-1} \\
    &= (\delta_2 \delta_1^{-1})\cdot \phi_{\delta_1, f}(\gamma)\cdot (\delta_2 \delta_1^{-1})^{-1}
\end{align*}
Thus, $[\phi_{\delta_1, f}] = [\phi_{\delta_2, f}]$, confirming $\phi$ is well-defined.

Next, to prove that $\phi$ is a group homomorphism, consider two maps $f, g \in \mathrm{Homeo}(X)$.  Let $\delta_f$ and $\delta_g$ be paths from $x_0$ to $f(x_0)$ and $g(x_0)$ respectively. Then
\begin{align*}
    \phi_{\delta_f,f}\circ \phi_{\delta_g,g}(\gamma)&=\phi_{\delta_f,f}(\delta_g g(\gamma)\delta_g^{-1})\\
    &=\delta_f f(\delta_g g(\gamma)\delta_g^{-1})\delta_f^{-1}\\
    &= (\delta_f f(\delta_g) )\cdot f (g(\gamma))\cdot(f(\delta_g^{-1})\delta_f^{-1})\\
    &= \phi_{fg,\delta_f f(\delta_g)}(\gamma).
\end{align*}
Therefore, $\phi$ is a group homomorphism. 
\end{proof}

Here are the definitions that lead to the key property of  prime order element $f\in\mathrm{Out}(F_n)$ which $f$ cannot be lifted to a finite order element of $\mathrm{Aut}(F_n)$.

\begin{defn}
Let $R_{r,m}$ be the following  graph:
\begin{itemize}
    \item Let $Z_{r}$ be the unique connected $2$-regular graph with $r$ edges. Thus, $Z_{r}$ is homeomorphic to a circle.
    \item Let $R_{r,m}$ be the graph obtained from $Z_{r}$ by attaching $m-1$ loops to each vertex.
\end{itemize}
\end{defn}
We denote the vertices of $Z_{r}$ by $v_0, v_1,\ldots, v_{r-1}$ and denote the edge connecting vertices $v_{i-1}$ and $v_i$ by $s_i$ for all $1\leq i\leq r$. We denote the loops attached at vertex $v_i$ of $Z_{r}$ as $\ell_{i,j}$ for all $1\leq j\leq m-1$. 
\begin{defn}
The rotation automorphism of $R_{r,m}$ is the automorphism that maps $s_i$ to $s_{i+1}$ and maps $\ell_{i-1,j}$ to $\ell_{i,j}$, for all $1\leq i\leq r$ and $1\leq j\leq m-1$.
\end{defn}

\begin{example}
    Figure \ref{cover1} is an example of graph $R_{r,m}$ where $r=5$ and $m=3$. 
\begin{figure}[htbp]
\centering
\includegraphics[scale=0.5]{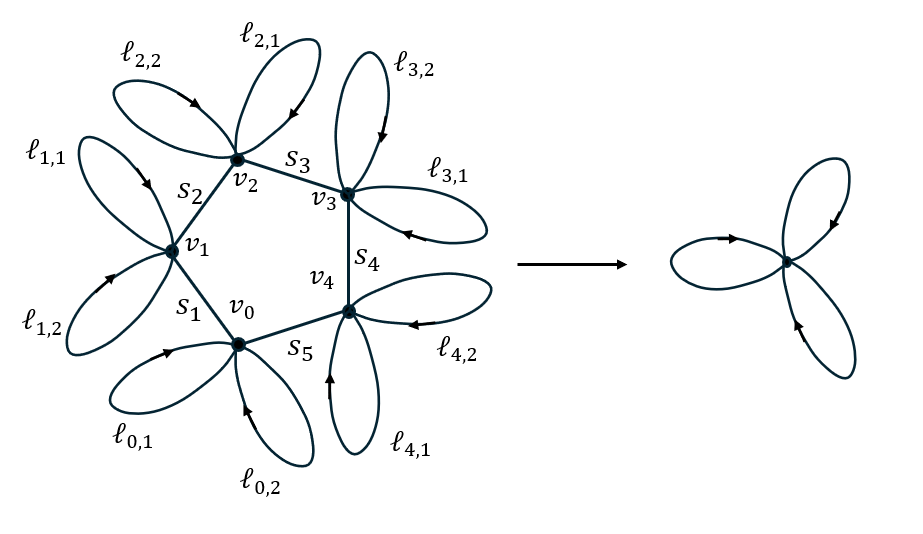}
\caption{$R_{5,3}$ can be viewed as a $\mathbb{Z}/ 5 \mathbb{Z}$ covering space of rose $R_3$\label{cover1}}
\end{figure}
\end{example}
In light of Proposition \ref{relation}, we give the following definition. 
\begin{defn}
We say that a finite order element $f\in\mathrm{Out}(F_n)$ is realized by an element $\tilde{f}\in\mathrm{Aut}(X)$ if the induced element $[\phi_{\delta,\tilde{f}}]\in \mathrm{Out}(\pi_1(X,x_0))$ has the property
\[\langle[\phi_{\delta,\tilde{f}}]\rangle\cong\langle f\rangle.\]
\end{defn}

\begin{prop}
\label{Out}
Let $f\in\mathrm{Out}(F_n)$ be of prime order $p$. Assume $f$ cannot be lifted to a finite order element of $\mathrm{Aut}(F_n)$. Then for some $m\geq 2$ with $p(m-1)+1=n$, we can realize $f$ by the rotation automorphism $\tilde{f} : R_{p,m}\rightarrow R_{p,m}$ such that $\tilde{f}(s_i)=s_{i+1}$ and $\tilde{f}(\ell_{i-1,j})=\ell_{i,j}$  for all $1\leq i\leq p$ and $1\leq j\leq m-1$.
\end{prop}

\begin{proof}
Under these hypotheses, Culler \cite{MC} proved that such $f$ can be realized by an automorphism $\tilde{f}_0$ of a graph $\tilde{X}_0$ where $\langle\tilde{f}_0\rangle$ acts freely on $\tilde{X}_0$. To construct a new realization that simplifies our analysis, let $T$ be a maximal tree in the quotient graph $\tilde{X}_0/\langle\tilde{f}_0\rangle$. Contracting each component in the preimage of $T$ in $\tilde{X}_0$ to a point, we obtain a new graph $\tilde{X}$, where the automorphism $\tilde{f}$, inherited from $\tilde{f}_0$, acts transitively on the vertices. Thus, $\tilde{X}/\langle\tilde{f}\rangle=R_m$ is a wedge of circles. 

Now we have a regular $\mathbb{Z}/p\mathbb{Z}$ cover $\tilde{X}\rightarrow  R_m$ which induces the map $\pi:\pi_1(R_m)\rightarrow\mathbb{Z}/p\mathbb{Z}$. Let \(x_1, \ldots, x_m\) be the generators of \(\pi_1(R_m)\). There must exist an element $x_k$ such that $\pi(x_k)\neq 0$, otherwise, $\pi$ is a zero map and $\tilde{X}\rightarrow  R_m$ cannot be a regular $\mathbb{Z}/p\mathbb{Z}$ cover. Then, $\pi(x_k)$ is a generator of  $\mathbb{Z}/p\mathbb{Z}$.  By applying an automorphism of $\mathbb{Z}/p\mathbb{Z}$ and relabeling the generators of \(\pi_1(R_m)\), we assume without loss of generality that \(\pi(x_1) = 1\).  Thus,
\[\pi: x_jx_1^{-\pi(x_j)}\mapsto \pi(x_j)-\pi(x_j)\cdot 1=0\]
where $2\leq j\leq m$. This allows us to change the basis of $\pi_1(R_m)$ such that $\pi$ maps the first generator to $1$ and all the others to $0$. Consequently, this setup achieves a regular $\mathbb{Z}/p\mathbb{Z}$ cover $\tilde{f} : R_{p,m}\rightarrow R_m$ that realizes $f$. Since $\tilde{f}$ has order $p$, and it acts freely and transitively on the vertices of $ R_{p,m}$, it follows that for some $2\leq j\leq p$, the realization $\tilde{f}(s_1)=s_{j}$ and  subsequently $\tilde{f}(s_{1+i})=s_{j+i}$ for all $1\leq i\leq p$. By replacing \(\tilde{f}\) with an appropriate power, we can ensure $\tilde{f}(s_1)=s_2$ and $\tilde{f}(s_i)=s_{i+1}$ for all $1\leq i\leq p$. After relabeling $\ell_{i,j}$, the realization $\tilde{f}(\ell_{i-1,j})=\ell_{i,j}$. We have now proved the proposition.
\end{proof}

\begin{proof}[Proof of Proposition \ref{prop2}]
Consider a prime order $p$ element $f\in\mathrm{Out}(F_n)$. If $f$ can be lifted to a finite order element of $\mathrm{Aut}(F_n)$, then by Proposition \ref {prop1},  we are done. If not, by Proposition \ref{Out}, we can realize $f$ by the rotation automorphism $\tilde{f} : R_{p,m}\rightarrow R_{p,m}$ where $\tilde{f}(s_i)=s_{i+1}$ and $\tilde{f}(\ell_{i-1,j})=\ell_{i,j}$. See Figure \ref{R5,3.PNG} for an illustration.
\begin{figure}[htbp]
\centering
\includegraphics[scale=0.5]{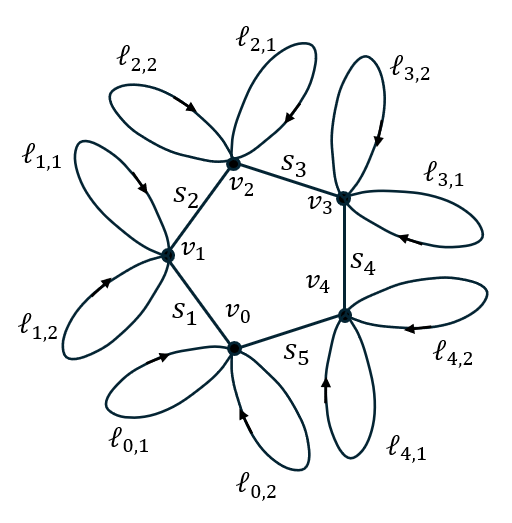}
\caption{$R_{5,3}$ can be viewed as a $\mathbb{Z}/ 5 \mathbb{Z}$ covering space of rose $R_3$\label{R5,3.PNG}}
\end{figure}
Notice that $s_1$ is the path from $v_0$ to $\tilde{f}(v_0)=v_1$. By Proposition \ref{relation}, the rotation automorphism $\tilde{f}$ induces $\phi_{s_1,\tilde{f}}$ in $\mathrm{Aut}(\pi_1(X,x_0))$. Consider the following basis of $\pi_1(R_{p,m}, v_0)$: 
\begin{align*}
x_{0,j}= \ell_{0,j}, \quad x_{i,j} = s_1\ldots s_i\ell_{i,j} (s_1\ldots s_i)^{-1}, \quad c = s_1 \ldots s_p
\end{align*}
where $1\leq i\leq p,\text{ and } 1\leq j\leq m-1$. The automorphism $\phi_{s_1,\tilde{f}}$ satisfies
\begin{align*}
\phi_{s_1,\tilde{f}}:      &x_{i,j} \mapsto x_{i+1,j}   &\phi_{s_1,\tilde{f}}^{-1}:     &x_{0,j} \mapsto c^{-1} x_{p-1,j} c \\
        &x_{p-1,j} \mapsto c x_{0,j} c^{-1}    &    &x_{i+1,j} \mapsto x_{i,j} \\  
        &c \mapsto c                 &            &c \mapsto c
\end{align*}
where $0\leq i\leq p-2$ and $1\leq j\leq m-1$.  Let $\Phi_1 = L_{x_{0,1}x_{1,1}}^{-1}\phi_{s_1,\tilde{f}}L_{x_{0,1}x_{1,1}}\phi_{s_1,\tilde{f}}^{-1}$ and $\Phi_2 = L_{x_{0,1}x_{1,1}}\phi_{s_1,\tilde{f}}L_{x_{0,1}x_{1,1}}^{-1}\phi_{s_1,\tilde{f}}^{-1} $.  The maps $\Phi_1$ and $\Phi_2$ are in $N(\phi_{s_1,\tilde{f}})$. We calculate $\Phi_1(x_{0,1})$ and $\Phi_1(x_{1,1})$ as follows:
\begin{align*}
     & x_{0,1} \xmapsto{\phi_{s_1,\tilde{f}}^{-1}} c^{-1}x_{p-1,1}c \xmapsto{L_{x_{0,1}x_{1,1}}} c^{-1}x_{p-1,1} c\xmapsto{ \phi_{s_1,\tilde{f}} } x_{0,1} \xmapsto{L_{x_{0,1}x_{1,1}}^{-1}} x_{1,1}^{-1}x_{0,1}, \\
     & x_{1,1} \xmapsto{\phi_{s_1,\tilde{f}}^{-1}} x_{0,1} \xmapsto{L_{x_{0,1}x_{1,1}}} x_{1,1}x_{0,1} \xmapsto{ \phi_{s_1,\tilde{f}} } x_{2,1}x_{1,1} \xmapsto{L_{x_{0,1}x_{1,1}}^{-1}} x_{2,1}x_{1,1}.
\end{align*}
The map $\Phi_1$ fixes the other basis elements in $\pi_1(\tilde{R}_{p,m}, v_0)$. Similarly, we calculate $\Phi_2(x_{0,1})$ and $\Phi_2(x_{1,1})$ as follows: 
\begin{align*}
     & x_{0,1} \xmapsto{\phi_{s_1,\tilde{f}}^{-1}} c^{-1}x_{p-1,1} c\xmapsto{L_{x_{0,1}x_{1,1}}^{-1}} c^{-1}x_{p-1,1} c\xmapsto{ \phi_{s_1,\tilde{f}} } x_{0,1} \xmapsto{L_{x_{0,1}x_{1,1}}} x_{1,1}x_{0,1}, \\
     & x_{1,1} \xmapsto{\phi_{s_1,\tilde{f}}^{-1}} x_{0,1} \xmapsto{L_{x_{0,1}x_{1,1}}^{-1}} x_{1,1}^{-1}x_{0,1} \xmapsto{ \phi_{s_1,\tilde{f}} } x_{2,1}^{-1}x_{1,1} \xmapsto{L_{x_{0,1}x_{1,1}}} x_{2,1}^{-1}x_{1,1}.
\end{align*}
The map $\Phi_2$ fixes the other basis elements in $\pi_1(X,v_0)$. Consider 
\[ \Phi_1 \circ \Phi_2 = L_{x_{0,1}x_{1,1}}^{-1}\phi_{s_1,\tilde{f}}L_{x_{0,1}x_{1,1}}\phi_{s_1,\tilde{f}}^{-1} \circ L_{x_{0,1}x_{1,1}}\phi_{s_1,\tilde{f}}L_{x_{0,1}x_{1,1}}^{-1}\phi_{s_1,\tilde{f}}^{-1}. \]
By direct computation, 
\begin{align*}
     & x_{0,1} \xmapsto{\Phi_2} x_{1,1}x_{0,1} \xmapsto{\Phi_1} x_{2,1}x_{0,1},\\
     & x_{1,1} \xmapsto{\Phi_2} x_{2,1}^{-1}x_{1,1} \xmapsto{\Phi_1} x_{1,1},
\end{align*}
and other basis elements are fixed. We can see that $\Phi_1 \circ \Phi_2 = L_{x_{0,1}x_{2,1}}\in N(\phi_{s_1,\tilde{f}})$. The image of $\phi_{s_1,\tilde{f}}$ in $\mathrm{GL}_n(\mathbb{Z})$ is the direct sum $M^k\oplus I_{n-pk}$, where $I$ is the identity matrix and $M$ is the $p$ by $p$ matrix
\[
\begin{pmatrix}
0 & 0 & 0 &\ldots & 0 & 1 \\
1 & 0 & 0 &\ldots & 0 & 0 \\
0 & 1 & 0 &\ldots & 0 & 0 \\
0 & 0 & 1 &\ldots & 0 & 0 \\
0 & 0 & 0 &\ldots & 0 & 0 \\
\vdots&\vdots&\vdots&&\vdots&\vdots\\
0 & 0 & 0 &\ldots & 1 & 0 \\
\end{pmatrix}.
\]
Then $\det(\bar{\phi}_{s_1,\tilde{f}})=(-1)^{p-1}=1$. Projecting $\Phi_1$, $\Phi_2$ and $\phi_{s_1,\tilde{f}}$ to $\mathrm{Out}(\pi_1(X,x_0))$, we conclude that $N([\phi_{s_1,\tilde{f}}])=\mathrm{SOut}(F_n)$ by Lemma \ref{left multi Out}. Since $\langle f \rangle\cong \langle[\phi_{s_1,\tilde{f}}]\rangle$, it follows that $N(f)=\mathrm{SOut}(F_n)$.
\end{proof}

\section{$p = 2$ case}
Let $f$ be an order $2$ element in $\mathrm{Aut}(F_n)$ and $\tilde{f}\in\mathrm{Aut}(X,v_0)$ be a good realization of $f$. The realization $\tilde{f}$ acts on the edges of $(X,v_0)$ as follows:
\begin{enumerate}
\item For each vertex $v$, the loops based at $v$ are permuted or \textit{inverted}. Some of them might be fixed.
\item For each pair of distinct vertices $v$ and $w$, the edges joining $v$ and $w$ are permuted. None of them are fixed.
\end{enumerate}
The \textit{inversion} in actions creates more scenarios to consider than when dealing with elements of higher orders. To handle these scenarios, we need to develop a key lemma first.

\begin{lemma}\label{level 2}
Let $f\in\mathrm{Aut}(F_n)$. Assume the following conditions are satisfied:
\begin{enumerate}
\item the Torelli subgroup $IA_n\subset N(f)$, and
\item the level $2$ subgroup $\Gamma_n(2)\subset\overline{N(f)}$.
\end{enumerate}
If $\bar{f}\in\Gamma_n(2)$, then when $\det(\bar{f}) = 1$, we have the short exact sequence
\[
1 \rightarrow IA_n \rightarrow N(f) \rightarrow \Gamma_n(2) \rightarrow 1
\]
and when $\det(\bar{f}) = -1$, we have the short exact sequence
\[
1 \rightarrow IA_n \rightarrow N(f)\rightarrow \mathrm{G}\Gamma_n(2) \rightarrow 1.
\]
If $\bar{f}\notin\Gamma_n(2)$, then 
\[
    N(f) = 
    \begin{cases}
        \mathrm{SAut}(F_n) & \text{if det}(\bar{f}) = 1, \\
        \mathrm{Aut}(F_n) & \text{if det}(\bar{f}) = -1.
    \end{cases}
\]
\end{lemma}

\begin{proof}
If $\bar{f}\in\Gamma_n(2)$, the conclusion is straightforward. If $\bar{f}\notin\Gamma_n(2)$,  consider 
\[G = \text{Im}(\overline{N(f)} \cap \mathrm{SL}_n(\mathbb{Z}) \rightarrow \mathrm{SL}_n(\mathbb{Z}/2\mathbb{Z}))\]
where the map is the natural projection map.  Then we have the short exact sequence:
\[  1 \rightarrow \Gamma_n(2) \rightarrow \overline{N(f)} \cap \mathrm{SL}_n(\mathbb{Z}) \rightarrow G \rightarrow 1\]
Since $\bar{f}\notin\Gamma_n(2)$, we have that $G$ is a nontrivial normal subgroup in $\mathrm{SL}_n(\mathbb{Z}/2\mathbb{Z})$. Since $\mathrm{SL}_n(\mathbb{Z}/2\mathbb{Z})\cong\mathrm{PSL}_n(\mathbb{Z}/2\mathbb{Z})$ which is simple for $n \geq 3$, see \cite{CJ}, we conclude that $G = \mathrm{SL}_n(\mathbb{Z}/2\mathbb{Z})$. This assertion implies that $\overline{N(f)}$ contains $\mathrm{SL}_n(\mathbb{Z})$ and the lemma will follow as in the proof of Lemma \ref{left multi}. 
\end{proof}

We will now proceed to examine the possible cases individually. For the following discussion, the map $\tilde{f}\in\mathrm{Aut}(X,v_0)$ is a good realization of $f$.

\setcounter{case}{0}
\begin{case}
If $X$ has only one vertex $v_0$, then $X$ is a rose $R_n$. Each loop denoted by $x_i(1\leq i\leq n)$ is a basis element of $\pi_1(X, v_0)$. We have two situations on how $\tilde{f}$ acts on $X$.
\begin{subcase}
In this case, the graph $X=R_n$ and the realization $\tilde{f}$ permutes all the loops in $X$. By relabeling each loop,  the action of $f$ on $\mathrm{Aut}(F_n)\cong\pi_1(X,v_0)$ satisfies: 
\begin{align*}
f:  &x_{1} \mapsto x_{2} \\
    &x_2\mapsto x_1 \\
     &\text{ }\vdots\quad\quad\text{ }\vdots\\
      &x_{n-1}\mapsto x_n\\
     &x_n\mapsto x_{n-1}
\end{align*}
Let $\Psi=I_{x_1}fI_{x_1}f\in N(f)$. By direct computation, $\Psi$ acts on the basis elements of $\pi_1(X,v_0)$ as follows:
\begin{align*}
     x_1 &\xmapsto{f} x_2 \xmapsto{I_{x_1}} x_2 \xmapsto{f}x_1\xmapsto{I_{x_1}}x_1^{-1},\\
     x_2 &\xmapsto{f} x_1 \xmapsto{I_{x_1}} x_1^{-1} \xmapsto{f}x_2^{-1}\xmapsto{I_{x_1}}x_2^{-1},
\end{align*}
and fixes the other basis elements. Therefore, consider $L_{x_1x_2}\Psi L_{x_1x_2}^{-1}\Psi $. It acts on the basis elements as follows:
\[
x_{1} \xmapsto{\Psi } x_{1}^{-1} \xmapsto{L_{x_{1}x_{2}}^{-1}} x_{1}^{-1}x_2 \xmapsto{ \Psi  } x_1x_2^{-1} \xmapsto{L_{x_1x_2}} x_{2}x_{1}x_2^{-1},
\]
and it fixes the other generators. Therefore, $L_{x_1x_2}\Psi L_{x_1x_2}^{-1}\Psi =C_{x_1x_2}\in N(f)$ implies $IA_n\subset N(f )$. Consider  $L_{x_1x_3}\Psi L_{x_1x_3}^{-1}\Psi $. It acts on  the basis elements as
\begin{align*}
    &x_{1} \xmapsto{\Psi } x_1^{-1} \xmapsto{L_{x_{1}x_{3}}^{-1}} x_1^{-1}x_3 \xmapsto{ \Psi  } x_1x_3 \xmapsto{L_{x_1x_3}} x_3x_1x_3,
\end{align*}
and fixes the other generators. We conclude that the image of $L_{x_1x_3}\Psi L_{x_1x_3}^{-1}\Psi $ in $\mathrm{GL}_n(\mathbb{Z})$ is $e_{3,1}^2$ which normally generates $\Gamma_n(2)$. Therefore, $\Gamma_n(2)\subset \overline{N( \Psi )}$. Since $N( \Psi )\subset N(f)$ and $f\notin \Gamma_n(2)$, by Lemma \ref{level 2}, we conclude that 
\[
    N(f) = 
    \begin{cases}
        \mathrm{SAut}(F_n) & \text{if det}(\bar{f}) = 1, \\
        \mathrm{Aut}(F_n) & \text{if det}(\bar{f}) = -1.
    \end{cases}
\]

\end{subcase}
\begin{subcase}
In this case, the graph $X=R_n$ and the realization $\tilde{f}$ permutes some of the loops in $X$  but not all of them. Notice that $\tilde{f}$ may or may not invert other loops. Recall that $n\geq 3$. By relabeling each loop,  the action of $f$ on $\mathrm{Aut}(F_n)\cong\pi_1(X,v_0)$ satisfies: 
\begin{align*}
f:  &x_{1} \mapsto x_{2} \\
    &x_2\mapsto x_1 \\
    &x_3\mapsto x_3^\delta,\quad \delta=\{-1,1\},\\
     &x_i\mapsto f(x_i),\quad 4\leq i\leq n.
\end{align*}
Here $f(x_i)$ is either $x_i^{\delta}$ or $x_j$ where $x_j$ is distinct from $x_1, x_2$ and $x_3$. Recall the definition of $P_{x_ix_j}$ in the Definition \ref{defn}. Let $\Psi=P_{x_2x_3}fP_{x_2x_3}f\in N(f)$. By direct computation, $\Psi$ acts on the basis element of $\pi_1(X,v_0)$ as follows:
\begin{align*}
     x_1 &\xmapsto{f} x_2 \xmapsto{P_{x_2x_3}} x_3 \xmapsto{f}x_3^{\delta}\xmapsto{P_{x_2x_3}}x_2^{\delta},\\
     x_2 &\xmapsto{f} x_1 \xmapsto{P_{x_2x_3}} x_1 \xmapsto{f}x_2\xmapsto{P_{x_2x_3}}x_3,\\
     x_3 &\xmapsto{f} x_3^{\delta} \xmapsto{P_{x_2x_3}} x_2^{\delta} \xmapsto{f}x_1^{\delta}\xmapsto{P_{x_2x_3}}x_1^{\delta},\\
    x_i &\xmapsto{f} f(x_i) \xmapsto{P_{x_2x_3}} f(x_i) \xmapsto{f}x_i\xmapsto{P_{x_2x_3}}x_i,
\end{align*}
where $4\leq i\leq n$. Therefore, we have
\begin{align*}
     x_1 &\xmapsto{\Psi} x_2^{\delta} \xmapsto{\Psi} x_3^{\delta} \xmapsto{\Psi}x_1,\\
     x_2 &\xmapsto{\Psi} x_3 \xmapsto{\Psi} x_1^{\delta} \xmapsto{\Psi}x_2,\\
     x_3 &\xmapsto{\Psi} x_1^{\delta} \xmapsto{\Psi} x_2 \xmapsto{\Psi}x_3,\\
    x_i &\xmapsto{\Psi} x_i \xmapsto{\Psi} x_i \xmapsto{\Psi}x_i, 
\end{align*}
where $4\leq i\leq n$. We can see that $\Psi$ is an order $3$ element of in $N(f)$. Using Proposition \ref{prop1}, we are done.
\end{subcase}
\begin{subcase}
In this case, the graph $X=R_n$ and the realization $\tilde{f}$ does not permute any loops but inverts at least one loop in $X$.  After reordering the basis,  the action of $f$ on $\mathrm{Aut}(F_n)$ satisfies: 
\begin{align*}
f:  &x_{1} \mapsto x_{1}^{-1} \\
    &x_i\mapsto x_i^{\delta_i},\quad \delta_i=\{-1,1\}
\end{align*}
For $2\leq i,j \leq n$, there are three subcases:
\begin{itemize}
    \item all $\delta_i=-1$,
    \item at least one $\delta_i=-1$ and one $\delta_j=1$,
    \item all $\delta_i=1$.
\end{itemize}
We will discuss them individually.
\begin{casea}
In this case, all $\delta_i=-1$ for $2\leq i\leq n$. Define $\Phi=L_{x_1x_2}fL_{x_1x_2}^{-1}f\in N(f)$. It acts on the basis elements as follows:
\[
x_{1} \xmapsto{f} x_{1}^{-1} \xmapsto{L_{x_{1}x_{2}}^{-1}} x_{1}^{-1}x_2 \xmapsto{ f } x_1x_2^{-1} \xmapsto{L_{x_1x_2}} x_{2}x_{1}x_2^{-1},
\]
and it fixes the other generators. Therefore, $\Phi=C_{x_1x_2}$ which implies \(IA_n \subset N(f)\). Given $\bar{f}_{-1}=-I$, we establish the short exact sequence:
\[
1 \rightarrow IA_n \rightarrow N(f) \rightarrow \{\pm I\} \rightarrow 1.
\]
\end{casea}
\begin{casea}
\label{p=2_subcase}In this case, at least one $\delta_i=-1$ and one $\delta_j=1$. After relabeling the basis elements, we assume $\delta_2=-1$ and  $\delta_3=1$. Therefore, we have the following automorphism
\begin{align*}
f:  &x_{1} \mapsto x_{1}^{-1} \\
    &x_2\mapsto x_2^{-1}\\
    &x_3\mapsto x_3\\
    &x_i\mapsto x_i^{\delta_i},\quad \delta_i=\{-1,1\}.
\end{align*}
Consider $L_{x_1x_2}fL_{x_1x_2}^{-1}f\in N(f)$. It acts on the basis elements as follows:
\[
x_{1} \xmapsto{f} x_{1}^{-1} \xmapsto{L_{x_{1}x_{2}}^{-1}} x_{1}^{-1}x_2 \xmapsto{ f } x_1x_2^{-1} \xmapsto{L_{x_1x_2}} x_{2}x_{1}x_2^{-1},
\]
and it fixes the other generators. Therefore, $L_{x_1x_2}fL_{x_1x_2}^{-1}f=C_{x_1x_2}$ implies $IA_n\subset N(f)$. Let  $\Psi=L_{x_1x_3}fL_{x_1x_3}^{-1}f$. The action of $\Psi$ on  the basis elements is
\begin{align*}
    &x_{1} \xmapsto{f} x_1^{-1} \xmapsto{L_{x_{1}x_{3}}^{-1}} x_1^{-1}x_3 \xmapsto{ f } x_1x_3 \xmapsto{L_{x_1x_3}} x_3x_1x_3,
\end{align*}
and it fixes the other generators. We conclude that the image of $\Psi$ in $\mathrm{GL}_n(\mathbb{Z})$ is $e_{3,1}^2$ which normally generates $\Gamma_n(2)$. Therefore,$\Gamma_n(2)\subset \overline{N( f)}$. Since $\bar{f}\in\Gamma_n(2)$, by Lemma \ref{level 2}, when $\det(\bar{f}) = 1$, we have the short exact sequence
\[
1 \rightarrow IA_n \rightarrow N(f) \rightarrow \Gamma_n(2) \rightarrow 1
\]
and when $\det(\bar{f}) = -1$, we have
\[
1 \rightarrow IA_n \rightarrow N(f)\rightarrow \mathrm{G}\Gamma_n(2) \rightarrow 1.
\]
\end{casea}
\begin{casea}
In this case, all $\delta_i=1$ for $2\leq i\leq n$. Consider $\Psi=P_{x_1x_2}f P_{x_1x_2}f\in N(f)$. It acts on the basis elements as follows:
\begin{align*}
    &x_{1} \xmapsto{f} x_1^{-1} \xmapsto{P_{x_{1}x_{2}}} x_2^{-1} \xmapsto{ f } x_2^{-1} \xmapsto{P_{x_1x_2}} x_1^{-1},\\
     &x_2 \xmapsto{f} x_2 \xmapsto{P_{x_{1}x_{2}}} x_1 \xmapsto{ f } x_1^{-1} \xmapsto{P_{x_1x_2}} x_2^{-1},\\
      &x_3 \xmapsto{f} x_3 \xmapsto{P_{x_{1}x_{2}}} x_3 \xmapsto{ f } x_3 \xmapsto{P_{x_1x_2}} x_3,
\end{align*}
and leaves other generators fixed. Notice that $\Psi$ is the map which $\delta_2=-1$, and all $\delta_i=1$ for $3\leq i\leq n$. It is a specific case of $f$ we defined in Case \ref{p=2_subcase}. Given $N(\Psi)\subset N(f)$, we have $IA_n\subset N(f
)$ and $\Gamma_n(2)\subset \overline{N(f)}$. Since  $\bar{f}\in\Gamma_n(2)$ and $\det(\bar{f}) = -1$, we conclude with the short exact sequence
\[
1 \rightarrow IA_n \rightarrow N(f)\rightarrow \mathrm{G}\Gamma_n(2) \rightarrow 1.
\]
\end{casea}
\end{subcase}
\end{case}

Now, we finished the discussion on $X=R_n$. If $X$ has an edge that is not a loop, then we have three situations.
\begin{case}
$X$ contains a subgraph $H_4$ where $H_4$ is a hairy graph $H_{2k}$ for $k=2$. By replacing $f\in \mathrm{Aut}(F_n)$ by a conjugate element, we can assume our basepoint $v_0$ lies in the subgraph $H_{4}$ of $X$. We denote the other vertex of $H_{4}$ to be $v_1$ and label the edges of $H_{4}$ as $s_{i},(1 \leq i \leq 4)$ such that $\tilde{f}(s_i) = s_{i+1}$.  See Figure \ref{H_4} for an illustration. 
Consider the following basis of $\pi_1(H_4, v_0)$:
\[x_1 = s_1s_2^{-1},\quad x_2= s_2s_3^{-1}s_4s_2^{-1}, \quad x_3 = s_2s_4^{-1}.\]
The action of $f|_{\pi_1(H_4,v_0)}$ satisfies:
\[f(x_1) = x_1^{-1},\quad f(x_2) = x_1x_2^{-1}x_1^{-1}, \quad f(x_3) = x_1x_2x_3.\]
\begin{figure}[htbp]
\centering
\includegraphics[scale=0.4]{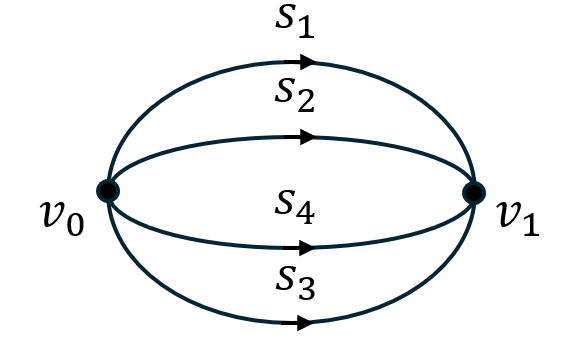}
\caption{subgraph $H_4$ \label{H_4}}
\end{figure}
Let $A = X \setminus (\text{edges in } H_4)$. We first deal with the case where $A$ is connected.
\begin{subcase}\label{A connected p=2}
In this case, $A = X \setminus (\text{edges in } H_4)$ which is connected. Let $\gamma_0$ be a path in $A$ connecting the vertices $v_1$ and $v_0$. See Figure \ref{H_4 A connected} for an illustration. We claim that
\[ \pi_1(X, v_0) = \pi_1(H_{4}, v_0) * \pi_1(A,v_0) * \Gamma,\quad\text{where }\Gamma\cong <s_3\gamma_0>.\]
The argument is identical to the proof in Case \ref{A connected} of the case  $p\geq 5$ of Proposition \ref{prop1}.
\begin{figure}[htbp]
\centering
\includegraphics[scale=0.4]{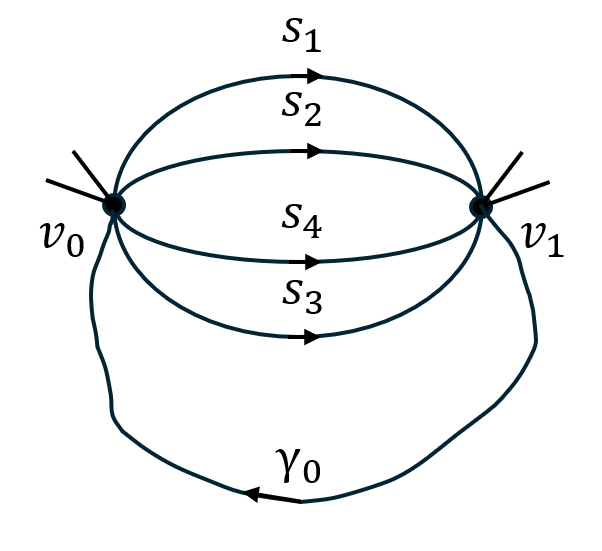}
\caption{subgraph $H_4$ and $A$ is connected \label{H_4 A connected}}
\end{figure}

Since the automorphism $\tilde{f}$ fixes all vertices,  the automorphism $f$ preserves $\pi_1(H_4, v_0)$ and $\pi_1(A, v_0)$. Let 
$\Phi = R_{x_1x_2}fR_{x_1x_2}^{-1}f \in N(f)$. It acts on $\pi_1(H_4, v_0)$ as follows:
\begin{align*}
     x_1 &\xmapsto{f} x_1^{-1} \xmapsto{R_{x_1x_2}^{-1}} x_2x_1^{-1} \xmapsto{f}x_1x_2^{-1}\xmapsto{R_{x_1x_2}}x_1,\\
     x_2 &\xmapsto{f} x_1x_2^{-1}x_1^{-1} \xmapsto{R_{x_1x_2}^{-1}}x_1x_2^{-1}x_1^{-1} \xmapsto{f}x_2\xmapsto{R_{x_1x_2}}x_2,\\
     x_3 &\xmapsto{f} x_1x_2x_3 \xmapsto{R_{x_1x_2}^{-1}}x_1x_3 \xmapsto{f}x_2x_3\xmapsto{R_{x_1x_2}}x_2x_3.
\end{align*}
Therefore, $\Phi|_{\pi_1(H_4, v_0)}= L_{x_3x_2}$. Since $R_{x_1x_2}$ acts trivially on $\pi_1(A,v_0)$, so we have
\[
    \Phi|_{\pi_1(A,v_0)} =(f\circ f)|_{\pi_1(A,v_0)} = \mathrm{Id}_{\pi_1(A,v_0)}.
\]
Since 
\[
    \tilde{f}(s_3\gamma_0) = \tilde{f}(s_3)\tilde{f}(\gamma_0)=s_4\tilde{f}(\gamma_0)=(s_4s_2^{-1})(s_2s_3^{-1}s_4s_2^{-1})(s_2s_4^{-1})(s_3\gamma_0)(\gamma_0^{-1}\tilde{f}(\gamma_0)),
\]
it follows that on $\Gamma$, we have
\begin{align}\label{finverse on gamma p=2}
     f(s_3\gamma_0)&=(s_4s_2^{-1})(s_2s_3^{-1}s_4s_2^{-1})(s_2s_4^{-1})(s_3\gamma_0)(\gamma_0^{-1}\tilde{f}^{-1}(\gamma_0))\\
     &=  x_3^{-1}x_2x_3(s_3\gamma_0)(\gamma_0^{-1}\tilde{f}(\gamma_0)).
\end{align}
Using the equation \eqref{finverse on gamma p=2}, 
\[f(x_3^{-1}x_2x_3(s_3\gamma_0)(\gamma_0^{-1}\tilde{f}(\gamma_0)))  = f(f(s_3\gamma_0)) = s_3\gamma_0.\]
Then we calculate the map $\Phi$ acting on $\Gamma$:
\begin{align*}
\Phi|_\Gamma: s_3\gamma_0 &\xmapsto{f}  x_3^{-1}x_2x_3(s_3\gamma_0)(\gamma_0^{-1}\tilde{f}(\gamma_0))\\             
&\xmapsto{R_{x_1x_2}^{-1}}  x_3^{-1}x_2x_3(s_3\gamma_0)(\gamma_0^{-1}\tilde{f}(\gamma_0))\\
&\xmapsto{f} s_3\gamma_0 \\
&\xmapsto{R_{x_1x_2}} s_3\gamma_0;
\end{align*}
This implies that $\Phi|_{\Gamma} = \text{Id}_{\Gamma}$. Thus, $\Phi = L_{x_3x_2}$ on $\pi_1(X,v_0)=F_n$. Since we have $L_{x_3x_2}\in N(f)$, by Lemma \ref{left multi}, 
\[
    N(f) = 
    \begin{cases}
        \mathrm{SAut}(F_n) & \text{if det}(\bar{f}) = 1, \\
        \mathrm{Aut}(F_n) & \text{if det}(\bar{f}) = -1.
    \end{cases}
\]
\end{subcase}

\begin{subcase}
In this case,  $A = X \setminus (\text{edges in } H_4)$ which is disconnected. Let $A = A_1 \cup A_2$ where $A_1$ contains the vertex $v_0$ and $A_2$ contains the vertex $v_1$. We have
\[ \pi_1(X, v_0) = \pi_1(H_4, v_0) * \pi_1(A_1, v_0) * \Gamma' , \]
where $\Gamma'=\{s_3\gamma s_3^{-1}|\gamma\in\pi_1(A_2, v_1)\}$.  See Figure \ref{H_4 A disconnected} for an illustration.

\begin{figure}[htbp]
\centering
\includegraphics[scale=0.4]{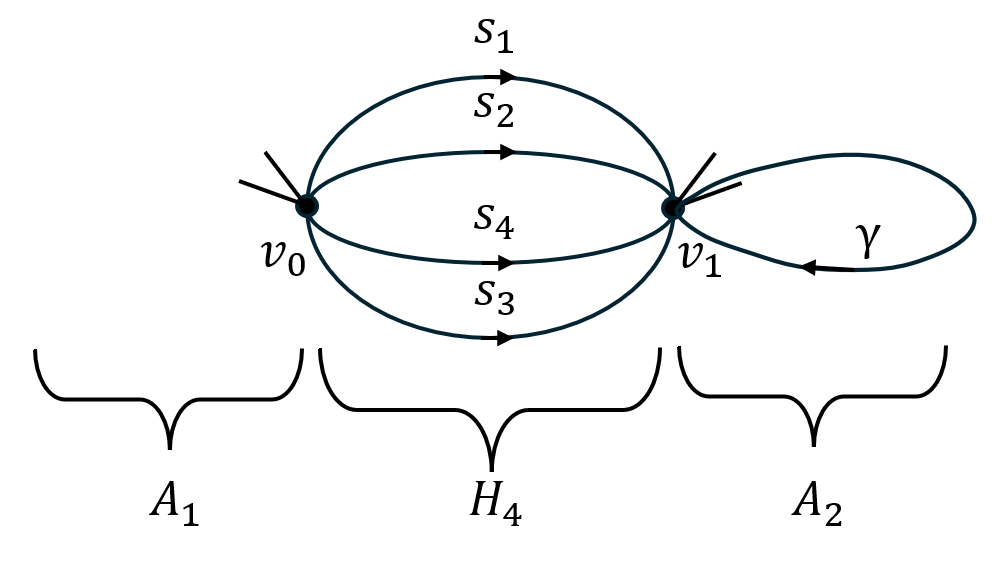}
\caption{subgraph $H_4$ and $A$ is disconnected \label{H_4 A disconnected}}
\end{figure}

Since the automorphism $\tilde{f}$ fixes all vertices,  the automorphism $f$ preserves $\pi_1(H_{pk}, v_0)$ and $\pi_1(A_1, v_0)$.  Using the same strategy as in the Case \ref{A connected p=2}, we consider $\Phi = R_{x_1x_2}fR_{x_1x_2}^{-1}f \in N(f)$. Then
\begin{align*}
    \Phi|_{\pi_1(H_4, v_0)} &=L_{x_3x_2},\text{ and}\\
    \Phi|_{\pi_1(A_1, v_0)} &=(f\circ f)|_{\pi_1(A_1, v_0)} = \text{Id}_{\pi_1(A_1, v_0)}.
\end{align*}
Since
\begin{align*}
\tilde{f}(s_3\gamma s_3^{-1}) &= \tilde{f}(s_3)\tilde{f}(\gamma)\tilde{f}(s_3^{-1})\\
&=s_4\tilde{f}(\gamma)s_4^{-1}\\
    &=(s_4s_2^{-1})(s_2s_3^{-1}s_4s_2^{-1})(s_2s_4^{-1})(s_3\tilde{f}(\gamma)s_3^{-1})((s_4s_2^{-1})(s_2s_3^{-1}s_4s_2^{-1})(s_2s_4^{-1}))^{-1}
\end{align*}
it follows that on $\Gamma'$,  we have 
\begin{align}\label{finverse on gamma2 p=2}
     f(s_3\gamma s_3^{-1})=x_3^{-1}x_2x_3(s_3\tilde{f}(\gamma)s_3^{-1})(x_3^{-1}x_2x_3)^{-1}, \quad\text{where }s_3\tilde{f}(\gamma)s_3^{-1}\in \Gamma'.
\end{align}
Using the equation  \eqref{finverse on gamma2 p=2},
\[f(x_3^{-1}x_2x_3(s_3\tilde{f}(\gamma)s_3^{-1})(x_3^{-1}x_2x_3)^{-1})  = f(f(s_3\gamma s_3^{-1}) = s_3\gamma s_3^{-1}.\]
Then we calculate the map $\Phi$ acting on $\Gamma'$:
\begin{align*}
\Phi|_{\Gamma'}: s_3\gamma s_3^{-1} &\xmapsto{f}  x_3^{-1}x_2x_3(s_3\tilde{f}(\gamma)s_3^{-1})(x_3^{-1}x_2x_3)^{-1}\\             
&\xmapsto{R_{x_1x_2}^{-1}}  x_3^{-1}x_2x_3(s_3\tilde{f}(\gamma)s_3^{-1})(x_3^{-1}x_2x_3)^{-1}\\
&\xmapsto{f} s_3\gamma s_3^{-1} \\
&\xmapsto{R_{x_1x_2}} s_3\gamma s_3^{-1};
\end{align*}
This implies that $\Phi|_{\Gamma'} = \text{Id}_{\Gamma'}$. Thus, $\Phi = L_{x_3x_2}$ on $\pi_1(X,v_0)=F_n$. Since we have $L_{x_3x_2}\in N(f)$, by Lemma \ref{left multi}, 
\[
    N(f) = 
    \begin{cases}
        \mathrm{SAut}(F_n) & \text{if det}(\bar{f}) = 1, \\
        \mathrm{Aut}(F_n) & \text{if det}(\bar{f}) = -1.
    \end{cases}
\]
\end{subcase}
\end{case}
To explain what cases remain to be dealt with, we make the following definitions.
\begin{defn} A graph is called a closed chain $C_k$ if it satisfies the following conditions: 
\begin{enumerate} 
    \item The graph consists of $k$ vertices, denoted as $v_0, v_1, \ldots, v_{k-1}$. 
    \item Taking indices modulo $k$ for $0\leq i\leq k-1$, the vertex $v_i$ is connected to the vertex $v_{i+1}$ by exactly two edges.
    \item The graph contains no other edges besides those described above.
\end{enumerate} 
\end{defn}

\begin{defn} A graph is called an open chain $O_k$ if it satisfies the following conditions: 
\begin{enumerate} 
    \item The graph consists of $k$ vertices, denoted as $v_0, v_1, \ldots, v_{k-1}$. 
    \item For $0\leq i\leq k-1$, each vertex $v_i$ is connected to the vertex $v_{i+1}$ by exactly two edges. Additionally, there is a loop attached at the vertices $v_0$ and $v_{k-1}$.
    \item The graph contains no other edges besides those described above.
\end{enumerate} 
\end{defn}
\begin{figure}[htbp]
    \centering
    \includegraphics[scale=0.4]{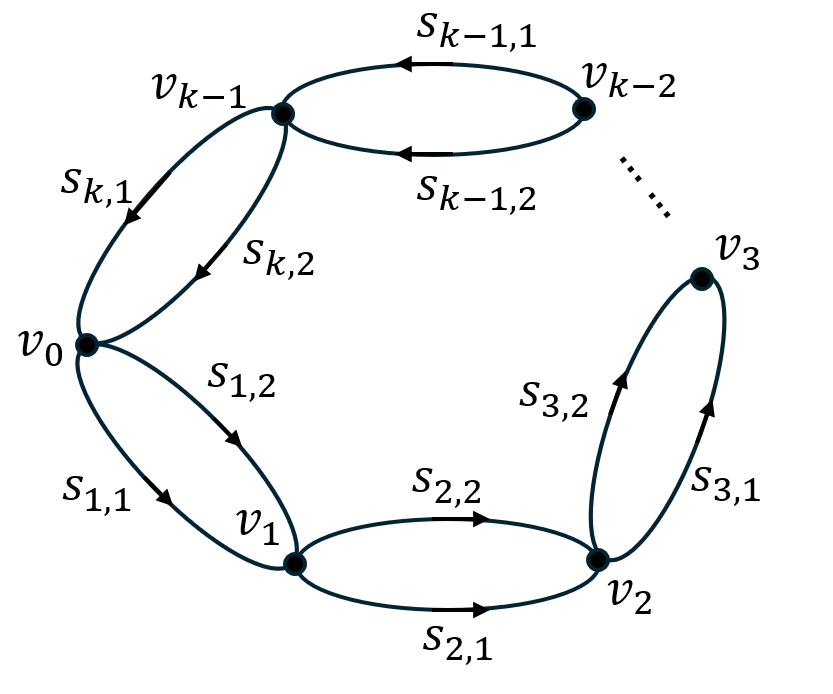}
    \caption{closed chain $C_k$}
\end{figure} 
\begin{figure}[htbp]
    \centering
    \includegraphics[scale=0.4]{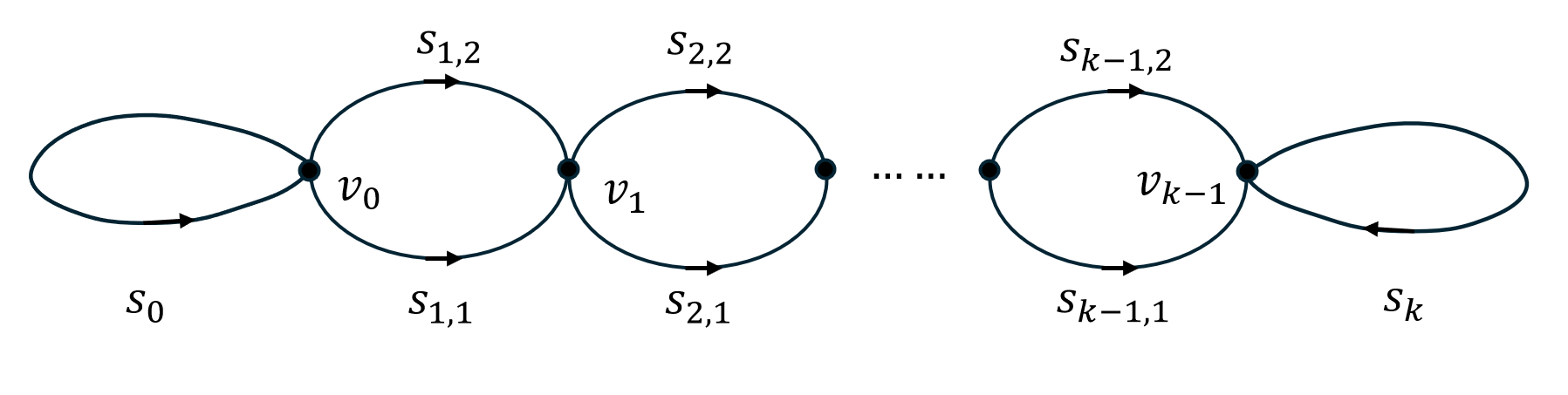}
    \caption{open chain $O_k$}
\end{figure} 

Consider a graph $X$ that has at least one edge which is not a loop, but does not contain the subgraph $H_4$. Under these conditions, $X$ must contain the subgraph $H_2$. We have two situations:
\begin{enumerate}
    \item the graph $X$ contains a closed chain;
    \item the graph $X$ does not contain a closed chain but contains an open chain.
\end{enumerate}
Now we begin our discussion of each case.

\begin{case}
In this situation, the graph \(X\) contains a closed chain $C_k$. It can be decomposed as $X= C_k\cup A_1\cup\ldots\cup A_{\ell}$  for some integer $\ell$. Here, for $1\leq i,j\leq \ell$,
\begin{itemize}
    \item  each $A_i$ is a connected graph,
    \item $A_i\cap A_j=\varnothing$,
    \item $A_i\cap C_k=\{v_{d(i,1)},\ldots, v_{d(i,m_i)}\}$ which is a set of vertices of $C_k$, arranged in an increasing order.
\end{itemize}
For each $1\leq i\leq \ell$ and $2\leq j\leq m_i$, let $\gamma_{ij}$ be a path in $A_i$ connecting $v_{d(i,1)}$ to $v_{d(i,j)}$ and define $\delta_{d(i,j)}$ as a path in $C_k$ connecting $v_{0}$ to $v_{d(i,j)}$ as follows:
\[
   \delta_{d(i,j)} = 
    \begin{cases}
        s_{1,2}s_{2,2}\ldots s_{d(i,j),2} & \text{if }d(i,j) \neq k-1, \\
        s_{k,2}^{-1} & \text{if }d(i,j) = k-1.
    \end{cases}
\]
Notice that $\delta_{d(i,1)}$ is empty. Then,
\[\pi_1(X,v_0)=\pi_1(C_k,v_0)*\Asterisk_{i=1}^{\ell}\delta_{d(i,1)}\pi_1(A_i, v_{d(i,1)})\delta_{d(i,1)}^{-1}*\Asterisk_{i=1}^{\ell}\Asterisk_{j=2}^{m_i}\langle\delta_{d(i,1)}\gamma_{ij}\delta_{d(i,j)}^{-1}\rangle.\]
See Figure \ref{closed_chain_ex} for an illustration.
\begin{figure}[htbp]
    \centering
    \includegraphics[scale=0.4]{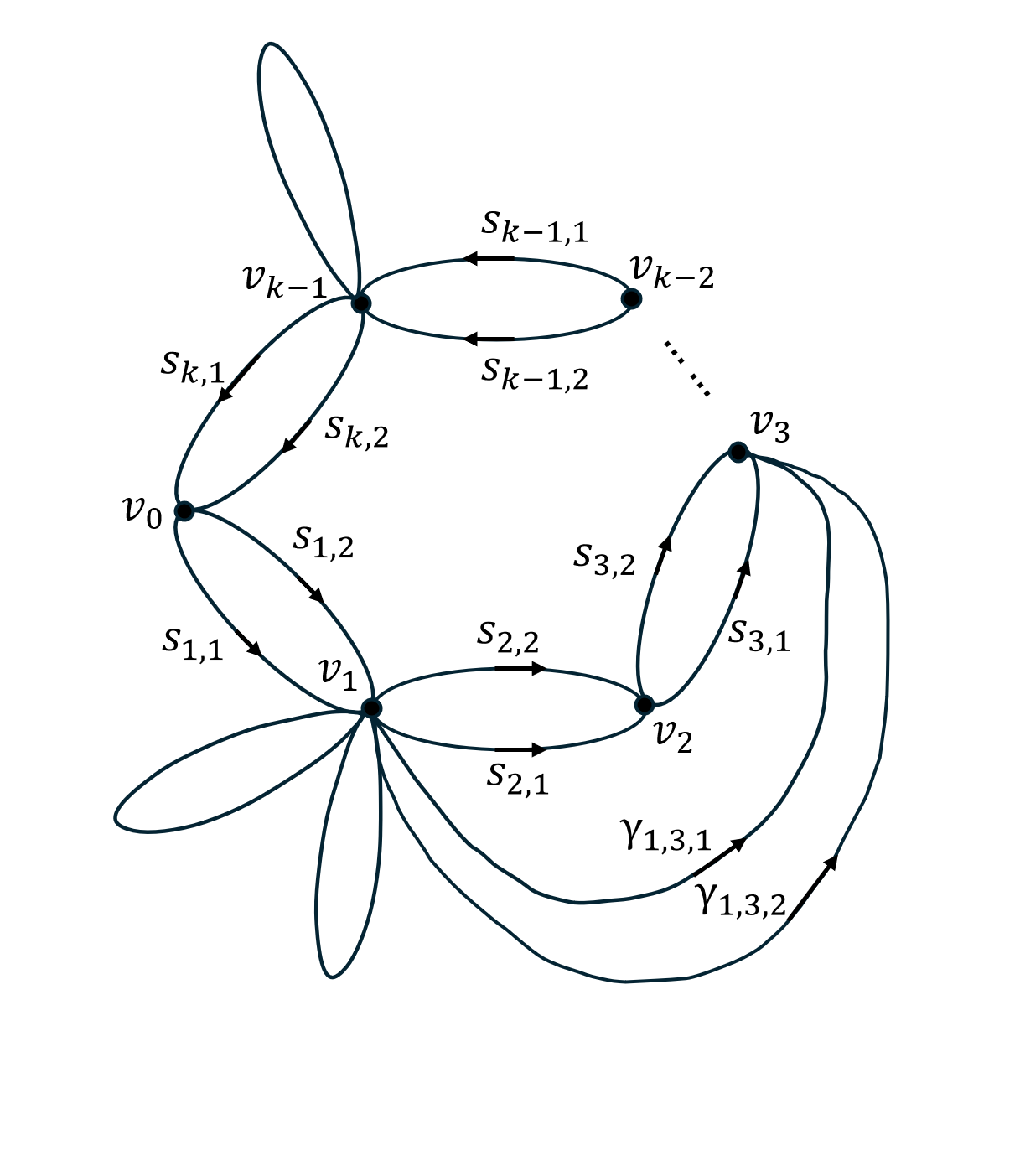}
    \caption{graph $X$ that contains a closed chain $C_k$\label{closed_chain_ex}}
\end{figure} 

Consider the following basis of $\pi_1(C_k, v_0)$:
\[x_1 = s_{1,1}s_{1,2}^{-1},\quad x_i= (s_{1,2}s_{2,2} \ldots s_{i-1,2})(s_{i,1}s_{i,2}^{-1})(s_{1,2}s_{2,2} \ldots s_{i-1,2})^{-1},\quad x_{k+1}=s_{1,2}s_{2,2} \ldots s_{k,2},\]
where $2\leq i\leq k$. The action of $f|_{\pi_1(C_k,v_0)}$ satisfies:
\begin{align*}
 f:&x_1 \mapsto x_1^{-1}\\
        &x_2 \mapsto x_1x_2^{-1}x_1^{-1}\\
        &x_3 \mapsto x_1x_2x_3^{-1}(x_1x_2)^{-1}\\
        &\vdots\quad \quad \quad \vdots\\
        &x_{k-1} \mapsto (x_1 \ldots x_{k-2})x_{k-1}^{-1}(x_1 \ldots x_{k-2})^{-1}\\
        &x_{k} \mapsto (x_1 \ldots x_{k-1})x_{k}^{-1}(x_1 \ldots x_{k-1})^{-1}\\
        &x_{k+1} \mapsto x_1x_2 \ldots x_kx_{k+1}.
\end{align*}
Consider $\Phi = fR_{x_{k-1}x_{k}}^{-1}fR_{x_{k-1}x_{k}}$. It acts on $\pi_1(C_k,v_0)$ as follows:
\begin{align*}
\Phi:  \text{ }& x_{k-1} \mapsto x_{k-1}x_{k} \mapsto(x_1 \ldots x_{k-2})x_{k}^{-1}(x_1 \ldots x_{k-1})^{-1}\mapsto (x_1 \ldots x_{k-2})x_{k-1}^{-1}(x_1 \ldots x_{k-2})^{-1}\mapsto x_{k-1},\\
       & x_{k} \mapsto x_{k} \mapsto (x_1 \ldots x_{k-1})x_k^{-1}(x_1 \ldots x_{k-1})^{-1}\mapsto (x_1 \ldots x_{k-1})x_{k}^{-1}(x_1 \ldots x_{k-1})^{-1} \mapsto x_k,\\
       & x_{k+1} \mapsto x_{k+1} \mapsto x_1x_2\ldots x_{k}x_{k+1} \mapsto x_1x_2\ldots x_{k-1}x_{k+1}\mapsto  x_{k}x_{k+1},
\end{align*}
and it fixes the other basis elements. Thus, $\Phi|_{\pi_1(C_k,v_0)}=L_{x_{k+1}x_k}$. Now, let us consider how $\Phi$ acts on $\delta_{d(i,1)}\pi_1(A_i, v_{d(i,1)})\delta_{d(i,1)}^{-1}$. Let $x_{\nu}$ be an element in $\delta_{d(i,1)}\pi_1(A_i, v_{d(i,1)})\delta_{d(i,1)}^{-1}$ where $\nu\in\pi_1(A_i, v_{d(i,1)})$. Then,
\[
   f(x_{\nu}) = 
    \begin{cases}
        (x_1 \ldots x_{d(i,1)})x_{\tilde{f}(\nu)}(x_1 \ldots x_{d(i,1)})^{-1}& \text{if }d(i,1) \neq k-1, \\
        (x_{k+1}^{-1}x_{k}^{-1}x_{k+1})x_{\tilde{f}(\nu)}(x_{k+1}^{-1}x_{k}^{-1}x_{k+1})^{-1} & \text{if }d(i,1) = k-1.
    \end{cases}
\]
The map $\Phi$ acts on $\delta_{d(i,1)}\pi_1(A_i, v_{d(i,j)})\delta_{d(i,1)}^{-1}$ as follows:
\begin{align*}
\text{when }&d(i,1) \neq k-1,\\
\Phi:  \text{ }& x_{\nu} \mapsto x_{\nu} \mapsto(x_1 \ldots x_{d(i,1)})x_{\tilde{f}(\nu)}(x_1 \ldots x_{d(i,1)})^{-1}\mapsto(x_1 \ldots x_{d(i,1)})x_{\tilde{f}(\nu)}(x_1 \ldots x_{d(i,1)})^{-1}\mapsto x_{\nu},\\
\text{when }&d(i,1) = k-1,\\
\Phi:  \text{ }& x_{\nu} \mapsto x_{\nu} \mapsto(x_{k+1}^{-1}x_{k}^{-1}x_{k+1})x_{\tilde{f}(\nu)}(x_{k+1}^{-1}x_{k}^{-1}x_{k+1})^{-1}\mapsto(x_{k+1}^{-1}x_{k}^{-1}x_{k+1})x_{\tilde{f}(\nu)}(x_{k+1}^{-1}x_{k}^{-1}x_{k+1})^{-1}\mapsto x_{\nu}.
\end{align*}
Therefore, the map $\Phi$ acts as identity on $\delta_{d(i,1)}\pi_1(A_i, v_{d(i,j)})\delta_{d(i,1)}^{-1}$. Let $x_{\gamma}=\delta_{d(i,1)}\gamma_{ij}\delta_{d(i,j)}^{-1}$. Because of the increasing order, $d(i,1)<d(i,j)$, then
\[
   f(x_{\gamma}) = 
    \begin{cases}
        (x_1 \ldots x_{d(i,1)})x_{\tilde{f}(\gamma)}(x_1 \ldots x_{d(i,j)})^{-1}& \text{if }d(i,j) \neq k-1, \\
        (x_1 \ldots x_{d(i,1)})x_{\tilde{f}(\gamma)}(x_{k+1}^{-1}x_{k}^{-1}x_{k+1})^{-1} & \text{if }d(i,j) = k-1.
    \end{cases}
\]
The map $\Phi$ acts on $\langle\delta_{d(i,1)}\gamma_{ij}\delta_{d(i,j)}^{-1}\rangle$ as follows:
\begin{align*}
\text{when }&d(i,j) \neq k-1,\\
\Phi:  \text{ }& x_{\gamma} \mapsto x_{\gamma} \mapsto(x_1 \ldots x_{d(i,1)})x_{\tilde{f}(\gamma)}(x_1 \ldots x_{d(i,j)})^{-1}\mapsto(x_1 \ldots x_{d(i,1)})x_{\tilde{f}(\gamma)}(x_1 \ldots x_{d(i,j)})^{-1}\mapsto x_{\gamma},\\
\text{when }&d(i,j) = k-1,\\
\Phi:  \text{ }& x_{\gamma} \mapsto x_{\gamma} \mapsto(x_1 \ldots x_{d(i,1)})x_{\tilde{f}(\gamma)}(x_{k+1}^{-1}x_{k}^{-1}x_{k+1})^{-1} \mapsto(x_1 \ldots x_{d(i,1)})x_{\tilde{f}(\gamma)}(x_{k+1}^{-1}x_{k}^{-1}x_{k+1})^{-1} \mapsto x_{\gamma}.
\end{align*}
Therefore, the map $\Phi$ acts as identity on $\langle\delta_{d(i,1)}\gamma_{ij}\delta_{d(i,j)}^{-1}\rangle$ as well. Hence, $\Phi =L_{x_{k+1}x_k}$ on $\pi_1(X,v_0)=F_n$. Since we have $L_{x_{k+1}x_k}\in N(f)$, by Lemma \ref{left multi}, 
\[
    N(f) = 
    \begin{cases}
        \mathrm{SAut}(F_n) & \text{if det}(\bar{f}) = 1, \\
        \mathrm{Aut}(F_n) & \text{if det}(\bar{f}) = -1.
    \end{cases}
\]
\end{case}
\begin{case}
In this situation, the graph \(X\) does not contain a closed chain but contains an open chain $O_k$. It can be decomposed as $X= O_k\cup A_1\cup\ldots\cup A_{\ell}$  for some integer $\ell$. Here, for $1\leq i,j\leq \ell$,
\begin{itemize}
    \item  each $A_i$ is a connected graph,
    \item $A_i\cap A_j=\varnothing$,
    \item $A_i\cap O_k=\{v_{d(i)}\}$ which is a single vertex of $O_k$.
\end{itemize}
\begin{figure}[htbp]
    \centering
    \includegraphics[scale=0.4]{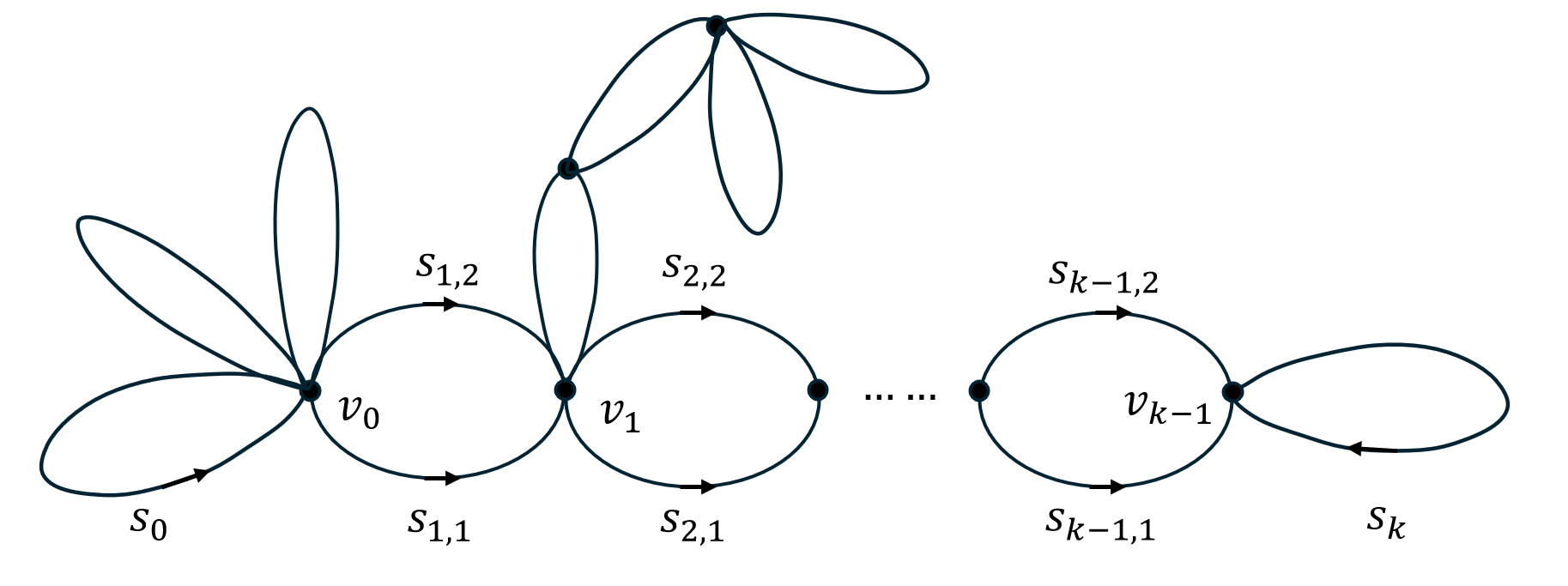}
    \caption{graph $X$ that contains an open chain $O_k$\label{open_chain_ex}}
\end{figure} 
For each $1\leq i\leq \ell$, define $\delta_{d(i)}$ as a path in $O_k$ connecting $v_{0}$ to $v_{d(i)}$ in the following way:
\[\delta_{d(i)} = s_{1,2}s_{2,2}\ldots s_{d(i),2} \]
Then,
\[\pi_1(X,v_0)=\pi_1(O_k,v_0)*\Asterisk_{i=1}^{\ell}\delta_{d(i)}\pi_1(A_i, v_{d(i)})\delta_{d(i)}^{-1}.\]
See Figure \ref{open_chain_ex} for an illustration. Consider the following basis of $\pi_1(O_k, v_0)$:
\begin{align*}
    &x_0 = s_0, \quad x_1 = s_{1,1}s_{1,2}^{-1},\quad x_i= (s_{1,2}s_{2,2} \ldots s_{i-1,2})(s_{i,1}s_{i,2}^{-1})(s_{1,2}s_{2,2} \ldots s_{i-1,2})^{-1}\\
     &x_{k}=(s_{1,2}s_{2,2} \ldots s_{k-1,2})(s_k)(s_{1,2}s_{2,2} \ldots s_{k-1,2})^{-1}
\end{align*}
where $2\leq i\leq k-1$. The action of $f$ on ${\pi_1(O_k,v_0)}$ has three situations, depending on how $f$ interacts with the terminal loops $x_0$ and $x_k$. Specifically, $f$ can either:
\begin{enumerate}
    \item fix both $x_0$ and $x_k$,
    \item invert at least one of them, or
    \item invert both $x_0$ and $x_k$.
\end{enumerate}
These cases correspond to the following maps:
\begin{align*}
f_1:&x_0\mapsto x_0 \\      
    &x_1 \mapsto x_1^{-1} \\       
    &x_2 \mapsto x_1x_2^{-1}x_1^{-1} \\  
    &\vdots \quad\quad\quad\vdots \\                          
    &x_k \mapsto (x_1 \ldots x_{k-1})x_k(x_{1} \ldots x_{k-1})^{-1}
\end{align*}
\begin{align*}
 f_2:&x_0\mapsto x_0^{-1} \\      
    &x_1 \mapsto x_1^{-1} \\       
    &x_2 \mapsto x_1x_2^{-1}x_1^{-1} \\  
    &\vdots \quad\quad\quad\vdots \\                            
    &x_k \mapsto (x_1 \ldots x_{k-1})x_k(x_{1} \ldots x_{k-1})^{-1}
\end{align*}
\begin{align*}
 f_3: &x_0\mapsto x_0^{-1} \\      
    &x_1 \mapsto x_1^{-1} \\       
    &x_2 \mapsto x_1x_2^{-1}x_1^{-1} \\  
    &\vdots \quad\quad\quad\vdots \\                            
    &x_k \mapsto (x_1 \ldots x_{k-1})x_k^{-1}(x_{1} \ldots x_{k-1})^{-1}
\end{align*}
\begin{subcase}
    For the map $f_1$, we are unable to find the conjugation map in $N(f_1)$ so that we fail to identify the group $N(f_1)$, see Question \ref{question}. However, we are able to handle the other two cases. 
\end{subcase}
\begin{subcase}\label{subcase_f2}
For the map $f_2$, let us consider the map $\Phi = L_{x_0x_1}f_2L_{x_0x_1}^{-1}f_2\in N(f_2)$. It acts on $\pi_1(O_k,v_0)$ as follows:
\begin{align*}
& x_0 \xmapsto{f_2} x_0^{-1} \xmapsto{L_{x_0x_1}^{-1}}x_0^{-1}x_1\xmapsto{f_2} x_0x_1^{-1}\xmapsto{L_{x_0x_1}} x_1x_0x_1^{-1}
\end{align*}
and it fixes the other basis elements. Thus, $\Phi|_{\pi_1(O_k,v_0)}=C_{x_0x_1}$. Now, let us consider how $\Phi$ acts on $\delta_{d(i)}\pi_1(A_i, v_{d(i)})\delta_{d(i)}^{-1}$. Let $x_{\nu}$ be an element in $\delta_{d(i)}\pi_1(A_i, v_{d(i)})\delta_{d(i)}^{-1}$ where $\nu\in\pi_1(A_i, v_{d(i)})$. Then,
\[
   f(x_{\nu}) = (x_1 \ldots x_{d(i)})x_{\tilde{f}(\nu)}(x_1 \ldots x_{d(i)})^{-1}
\]
The map $\Phi$ acts on $\delta_{d(i)}\pi_1(A_i, v_{d(i)})\delta_{d(i)}^{-1}$ as follows:
\begin{align*}
\Phi:  \text{ }& x_{\nu} \mapsto(x_1 \ldots x_{d(i)})x_{\tilde{f}(\nu)}(x_1 \ldots x_{d(i)})^{-1}\mapsto(x_1 \ldots x_{d(i)})x_{\tilde{f}(\nu)}(x_1 \ldots x_{d(i)})^{-1}\mapsto x_{\nu}\mapsto x_{\nu} .
\end{align*}
Therefore, the map $\Phi$ acts as identity on $\delta_{d(i)}\pi_1(A_i, v_{d(i)})\delta_{d(i)}^{-1}$. Hence, $\Phi=C_{x_0x_1}\in N(f_2)$ which implies $IA_n\subset N(f_2)$.

Then we consider the map $\Psi = L_{x_0x_k}f_2L_{x_0x_k}^{-1}f_2$. It acts on $\pi_1(O_k,v_0)$ as follows:
\begin{align*}
 x_0 &\xmapsto{f_2} x_0^{-1} \xmapsto{L_{x_0x_k}^{-1}}x_0^{-1}x_k\xmapsto{f_2} x_0(x_1 \ldots x_{k-1})x_k(x_{k-1} \ldots x_1^{-1})^{-1}\\
&\xmapsto{L_{x_0x_k}} x_kx_0(x_1 \ldots x_{k-1})x_k(x_{k-1} \ldots x_1^{-1})^{-1}
\end{align*}
and it fixes the other basis elements.  Moreover, the map $\Psi$ acts on $\delta_{d(i)}\pi_1(A_i, v_{d(i)})\delta_{d(i)}^{-1}$ as follows:
\begin{align*}
\Psi:  \text{ }& x_{\nu} \mapsto(x_1 \ldots x_{d(i)})x_{\tilde{f}(\nu)}(x_1 \ldots x_{d(i)})^{-1}\mapsto(x_1 \ldots x_{d(i)})x_{\tilde{f}(\nu)}(x_1 \ldots x_{d(i)})^{-1}\mapsto x_{\nu}\mapsto x_{\nu} .
\end{align*}
Therefore, the map $\Psi$ acts as identity on $\delta_{d(i)}\pi_1(A_i, v_{d(i)})\delta_{d(i)}^{-1}$. We conclude that the image of  $\Psi$ in $\mathrm{GL}_n(\mathbb{Z})$ is $e_{k+1,1}^2$ which normally generates $\Gamma_n(2)$. Hence, we have $\Gamma_n(2)\subset\overline{N(f_2)}$. We can then apply Lemma \ref{level 2}  to determine the normal closure $N(f_2)$.
\end{subcase}
\begin{subcase}
For the map $f_3$, let us consider the map $\Phi = L_{x_0x_1}f_3L_{x_0x_1}^{-1}f_3\in N(f_3)$. It acts on $\pi_1(O_k,v_0)$ as follows:
\begin{align*}
& x_0 \xmapsto{f_3} x_0^{-1} \xmapsto{L_{x_0x_1}^{-1}}x_0^{-1}x_1\xmapsto{f_3} x_0x_1^{-1}\xmapsto{L_{x_0x_1}} x_1x_0x_1^{-1}
\end{align*}
and it fixes the other basis elements. Thus, $\Phi|_{\pi_1(O_k,v_0)}=C_{x_0x_1}$. Now, let us consider how $\Phi$ acts on $\delta_{d(i)}\pi_1(A_i, v_{d(i)})\delta_{d(i)}^{-1}$. Let $x_{\nu}$ be an element in $\delta_{d(i)}\pi_1(A_i, v_{d(i)})\delta_{d(i)}^{-1}$ where $\nu\in\pi_1(A_i, v_{d(i)})$. Then,
\[
   f(x_{\nu}) = (x_1 \ldots x_{d(i)})x_{\tilde{f}(\nu)}(x_1 \ldots x_{d(i)})^{-1}
\]
The map $\Phi$ acts on $\delta_{d(i)}\pi_1(A_i, v_{d(i)})\delta_{d(i)}^{-1}$ as follows:
\begin{align*}
\Phi:  \text{ }& x_{\nu} \mapsto(x_1 \ldots x_{d(i)})x_{\tilde{f}(\nu)}(x_1 \ldots x_{d(i)})^{-1}\mapsto(x_1 \ldots x_{d(i)})x_{\tilde{f}(\nu)}(x_1 \ldots x_{d(i)})^{-1}\mapsto x_{\nu}\mapsto x_{\nu} .
\end{align*}
Therefore, the map $\Phi$ acts as identity on $\delta_{d(i)}\pi_1(A_i, v_{d(i)})\delta_{d(i)}^{-1}$. Hence, $\Phi=C_{x_0x_1}\in N(f_3)$ which implies $IA_n\subset N(f_3)$.

Given $\bar{f_3}|_{\pi_1(O_k,v_0)}=-I$ in $\mathrm{GL}_{k+1}(\mathbb{Z})$, the normal closure $N(f_3)$ relies on how $f_3$ acts on the rest part of the graph $X$. 
\setcounter{casea}{0}
\begin{casea}
If $\bar{f_3}=-I$ in $\mathrm{GL}_{n}(\mathbb{Z})$, then  we have the short exact sequence:
\[
1 \rightarrow IA_n \rightarrow N(f_3)\rightarrow \{\pm I\} \rightarrow 1.
\]
\end{casea}
\begin{casea}
If there exists a loop that is fixed by $f_3$ and no loops are permuted, then by replacing $f_3$ by a conjugate element, we can assume that the new basepoint $v_0'$ lies at the fixed loop. After reordering the basis of $\pi(X,v_0')$, we can label this fixed loop as $x_0$. This leads us to the Case \ref{subcase_f2}. 
\end{casea}
\begin{casea}
If there exists two loops that are permuted by $f_3$,  then by replacing $f_3$ by a conjugate element, we can assume that the new basepoint $v_0'$ lies at the permuted loops. After reordering the basis of $\pi(X,v_0')$, we label the two permuted loops as $x_0$ and $x_1$. The map $f_3$ satisfies:
\begin{align*}
 f_3: &x_0\mapsto x_1 \\      
    &x_1 \mapsto x_0 \\       
    &x_i \mapsto f(x_i)
\end{align*}
where $1\leq i\leq n-2$. Notice that $f_3$ cannot wrap the basis around the loop $x_0$ and $x_1$. Therefore, $f(x_i)$ does not contain $x_0$ and $x_1$. Consider $\Psi=f_3I_{x_0}f_3I_{x_0}\in N(f_3)$. Then
\begin{align*}
\Psi:  \text{ }& x_{0} \mapsto x_0^{-1}\mapsto x_1^{-1}\mapsto x_1^{-1}\mapsto x_0^{-1},\\ 
& x_{1} \mapsto x_1\mapsto x_0\mapsto x_0^{-1}\mapsto x_1^{-1},\\
& x_{i} \mapsto x_i\mapsto f(x_i)\mapsto f(x_i)\mapsto x_i.
\end{align*}
 Consider $L_{x_0x_2}\Psi L_{x_0x_2}^{-1}\Psi$. It  acts on  the basis elements as follows:
\begin{align*}
    &x_{0} \xmapsto{\Psi} x_0^{-1} \xmapsto{L_{x_{0}x_{2}}^{-1}} x_0^{-1}x_2 \xmapsto{ \Psi } x_0x_2 \xmapsto{L_{x_0x_2}} x_2x_0x_2,
\end{align*}
and it fixes the other generators. We conclude that the image of $L_{x_0x_2}\Psi L_{x_0x_2}^{-1}\Psi$ in $\mathrm{GL}_n(\mathbb{Z})$ is $e_{3,1}^2$ which normally generates $\Gamma_n(2)$. Therefore,$\Gamma_n(2)\subset \overline{N( f_3)}$. Since $\bar{f_3}\notin\Gamma_n(2)$, by Lemma \ref{level 2}, 
\[
    N(f) = 
    \begin{cases}
        \mathrm{SAut}(F_n) & \text{if det}(\bar{f}) = 1, \\
        \mathrm{Aut}(F_n) & \text{if det}(\bar{f}) = -1.
    \end{cases}
\]
\end{casea}
\end{subcase}
\end{case}

Now, we finished the discussion of order $2$ elements in $\mathrm{Aut}(F_n)$. Let us continue examining the case in $\mathrm{Out}(F_n)$.  Let $f\in\mathrm{Out}(F_n)$ be of order $2$. Assume $f$ cannot be lifted to a finite order element of $\mathrm{Aut}(F_n)$. By Proposition \ref{Out}, there exists an $m\geq 1$ with $2(m-1)+1=n$ such that we can realize $f$ by the rotation automorphism $\tilde{f} : R_{2,m}\rightarrow R_{2,m}$. It satisfies:
\begin{align*}
\tilde{f}:&s_1\mapsto s_2 \\      
    &s_2 \mapsto s_1 \\       
    &\ell_{0,j} \mapsto \ell_{1,j} \\
    &\ell_{1,j} \mapsto \ell_{0,j}
\end{align*}
where $1\leq j\leq m-1$. See Figure \ref{R_2,3} for an illustration.

\begin{figure}[htbp]
\centering
\includegraphics[scale=0.4]{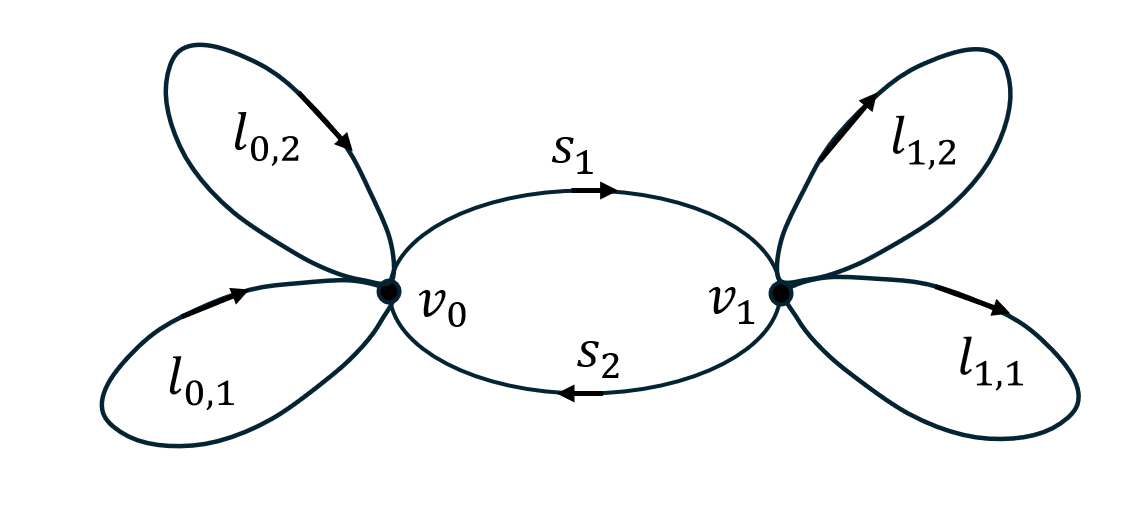}
\caption{graph $R_{2,3}$\label{R_2,3}}
\end{figure}
Consider the following basis of $\pi_1(R_{2,m}, v_0)$:
\begin{align*}
x_{0,j}= \ell_{0,j}, \quad x_{1,j} = s_1\ell_{1,j} s_1^{-1}, \quad c = s_1s_2 
\end{align*}
where $1\leq j\leq m-1$. The induced automorphism $\phi_{s_1,\tilde{f}}\in\mathrm{Aut}(\pi_1(R_{2,m}, v_0))\cong\mathrm{Aut}(F_n)$  satisfies:
\begin{align*}
\phi_{s_1,\tilde{f}}:      &x_{0,j} \mapsto x_{1,j}   &\phi_{s_1,\tilde{f}}^{-1}:     &x_{0,j} \mapsto c^{-1} x_{1,j} c \\
        &x_{1,j} \mapsto c x_{0,j} c^{-1}    &    &x_{1,j} \mapsto x_{0,j} \\  
        &c \mapsto c                 &            &c \mapsto c
\end{align*}
where $1\leq j\leq m-1$.  Let $\Phi=\phi_{s_1,\tilde{f}}I_{x_{0,1}}\phi_{s_1,\tilde{f}}^{-1}I_{x_{0,1}}$, then
\begin{align*}
\Phi:& x_{0,1} \xmapsto{I_{x_{0,1}}} x_{0,1}^{-1} \xmapsto{\phi_{s_1,\tilde{f}}^{-1}}c^{-1} x_{1,1}^{-1} c\xmapsto{I_{x_{0,1}}}c^{-1} x_{1,1}^{-1}\xmapsto{\phi_{s_1,\tilde{f}}} x_{0,1}^{-1},\\
& x_{1,1} \xmapsto{I_{x_{0,1}}} x_{1,1} \xmapsto{\phi_{s_1,\tilde{f}}^{-1}}x_{0,1} c\xmapsto{I_{x_{0,1}}}x_{0,1}^{-1}\xmapsto{\phi_{s_1,\tilde{f}}} x_{1,1}^{-1},
\end{align*}
and it fixes the other basis elements in $\pi_1(R_{2,m}, v_0)$. Now we consider $L_{x_{0,1}x_{1,1}}\Phi L_{x_{0,1}x_{1,1}}^{-1}\Phi$. It acts on $\pi_1(R_{2,m}, v_0)$ as follows:
\begin{align*}
& x_{0,1} \xmapsto{\Phi} x_{0,1}^{-1} \xmapsto{L_{x_{0,1}x_{1,1}}^{-1}}x_{0,1}^{-1}x_{1,1}\xmapsto{\Phi} x_{0,1}x_{1,1}^{-1}\xmapsto{L_{x_{0,1}x_{1,1}}} x_{1,1}x_{0,1}x_{1,1}^{-1}
\end{align*}
and it fixes the other basis elements. Thus, the map $L_{x_{0,1}x_{1,1}}\Phi L_{x_{0,1}x_{1,1}}^{-1}\Phi= C_{x_{0,1}x_{1,1}}\in N(\phi_{s_1,\tilde{f}})$. It implies that $IA_n\subset N(\phi_{s_1,\tilde{f}})$. Now, we consider $L_{x_{0,1}c}\Phi L_{x_{0,1}c}^{-1}\Phi$. It acts on $\pi_1(R_{2,m}, v_0)$ as follows:
\begin{align*}
& x_{0,1} \xmapsto{\Phi} x_{0,1}^{-1} \xmapsto{L_{x_{0,1}c}^{-1}}x_{0,1}^{-1}c\xmapsto{\Phi} x_{0,1}c\xmapsto{L_{x_{0,1}c}} cx_{0,1}c
\end{align*}
and it fixes the other basis elements. Then the image of $L_{x_{0,1}c}\Phi L_{x_{0,1}c}^{-1}\Phi$ in $GL_n(\mathbb{Z})$ is $e_{1,n}^2$ which normally generates $\Gamma_n(2)$. Therefore, $\Gamma_n(2)\subset\overline{N(\phi_{s_1,\tilde{f}})}$ . Since $\bar{\phi}_{s_1,\tilde{f}}\notin \Gamma_n(2)$, by Lemma \ref{level 2}, we conclude that
\[
    N(\phi_{s_1,\tilde{f}}) = 
    \begin{cases}
        \mathrm{SAut}(F_n) & \text{if det}(\bar{\phi}_{s_1,\tilde{f}}) = 1, \\
        \mathrm{Aut}(F_n) & \text{if det}(\bar{\phi}_{s_1,\tilde{f}}) = -1.
    \end{cases}
\]
Projecting $\phi_{s_1,\tilde{f}}$ to $\mathrm{Out}(\pi_1(R_{2,m}, v_0))$, since $\langle f\rangle\cong\langle [\phi_{s_1,\tilde{f}}]\rangle$, we conclude that
\[
    N(f) = 
    \begin{cases}
        \mathrm{SOut}(F_n) & \text{if det}(\bar{f}) = 1, \\
        \mathrm{Out}(F_n) & \text{if det}(\bar{f}) = -1.
    \end{cases}
\]

\end{document}